\documentclass[11pt]{amsart}

\usepackage[a4paper,bindingoffset=0.2in,%
            left=1in,right=1in,top=1in,bottom=1in,%
            footskip=.25in]{geometry}

\usepackage{amsthm}

\usepackage{amsmath}
\usepackage{amssymb}
\usepackage{mathrsfs}
\usepackage{verbatim}
\usepackage{enumitem}
\usepackage{svn}
\usepackage{stmaryrd}
\usepackage[all]{xy}
\usepackage{amsfonts}
\usepackage{comment}
\usepackage{tabu} %for math mode of tables
 \usepackage{geometry}
\usepackage[table]{xcolor}
\usepackage{adjustbox}
\usepackage{color}
\definecolor{ghcolor}{RGB}{0, 150, 200} %defines colors if you need in hyperref
\definecolor{winestain}{rgb}{0.5,0,0}
\usepackage[linktocpage=true, hidelinks, colorlinks, linkcolor=blue, citecolor=red, bookmarksopen=true]{hyperref} %linktocpage makes pages in toc clickable

\newtheorem{theorem}[subsubsection]{Theorem}
\newtheorem{thm}[subsubsection]{Theorem}
\newtheorem{lemma}[subsubsection]{Lemma}
\newtheorem{lem}[subsubsection]{Lemma}
\newtheorem{cor}[subsubsection]{Corollary}
\newtheorem{corollary}[subsubsection]{Corollary}

\newtheorem{prop}[subsubsection]{Proposition}

\theoremstyle{definition}
\newtheorem{defn}[subsubsection]{Definition}

\newtheorem{definition}[subsubsection]{Definition}
\newtheorem{example}[subsubsection]{Example}

\newtheorem{Notation}[subsubsection]{Notation}
\newtheorem{notation}[subsubsection]{Notation}

\newtheorem{Convention}[subsubsection]{Convention}

\theoremstyle{remark}
\newtheorem{remark}[subsubsection]{Remark}
\newtheorem{rem}[subsubsection]{Remark}

\numberwithin{equation}{subsection}

\def\upi{\underline \pi}

\def\Mat{\text{\rm Mat}}

\def\val{\text{\rm val}}

\def\O{\mathcal{O}}
\def\E{\mathcal{E}}

 \def \E{\mathcal E}

\def \Z {\mathbb Z}
\def \inj {\hookrightarrow }
\def \to {\rightarrow}

\def \spec \text{spec}
\def \cont \text{cont}
 \def \M{\mathfrak M}

\def \GL {\mathrm{GL}}

\def \an {\mathrm{an} }
\def \la {\mathrm{la}}
\def \Tr {\mathrm{Tr}}

\DeclareMathOperator{\gal}{Gal}

\DeclareMathOperator{\dR}{dR}

\def \Z {\mathbb Z}

\def \O {\mathcal O}

\def \gs {\mathfrak S}

\def \ur {\mathrm{ur}}

\def \dR {{\textnormal{dR}}}

\def \Ker {\textnormal{Ker}}

\def \Fr {\textnormal{Fr}}

\def \upi {\underline \pi }

\def \< {\left <}
\def \> {\right >}

\def \upi {{\underline{\pi}}}

\def \Md {{\rm M}_d}

\def \rig {{\rm{rig}}}

\def \Zp { \mathbb Z_p}
\def \Qp { \mathbb Q_p}

\DeclareMathOperator{\Gal}{Gal}

\newcommand*{\wt}[1]{\widetilde{#1}}

\renewcommand{\phi}{\varphi}

\renewcommand{\projlim}{\varprojlim}

\newcommand{\cbf}{\mathbf{c}}

\newcommand{\kbf}{\mathbf{k}}

\newcommand{\pa}{\mathrm{pa}}
\newcommand{\dan}{\text{$\mbox{-}\mathrm{an}$}}
\newcommand{\dla}{\text{$\mbox{-}\mathrm{la}$}}
\newcommand{\dpa}{\text{$\mbox{-}\mathrm{pa}$}}
\newcommand{\dacc}[1]{\{\!\{ #1 \}\!\}}
\newcommand{\smat}[1]{\left( \begin{smallmatrix} #1 \end{smallmatrix} \right)}
\def \Kpinfty {K_{p^\infty}}
\begin{document}

\title{Locally analytic vectors and overconvergent $(\varphi, \tau)$-modules }
\date{\today}
%    Information for first author
\author{Hui Gao}
%    Address of record for the research reported here
\address{Department of Mathematics and Statistics, University of Helsinki, FI-00014, Finland}
\email{hui.gao@helsinki.fi}

\author{L\'{e}o Poyeton}
\address{UMPA, \'{E}cole Normale Sup\'erieure de Lyon, 46 all\'ee d'Italie, 69007 Lyon, France}
\email{leo.poyeton@ens-lyon.fr}

%\thanks{}
%    \thanks will become a 1st page footnote.
%\thanks{The first author was supported in part by NSF Grant \#000000.}

%    General info


\begin{abstract} Let $p$ be a prime, let $K$ be a complete discrete valuation field of characteristic $0$ with a perfect residue field of characteristic $p$, and let $G_K$ be the Galois group. Let $\pi$ be a fixed uniformizer of $K$, let $K_\infty$ be the extension by adjoining to $K$ a system of compatible $p^n$-th roots of $\pi$ for all $n$, and let $L$ be the Galois closure of $K_\infty$. Using these field extensions, Caruso constructs the $(\varphi, \tau)$-modules, which classify $p$-adic Galois representations of $G_K$. In this paper, we study locally analytic vectors in some period rings with respect to the $p$-adic Lie group $\gal(L/K)$, in the spirit of the work by Berger and Colmez. Using these locally analytic vectors, and using the classical overconvergent $(\varphi, \Gamma)$-modules, we can establish the overconvergence property of the  $(\varphi, \tau)$-modules.

%In an upcoming work by one of us, the ideas and results of this paper will be generalized for an arithmetic family of Galois representations, which in turn will be used to prove a conjecture of Bellovin on sheaves of Fontaine periods.
%In a previous joint work by one of us and Tong Liu, the overconvergence property of  $(\varphi, \tau)$-modules is established when $K$ is a finite extension of $\Qp$, via a completely different method that does not work for general $K$ or for arithmetic families.
\end{abstract}

\maketitle

\tableofcontents

\section{Introduction}
\subsection{Overview and main theorem}

Let  $p$ be a prime, and let $K$ be a complete discrete valuation field of characteristic $0$ with a perfect residue field of characteristic $p$. We fix an algebraic closure $\overline {K}$ of $K$ and set $G_K:=\Gal(\overline{K}/K)$. In $p$-adic Hodge theory, we use various ``linear algebra" tools to study $p$-adic representations of $G_K$. A key idea in $p$-adic Hodge theory is to first restrict the Galois representations to some subgroups of $G_K$. For example, the classical $(\varphi, \Gamma)$-modules are constructed by using the subgroup $G_{p^\infty}:= \gal (\overline K / K_{p^\infty})$ where $K_{p^\infty}$ is the extension of $K$ by adjoining a compatible system of $p^n$-th  primitive roots of $1$ for all $n$ (cf. Notation \ref{nota fields} below). Later, it becomes clear that it is also important to study other possible theories arising from other subgroups. In this paper, we will study the $(\varphi, \tau)$-modules, which are constructed by using the subgroup $G_{\infty}:= \gal (\overline K / K_{\infty})$ where $K_{\infty}$ is the extension of $K$ by adjoining a compatible system of $p^n$-th roots of a fixed uniformizer of $K$ for all $n$ (cf. Notation \ref{nota fields} below).

The $(\varphi, \tau)$-modules, firstly constructed by Caruso (cf. \cite{Car13}), originated from works by Breuil and Kisin (cf. e.g., \cite{Bre99b, Kis06}); they look quite similar to the $(\varphi, \Gamma)$-modules, but in certain situations (in particular, if we consider the semi-stable representations), give much more useful information than the later. For example, these semi-stable $(\varphi, \tau)$-modules (called Kisin modules, or Breuil-Kisin modules, or $(\varphi, \hat{G})$-modules in various contexts) can be used to study Galois deformation rings (cf. \cite{Kis08}),
to classify semi-stable (integral) Galois representations (cf. \cite{Liu10}),
and to study integral models of Shimura varieties (cf. \cite{Kis10}), to name just a few. In contrast, the $(\varphi, \Gamma)$-modules can only achieve very partial results in the aforementioned situations. However, the $(\varphi, \Gamma)$-modules have their own advantages; for example, they can be used to interprete Iwasawa cohomology (cf. \cite{CC99}), to prove $p$-adic monodromy theorem (cf. \cite{Ber02}), and most fantastically, to construct $p$-adic Langlands correspondence in the $\GL_2(\Qp)$-situation (cf. \cite{Col10}). To explore other possible applications of the $(\varphi, \tau)$-modules (and also the $(\varphi, \Gamma)$-modules), it is desirable to establish more parallel properties and build more links between these two theories. In this paper, we will study the overconvergence property of the $(\varphi, \tau)$-modules; the analogous property of the $(\varphi, \Gamma)$-modules, first established by Cherbonnier and Colmez (cf. \cite{CC98}), played a fundamental role  in almost all applications of the $(\varphi, \Gamma)$-modules.

Let us be more precise now.

\begin{Notation} \label{nota fields}
Let $k$ be the (perfect) residue field of $K$, let $W(k)$ be the ring of Witt vectors, and let $K_0 :=W(k)[1/p]$.
Thus $K$ is a totally ramified finite extension of $K_0$; write $e := [K: K_0]$.
Let $C_p$ be the $p$-adic completion of $\overline{K}$. Let $v_p$ be the valuation on $C_p$ such that $v_p(p)=1$.
For any subfield $Y \subset C_p$, let $\mathcal{O}_Y$ be its ring of integers.

Let $\pi \in   K$ be a uniformizer, and let $E(u) \in W(k)[u]$ be the irreducible polynomial of $\pi$ over $K_0$. Define a sequence of elements $\pi _n \in \overline K$ inductively such that $\pi_0 = \pi$ and $(\pi_{n+1})^p = \pi_n$. Define $\mu _n \in \overline K$ inductively such that $\mu_1$ is a primitive $p$-th root of unity and $(\mu_{n+1})^p = \mu_n$.
Let
$$K_{\infty} : = \cup _{n = 1} ^{\infty} K(\pi_n), \quad K_{p^\infty}=  \cup _{n=1}^\infty
K(\mu_{n}), \quad L:=  \cup_{n = 1} ^{\infty} K(\pi_n, \mu_n).$$
Let $$G_{\infty}:= \gal (\overline K / K_{\infty}), \quad G_{p^\infty}:= \gal (\overline K / K_{p^\infty}), \quad G_L: =\gal(\overline K/L), \quad \hat G: =\gal (L/K) .$$
\end{Notation}

Let $V$ be a finite dimensional $\Qp$-vector space equipped with a continuous $\Qp$-linear $G_K$-action.
In \cite{Car13}, using the theory of field of norms for the field $K_\infty$, Caruso associates to $V$  an \'etale $(\varphi, \tau)$-module (if one uses the field $K_{p^\infty}$ instead, one would get the usual \'etale $(\varphi, \Gamma)$-module); this induces an equivalence between the category of $p$-adic representations of $G_K$ and the category of \'etale $(\varphi, \tau)$-modules.
An \'etale $(\varphi, \tau)$-module is a triple $\hat D = (D, \varphi_D, \hat{G})$ (see Def. \ref{defn phi tau mod} for more details).
Here, we only mention that $D$ is a finite dimensional vector space over the field $\mathbf{B}_{K_\infty}:=\mathbf{A}_{K_\infty}[1/p]$ where
$$\mathbf{A}_{K_\infty}: =\{ \sum_{i=-\infty}^{+\infty} a_i u^i : a_i \in W(k), v_p(a_i) \to +\infty, \text{ as } i \to -\infty \},$$
and $\varphi_D$ is a certain map $D \to D$ (here, we ignore the discussion of the $\hat{G}$-data).
 We say that $\hat D$ is \emph{overconvergent} if
  we can ``descend" the module $D$ to a $\varphi$-stable submodule $D^\dagger$ over a subring $\mathbf{B}_{K_\infty}^\dagger$ (called the overconvergent subring) of $\mathbf{B}_{K_\infty}$, where
$$\mathbf{B}_{K_\infty}^\dagger: =\{ \sum_{i=-\infty}^{+\infty} a_i u^i \in \mathbf{B}_{K_\infty}, v_p(a_i) +i\alpha \to +\infty \text{ for some } \alpha >0, \text{ as } i \to -\infty \}.$$
The following is our main theorem.

\begin{thm} \label{thm intro main}
For any finite dimensional $\Qp$-representation $V$ of $G_K$, its associated $(\varphi, \tau)$-module is overconvergent.
\end{thm}
\begin{rem}
\begin{enumerate}
\item Thm. \ref{thm intro main} is originally proposed as a question by Caruso in \cite[\S 4]{Car13}, as an analogue of the classical overconvergence theorem for \'etale $(\varphi, \Gamma)$-modules by Cherbonnier and Colmez (\cite{CC98}).
  \item In a previous joint work by the first named author and T. Liu, Thm \ref{thm intro main} is established when $K$ is a finite extension of $\Qp$, using a completely different method (see \cite{GL}); a key ingredient in \emph{loc. cit.} is the construction of  ``loose crystalline lifts" of torsion Galois representations, which requires the finiteness of $k$ (see e.g., \cite[Rem. 1.1.2]{GL}).

  \item There does not seem to be any obvious comparison between the proof in this paper and that in \cite{GL}. The main idea in \cite{GL} is to ``approximate" a general $p$-adic Galois representation by torsion crystalline representations; whereas we do not use any torsion representations in the current paper.
\end{enumerate}
\end{rem}

\begin{rem}
\begin{enumerate}

\item In an upcoming work  by the first named author, the overconvergence property will also be established for $(\varphi, \tau)$-modules attached to an arithmetic family of Galois representations $V_S$ over a rigid analytic space $S$ (we need to assume $K/\Qp$ finite there).
Furthermore, we will use these family of overconvergent $(\varphi, \tau)$-modules to study sheaves of Fontaine periods (e.g., as in  \cite{Bel15}).

\item Using ideas and methods in this paper, it also seems very plausible to formulate and prove overconvergence results for \emph{geometric} families of  $(\varphi, \tau)$-modules, in analogy with results in \cite{KL2}.

\item In contrast, the methods in \cite{GL} can not be generalized to families (either arithmetic or geometric) of Galois representations.
\end{enumerate}
\end{rem}

\begin{rem}
  We refer to  \cite[\S 1.2]{GL} for some discussions of the importance and usefulness of overconvergence results in $p$-adic Hodge theory. In particular, in \emph{loc. cit.}, we mentioned about the \emph{link} between the category of all Galois representations and the category of geometric (i.e., semi-stable, crystalline) representations. Indeed, in \emph{loc. cit.}, we used this link to prove the overconvergence theorem. In the current paper, we do not use any semi-stable representations; instead, some results we obtain in the current paper will be used to study semi-stable representations. One result worth mentioning is Thm. \ref{thm loc ana gamma 1}(4) (see also Rem. \ref{rem kisin ring}), where we show certain ring of locally analytic vectors is related with the ring $\mathcal{O}_{[0, 1)}$ in \cite{Kis06}. We will report some progress (in particular, on the theory of $(\varphi, \hat{G})$-modules) in a future work by the first named author and T. Liu.
     \end{rem}
      %The key tool in \cite{GL} is the Kisin modules (cf. \cite{Kis06}). In analogy with the theory of Wach modules (cf. \cite{Col99} and \cite[\S 3.3]{Ber02}), it is very possible we can reprove the existence of Kisin modules (for semi-stable Galois representations) by using our overconvergence result. It is also interesting to investigate if we can apply our results to the theory of $(\varphi, \hat{G})$-modules (cf. \cite{Liu10}).

\subsection{Strategy of proof}
The key ingredient for the proof of Thm. \ref{thm intro main} is the calculation of locally analytic vectors in some period rings, in the spirit of the work by Berger and Colmez (\cite{BC16, Ber16}).
The philosophy that overconvergence of Galois representations is related with locally analytic vectors is first observed by Colmez, in the framework of $p$-adic Langlands correspondence (cf. \cite[Intro. 13.3]{Col10}). For example, overconvergent $(\varphi, \Gamma)$-modules (cf. \cite{CC98}) are closely related with locally analytic vectors in the $p$-adic Langlands correspondence for $\GL_2(\Qp)$ (cf. \cite{LXZ12, Col14}), i.e., via the ``\emph{locally analytic $p$-adic Langlands correspondence}".

To study the $p$-adic Langlands correspondence for $\GL_2(F)$ where $F/\Qp$ is a finite extension, Berger recently proves overconvergence of the Lubin-Tate $(\varphi, \Gamma)$-modules (cf. \cite{Ber16}).
The key idea in \emph{loc. cit.}, very roughly speaking, is that there should exist ``enough" locally analytic vectors in the Lubin-Tate $(\varphi, \Gamma)$-modules.
To find these locally analytic vectors, one first ``enlarges" the space of Lubin-Tate $(\varphi, \Gamma)$-modules over a bigger period ring; then there are indeed enough locally analytic vectors, by \emph{using the classical overconvergent $(\varphi, \Gamma)$-modules as an input} (cf. \cite[Thm. 9.1]{Ber16}). One then descends from the bigger space of locally analytic vectors to the level of  Lubin-Tate $(\varphi, \Gamma)$-modules, via a monodromy theorem (cf. \cite[\S 6]{Ber16}).

%In some sense, the Lubin-Tate $(\varphi, \Gamma)$-modules and the classical $(\varphi, \Gamma)$-modules become isomorphic over this big period ring.

The key idea in our paper is similar to that in \cite{Ber16}. Indeed, (very roughly speaking), we first ``enlarge" the space of the $(\varphi, \tau)$-module  over the big period ring $\wt{\mathbf{B}}^{\dagger}_{\rig, L}$ (which is $\gal(\overline{K}/L)$-invariant of the well-known ring  $\wt{\mathbf{B}}^{\dagger}_{\rig}$); there are enough locally analytic vectors on this level, by \emph{using the classical overconvergent $(\varphi, \Gamma)$-modules as an input} again (cf. the proof of Thm. \ref{thm final}).
To descend these locally analytic vectors to the level of $(\varphi, \tau)$-modules, we can use a Tate-Sen descent or a monodromy descent (see Prop. \ref{prop M triv} and Rem. \ref{rem mono des} for more details).
%We then use a Tate-Sen descent (followed by an \'etale descent) to descend these locally analytic vectors to the level of $(\varphi, \tau)$-modules.

% show that the $(\varphi, \tau)$-module (associated to $V$) becomes isomorphic to the classical $(\varphi, \Gamma)$-module (associated to $V$) over the big period ring $\wt{\mathbf{B}}^I_L$ (cf. the proof of Thm. \ref{thm final}); thus we have enough locally analytic vectors on this level (cf. Eqn. \eqref{phitaupa}).

As the strategy suggests, one needs to compute locally analytic vectors in some period rings (e.g., $\wt{\mathbf{B}}^{\dagger}_{\rig, L}$). In the case of $(\varphi, \Gamma)$-modules, the concerned $p$-adic Lie group is  $\gal(K_{p^\infty}/K)$ (see Notation \ref{nota fields}), which is one-dimensional. In the case of Lubin-Tate $(\varphi, \Gamma)$-modules, the $p$-adic Lie group is $\mathcal{O}_F^\times$, which is of dimension $[F: \Qp]$. In general, it would be very difficult to calculate locally analytic vectors for $p$-adic Lie groups of dimension higher than one. In \cite{Ber16}, Berger considers firstly the ``$F$-analytic" locally analytic vectors, which behave similar to the one-dimensional case. He then uses these ``$F$-analytic" locally analytic vectors to determine the full space of $\mathcal{O}_F^\times$-locally analytic vectors.
In our paper, the concerned $p$-adic Lie group is $\hat{G}=\gal(L/K)$, which is of dimension two. The key observation is that we need to firstly consider $\hat{G}$-locally analytic vectors which are \emph{furthermore $\gal(L/K_\infty)$-invariant}; these locally analytic vectors then again behave similar to the one-dimensional case. Indeed, we have:

\begin{theorem}
Let $(\wt{\mathbf{B}}^{\dagger}_{\rig, L})^{\tau\dpa, \gamma=1}$ denote the set of $\gal(L/K_{p^\infty})$-(pro)-locally analytic vectors which are furthermore fixed by $\gal(L/K_{\infty})$. Then we have
$$  (\wt{\mathbf{B}}^{\dagger}_{\rig, L})^{\tau\dpa, \gamma=1} =\cup_{m \geq 0}\varphi^{-m} ( \mathbf{B}_{\rig, K_\infty}^{\dagger}  ) ,$$
where $\mathbf{B}_{\rig, K_\infty}^{\dagger}$ is the ``Robba ring with coefficients in $K_0$" (cf. Def. \ref{defn rig ring}).
\end{theorem}

With the above theorem established, we can also completely determine the $\hat{G}$-locally analytic vectors in $\wt{\mathbf{B}}^{\dagger}_{\rig, L}$; since the statement is too technical, we refer the reader to Thm. \ref{thm Ghat la}.

%Indeed, we will see that these locally analytic vectors are more or less just the \emph{Robba ring} (and its Frobenius inverses), see \S \ref{sec: loc ana}.

\subsection{Structure of the paper}
In \S \ref{sec: rings}, we study the rings $\wt{\mathbf{B}}^{I}$ and $\mathbf{B}^{I}$ (where $I$ is an interval), as well as their $\gal(\overline{K}/K_\infty)$-invariants which are denoted as $\wt{\mathbf{B}}^{I}_{K_\infty}$ and $\mathbf{B}^{I}_{K_\infty}$. In \S \ref{sec: loc ana}, we  compute locally analytic vectors in $\wt{\mathbf{B}}^{I}_{K_\infty}$; and in \S \ref{sec fieldnorm}, we need to carry out similar calculations when we replace $K_\infty$ with a finite extension.
In \S \ref{sec Ghat la}, we compute the $\hat{G}$-locally analytic vectors in $\wt{\mathbf{B}}^{I}_L$.
All these calculations will be used in \S \ref{sec oc} to carry out the descent of locally analytic vectors, giving us the desired overconvergence result.

\subsection{Notations} \label{subsec notations}

\subsubsection{Convention on ring notations.} \label{subsub style}
 In this paper, we will use many  rings. Let us mention some of the conventions about how we choose the notations; it also serves as a brief index of ring notations.
\begin{enumerate}[leftmargin=*]
\item In \S \ref{period rings}, we define some basic   rings. We also compare them with notations commonly used in integral $p$-adic Hodge theory (see Rem. \ref{rem notation}).

\item In \S \ref{subsec Btilde}, we define the rings $\wt{\mathbf{A}}^{I}$ and  $\wt{\mathbf{B}}^{I}$ (where $I$ is an interval), which are exactly the same as $\wt{\mathbf{A}}^{I}$ and  $\wt{\mathbf{B}}^{I}$ in \cite{Ber08ANT} (which are $\wt{\mathbf{A}}_I$ and  $\wt{\mathbf{B}}_I$ in \cite{Ber02}). (See also the table in \cite[\S 1.1]{Ber08ANT} for a comparison of notations with those of Colmez and Kedlaya).

\item When $Y$ is a ring with a $G_K$-action, $X \subset \overline{K}$ is a subfield, we use $Y_X$ to denote the $\gal(\overline{K}/X)$-invariants of  $Y$. Some examples include when $Y=\wt{\mathbf{A}}^{I}, \wt{\mathbf{B}}^{I}, {\mathbf{A}}^{I}, {\mathbf{B}}^{I}$ and $X=L, K_\infty , M$ where $M/K_\infty$ is a finite extension. This ``style of notation" imitates that of \cite{Ber08ANT}, which uses the subscript $\ast_{K}$ to denote $G_{p^\infty}$-invariants.

%We will use $\wt{\mathbf{A}}^{I}_X$ and  $\wt{\mathbf{B}}^{I}_X$ (where $X$ is a subfield of $\overline{K}$) to mean the $\gal(\overline{K}/X)$-invariants of $\wt{\mathbf{A}}^{I}$ and  $\wt{\mathbf{B}}^{I}$; we will use the cases where $X$ is $L$, or $K_\infty$, or a finite extension $M$ of $K_\infty$.

%$\wt{\mathbf{A}}^{I}_K$ and $\wt{\mathbf{B}}^{I}_K$ to mean the $G_{p^\infty}$-invariants of $\wt{\mathbf{A}}^{I}$ and  $\wt{\mathbf{B}}^{I}$; in the style of our current paper, Berger's rings would be denoted as  $\wt{\mathbf{A}}^{I}_{K_{p^\infty}}$  and $\wt{\mathbf{B}}^{I}_{K_{p^\infty}}$ .

\item In \S \ref{subsec BI}, we define the rings $\mathbf{A}^{I}$ and $\mathbf{B}^{I}$ and study their $G_\infty$-invariants: $\mathbf{A}^{I}_{K_\infty}$ and $\mathbf{B}^{I}_{K_\infty}$. These rings ``correspond" to those rings studied in \cite[\S 6.3, \S 7]{Col08}. Our $\mathbf{A}^{I}$ and $\mathbf{B}^{I}$ are \emph{different} from $\mathbf{A}^{I}$ and $\mathbf{B}^{I}$ in \cite{Col08} (cf. Rem. \ref{rem notation}); fortunately, we are mostly interested in $\mathbf{A}^{I}_{K_\infty}$ and $\mathbf{B}^{I}_{K_\infty}$, and since we are using $K_\infty$ as subscripts, confusions are avoided.
\end{enumerate}
%\item We will denote the rings $\wt{\mathbf{B}}^{[r, +\infty)}_{K_\infty}$ (resp. ${B}^{[r, +\infty)}_{K_\infty}$) as $\wt{\mathbf{B}}^{r, \dagger}_{\rig, K_\infty}$ (resp. ${B}^{r, \dagger}_{\rig, K_\infty}$) in \textbf{blah}. These rings ``correspond" to $\wt{\mathbf{B}}^{r, \dagger}_{\rig, K}$ (resp. ${B}^{r, \dagger}_{\rig, K}$) in \cite{Ber08ANT} (in our notation style, Berger's rings would be denoted as $\wt{\mathbf{B}}^{r, \dagger}_{\rig, K_{p^\infty}}$ and ${B}^{r, \dagger}_{\rig, K_{p^\infty}}$).

%For rings that are used  often in ``\emph{overconvergent} $p$-adic Hodge theory", we will use the ``notation styles" as in \cite{Ber08ANT} (which is a slight variation of those in \cite{Ber02}). (Also.

\subsubsection{Period rings}  \label{period rings}
Let $\wt{\mathbf{E}}^+:=\varprojlim \limits \O_{\overline K}/ p \O_{\overline K}$ where the transition maps are $x\mapsto x^p$, let $\wt{\mathbf{E}}:=\mathrm{Fr} \widetilde{\mathbf{E}}^+$. An element of $\wt{\mathbf{E}}$ can be uniquely represented by $(x^{(n)})_{n \geq 0}$ where $x^{(n)} \in C_p$ and $(x^{(n+1)})^{p}=(x^{(n)})$; let $v_{\wt{\mathbf{E}}}$ be the usual valuation where $v_{\wt{\mathbf{E}}}(x):=v_p(x^{(0)})$.
Let
$$\wt{\mathbf{A}}^+:= W(\wt{\mathbf{E}}^+), \quad  \wt{\mathbf{A}}:= W(\wt{\mathbf{E}}), \quad \wt{\mathbf{B}}^+:= \wt{\mathbf{A}}^+[1/p], \quad \wt{\mathbf{B}}:= \wt{\mathbf{A}}[1/p],$$
where $W(\cdot)$ means the ring of Witt vectors.
There is a unique surjective ring homomorphism  $\theta : \wt{\mathbf{A}}^+ \to   \O_{C_p}$,
which lifts the  projection $\wt{\mathbf{E}}^+ \to \O_{\overline K}/ p$
onto the first factor in the inverse limit.
Let $\mathbf{B}_\dR^+$ be the $\Ker \theta[1/p]$-adic completion of $\wt{\mathbf{B}}^+$ (so the $\theta$-map extends to $\mathbf{B}_\dR^+$).
Let $\underline{\varepsilon}=\{\mu_n\}_{n \geq 0} \in \wt{\mathbf{E}}^+$, let $[\underline{\varepsilon}] \in \wt{\mathbf{A}}^+$ be its Teichm\"uller lift, and let $t:=\log([\underline{\varepsilon}]) \in \mathbf{B}_{\dR}^+$ as usual.

%We denote by $\mathbf{A}_{\cris}$ the $p$-adic completion of the divided power envelope of $\wt{\mathbf{A}}^+$ with respect to $\Ker(\theta)$.

% As usual, we write $\mathbf{B}_\cris^+=\mathbf{A}_\cris[1/p]$ and   $\mathbf{B}_\dR^+$ the $\Ker(\theta)$-adic completion of $\wt{\mathbf{A}}^+[1/p]$ (so the $\theta$ map extends to $\mathbf{B}_\dR^+$). For any subring $A \subset \mathbf{B}^+_\dR$, we define filtration on $A$ by  $\Fil^{I} A = A \cap (\Ker(\theta))^{I}\mathbf{B}^+_\dR$.

Let $\upi:=\{\pi_n\}_{n \geq 0} \in \wt{\mathbf{E}}^+$.
Let $\mathbf{E}^+_{K_\infty} :=k[\![\upi]\!]$, $\mathbf{E}_{K_\infty} :=k((\upi))$, and let $\mathbf{E}$ be the separable closure of $\mathbf{E}_{K_\infty}$ in $\wt{\mathbf{E}}$. By the theory of field of norms (cf. \S \ref{sec fieldnorm}), $\gal(\mathbf{E}/\mathbf{E}_{K_\infty}) \simeq G_{\infty}$. Furthermore, the completion of $\mathbf{E}$ with respect to $v_{\wt{\mathbf{E}}}$ is $\wt{\mathbf{E}}$.

Let $[\underline \pi ]\in \wt{\mathbf{A}}^+$ be the Teichm\"uller lift of $\upi$.
Let $\mathbf{A}^+_{K_\infty} : = W[\![u]\!]$ with Frobenius $\varphi$ extending the arithmetic Frobenius on $W(k)$ and $\varphi (u ) = u ^p$.
There is a $W(k)$-linear Frobenius-equivariant embedding $\mathbf{A}^+_{K_\infty} \inj \wt{\mathbf{A}}^+$ via $u\mapsto [\underline \pi]$.
Let $\mathbf{A}_{K_\infty}$ be the $p$-adic completion of $\mathbf{A}^+_{K_\infty}[1/u]$.
Our fixed embedding $\mathbf{A}^+_{K_\infty}\hookrightarrow \wt{\mathbf{A}}^+$ determined by  $\upi$
uniquely extends to a $\varphi$-equivariant embedding $\mathbf{A}_{K_\infty} \hookrightarrow \wt{\mathbf{A}}$, and we identify $\mathbf{A}_{K_\infty}$ with its image in $\wt{\mathbf{A}}$.
We note that $\mathbf{A}_{K_\infty}$ is a complete discrete valuation ring with uniformizer $p$ and residue field
$\mathbf{E}_{K_\infty}$.

Let $\mathbf{B}_{K_\infty}:=\mathbf{A}_{K_\infty}[1/p]$. Let $\mathbf{B}$ be the $p$-adic completion of the maximal unramified extension of $\mathbf{B}_{K_\infty}$ inside $\wt{\mathbf{B}}$, and let $\mathbf{A} \subset \mathbf{B}$ be the ring of integers.
Let $\mathbf{A}^+ : = \wt{\mathbf{A}}^+ \cap \mathbf{A}$. Then we have:
$$ (\mathbf{A})^{G_\infty} = \mathbf{A}_{K_\infty}, \quad (\mathbf{B})^{G_\infty} = \mathbf{B}_{K_\infty}, \quad
(\mathbf{A}^+)^{G_\infty} = \mathbf{A}^+_{K_\infty}.  $$

\begin{remark} \label{rem notation}
\begin{enumerate}
\item
The following rings (and their ``$\mathbf{B}$-variants") that we defined above,
$$\wt{\mathbf{E}}^+, \quad \wt{\mathbf{E}}, \quad \wt{\mathbf{A}}^+,\quad \wt{\mathbf{A}}, \quad \mathbf{A}^+_{K_\infty},\quad \mathbf{A}_{K_\infty}, \quad \mathbf{A}, \quad \mathbf{A}^+ $$
are precisely the following rings which are commonly used in integral $p$-adic Hodge theory (e.g., in \cite{GL}):
$$R, \quad \Fr R,\quad W(R),\quad W(\Fr R),\quad \gs,\quad \O_\E,\quad \O_{\widehat{\E} ^\ur}, \quad\gs^{\ur}. $$
\item The rings $\mathbf{A}$ and $\mathbf B$ (and their variants, e.g., $\mathbf{A}^I, \mathbf{B}^I$, in \S \ref{subsec BI}) are \emph{different} from the ``$\mathbf{A}$" and ``$\mathbf B$" in \cite{Ber08ANT} or \cite{Col08}. Indeed, they are the same algebraic rings, but with different structures (e.g., Frobenius structure).
In the proof of our final main theorem  (Thm. \ref{thm final}), we will use the font  $\mathbb A$, $\mathbb B$ to denote those rings in the $(\varphi, \Gamma)$-module setting.
 \end{enumerate}
\end{remark}

%We will make use of those rings in \cite{Ber08ANT} or \cite{Col08} (i.e., in the $(\varphi, \Gamma)$-module setting) only once, when we use overconvergent $(\varphi, \Gamma)$-module as input to prove our main overconvergence theorem (Thm. \ref{thm final}); we will use $\mathcal{\mathbf A}$, $\mathcal{\mathbf B}$ to denote

\subsubsection{Valuations and norms} \label{subsub val}
A non-Archimedean valuation of a ring $ A$ is a map $v:  {A} \to \mathbb{R}\cup \{+\infty\}$ such that
$v(x)=+\infty \Leftrightarrow x=0$ and $v(x+y) \geq \inf\{v(x), v(y)\}$. It is called \emph{sub-multiplicative} (resp. \emph{multiplicative}) if $v(xy)\geq v(x)+v(y)$ (resp. $v(xy)=v(x)+v(y)$), for all $x, y$. All the valuations in this paper are sub-multiplicative (some are multiplicative). Given a matrix $T=(t_{i, j})_{i, j}$ over $A$, let $v(T):=\min \{v(t_{i, j})\}$. A non-Archimedean valuation $v$ on $A$ induces a non-Archimedean norm where $\|a\|:=p^{-v(a)}$, and vice versa.

%; $\|\cdot \|$ is {sub-multiplicative} (resp. {multiplicative}) if and only if $v$ is so.
%We say that $A$ is complete with respect to $v$ (resp. $\|\cdot \|$) if $A$ is complete with respect to the topology induced by $v$ (resp. $\|\cdot \|$).

\subsubsection{Some other notations}
Throughout this paper, we reserve $\varphi$ to denote Frobenius operator. We sometimes add subscripts to indicate on which object Frobenius is defined. For example, $\varphi_\M$ is the Frobenius defined on $\M$. We always drop these subscripts if no confusion arises. We use $\Md (A)$ (resp. $\GL_d(A)$) to denote the set of $d \times d$-matrices (resp. invertible $d \times d$-matrices) with entries in $A$.

% Let $A$ be a ring endowed with Frobenius $\varphi_A$ and  $M$  a module over $A$. We always denote $\varphi ^*M := A \otimes_{ \varphi_A, A } M$. Note that if $M$ has a $\varphi _A$-semi-linear endomorphism $\varphi_M: M \to M$ then $1 \otimes \varphi_M : \varphi ^* M \to M$ is an $A$-linear map.Finally $\Md (A)$ always denotes the ring of $d \times d$-matrices with entries in $A$ and $I_d$ denotes the $d \times d $-identity matrix.

\subsection*{Acknowledgement}
We would like to dedicate this paper to Professor Jean-Marc Fontaine and Professor Jean-Pierre Wintenberger, with great admiration.
We thank Laurent Berger and Tong Liu for many useful discussions.
The influence of the work of Laurent Berger and Pierre Colmez in this paper will be evident to the reader. We thank the anonymous referees for several useful comments.
H.G. is partially supported by a postdoctoral position in University of Helsinki, funded by Academy of Finland through Kari Vilonen.
L.P. is currently a PhD student of Laurent Berger at the ENS de Lyon.

\section{A study of some rings} \label{sec: rings}
In this section, we study some rings which are denoted as $\wt{\mathbf{B}}^{I}$ and $\mathbf{B}^{I}$ (where $I$ is an interval). In particular, we study their $G_\infty$-invariants (see \ref{nota fields} for $G_\infty$), which are denoted as $\wt{\mathbf{B}}^{I}_{K_\infty}$ and $\mathbf{B}^{I}_{K_\infty}$.
The results will be used in Section \ref{sec: loc ana} to further determine the link between these rings.
All results in this section are analogues of their $G_{p^\infty}$-versions, established in \cite{Ber02, Col08}; the proofs are also similar.

\subsection{The ring $\wt{\mathbf{B}}^{I}$ and its $G_\infty$-invariants} \label{subsec Btilde}

Let $\overline \pi =\underline{\varepsilon} -1 \in \wt{\mathbf E}^+$ (this is not $\underline \pi$), and let $[\overline \pi] \in \wt{\mathbf{A}}^+$ be its Teichm\"uller lift.
When $ A$ is a $p$-adic complete ring, we use $  A\{X, Y\}$ to denote the $p$-adic completion of  $ A[X, Y]$.
As in \cite[\S 2]{Ber02}, we define the following rings.

\begin{defn} \label{defn wt rings}
  \begin{enumerate}
    \item  Let
    \begin{eqnarray*}
    \wt{\mathbf{A}}^{[r, s]} : &=& \wt{\mathbf{A}}^+ \{\frac{p}{[\overline \pi]^r}, \frac{[\overline \pi]^s}{p} \}, \text{ when } r\leq s \in \mathbb{Z}^{\geq 0}[1/p], s>0; \\
     \wt{\mathbf{A}}^{[r, +\infty]}: &=& \wt{\mathbf{A}}^+ \{\frac{p}{[\overline \pi]^r} \}, \text{ when } r \in \mathbb{Z}^{\geq 0}[1/p];\\
     \wt{\mathbf{A}}^{[+\infty, +\infty]}: &=& \wt{\mathbf{A}}.
    \end{eqnarray*}
Here, to be rigorous, $\wt{\mathbf{A}}^+ \{ {p}/{[\overline \pi]^r},  {[\overline \pi]^s}/{p} \}$ is defined as $\wt{\mathbf{A}}^+ \{X, Y \}/([\overline \pi]^rX-p, pY-[\overline \pi]^s, XY-[\overline \pi]^{s-r})$, and similarly for $\wt{\mathbf{A}}^+ \{ {p}/{[\overline \pi]^r} \}$ (and other similar occurrences later); see \cite[\S 2]{Ber02} for more details.

\item If $I$ is one of the closed intervals above, then let $   \wt{\mathbf{B}}^{I}:=\wt{\mathbf{A}}^{I}[1/p]$.
      \end{enumerate}
\end{defn}

\begin{rem}\label{remnoa00}
  We do not define $\wt{\mathbf{A}}^{[0, 0]}$. Indeed, we will refrain from using  the interval $[0, 0]$ throughout the paper; see Rem. \ref{rem donot defn} and Rem. \ref{remeqnw00} for more remarks concerning $[0, 0]$.
\end{rem}

\subsubsection{} \label{subsubcontain}
If $I$ is one of the closed intervals above, then $\wt{\mathbf{A}}^{I}$ is $p$-adically separated and complete; we use $V^{I}$ to denote its $p$-adic valuation (which is sub-multiplicative). When $I  \subset J$ are two closed intervals as above, then by \cite[Lem. 2.5]{Ber02}, there exists a natural (continuous) embedding
$\wt{\mathbf{A}}^{J}   \hookrightarrow  \wt{\mathbf{A}}^{I}$; we identify $\wt{\mathbf{A}}^{J}$ with its image (as algebraic rings) in this case.

\begin{defn}
When $r \in \mathbb{Z}^{\geq 0}[1/p]$, let
$$\wt{\mathbf{B}}^{[r, +\infty)}: = \bigcap_{n \geq 0} \wt{\mathbf{B}}^{[r, s_n]}$$
where $s_n  \in \mathbb{Z}^{>0}[1/p]$ is any sequence increasing to $+\infty$.  We equip $\wt{\mathbf{B}}^{[r, +\infty)}$ with its natural Fr\'echet  topology.
\end{defn}

\begin{lemma}\label{lemdensity}
\begin{enumerate}
\item \label{item dense} Let $I  \subset J$ be as in \S \ref{subsubcontain}. If $0 \notin J$, then $\wt{\mathbf{B}}^{J}$  is dense in $\wt{\mathbf{B}}^{I}$ with respect to $V^I$.

\item \label{item closed} Suppose $r\leq s \in \mathbb{Z}^{\geq 0}[1/p]$  and $s>0$, then $\wt{\mathbf{B}}^{[0, s]}$ is closed in $\wt{\mathbf{B}}^{[r, s]}$ with respect to $V^{[r, s]}$.

\item \label{item closure}Suppose $0 \leq s_1 \leq s_2 \leq s \leq +\infty$ and $s_2>0$,
then the   closure of $\wt{\mathbf{B}}^{[0, s]}$ in $\wt{\mathbf{B}}^{[s_1, s_2]}$ (with respect to $V^{[s_1, s_2]}$) is $\wt{\mathbf{B}}^{[0, s_2]}$.

\item \label{item frechet} When $r \in \mathbb{Z}^{\geq 0}[1/p]$, $\wt{\mathbf{B}}^{[r, +\infty)}$ is complete with respect to its Fr\'echet topology, and contains
$\wt{\mathbf{B}}^{[r, +\infty]}$ as a dense subring.

\end{enumerate}
\end{lemma}

\begin{proof}
Item \ref{item dense} is easy. To prove Item \ref{item closed}, it suffices to show that
\begin{equation}\label{eqptimes}
\wt{\mathbf{A}}^{[0, s]} \cap p\wt{\mathbf{A}}^{[r, s]} =p\wt{\mathbf{A}}^{[0, s]}
\end{equation}
This is indeed \cite[Lem. 3.2(3)]{Ber16}; however, in \emph{loc. cit.}, the definitions of $\wt{\mathbf{A}}^{[0, s]}$ and $\wt{\mathbf{A}}^{[r, s]}$ rely on the valuations $W^I$  (denoted as ``$V(x, I)$" in \emph{loc. cit.}) which we will recall in Def. \ref{defnew}. Here we give a ``direct" proof using the \emph{explicit} structure of these rings per our Def. \ref{defn wt rings}.
Let $x \in \wt{\mathbf{A}}^{[r, s]}$ such that $px \in \wt{\mathbf{A}}^{[0, s]}$. We can decompose $x=x^- + x^+$ with $x^- \in \wt{\mathbf{A}}^{[r, +\infty]}$ and $x^+ \in \wt{\mathbf{A}}^{[0, s]}$ (the decomposition is not unique). It suffices to show that $px^- \in p\wt{\mathbf{A}}^{[0, s]}$. But indeed,
\begin{eqnarray*}
px^- &\in&  p\wt{\mathbf{A}}^{[r, +\infty]} \cap \wt{\mathbf{A}}^{[0, s]}\\
  &=& p\wt{\mathbf{A}}^{[r, +\infty]} \cap (\wt{\mathbf{A}}^{[r, +\infty]} \cap   \wt{\mathbf{A}}^{[0, s]})\\
  &\subset& p\wt{\mathbf{A}}^{[r, +\infty]} \cap (\wt{\mathbf{A}}^{[s, +\infty]} \cap  \wt{\mathbf{A}}^{[0, s]})\\
   &=& p\wt{\mathbf{A}}^{[r, +\infty]} \cap \wt{\mathbf{A}}^{[0, +\infty]}, \text{ by \cite[Lem. 2.15]{Ber02}}  \\
   &\subset& p\wt{\mathbf{A}} \cap \wt{\mathbf{A}}^{[0, +\infty]}\\
&=&   p\wt{\mathbf{A}}^{[0, +\infty]}.
\end{eqnarray*}
To prove Item \ref{item closure}, simply note that $\wt{\mathbf{B}}^{[0, s]}$ is contained in $\wt{\mathbf{B}}^{[0, s_2]}$ but its closure contains $\wt{\mathbf{B}}^{[0, s_2]}$,  and then apply Item \ref{item closed}. (Note that Items \ref{item closed} and \ref{item closure} correct the statements above \cite[Rem. 2.6]{Ber02}, as Berger never explicitly requires $0 \notin J$.) Item \ref{item frechet} is \cite[Lem. 2.19]{Ber02} (the proof there works for $r=0$ as well).
\end{proof}

\begin{remark} \label{rem Ic}
\begin{enumerate}
  \item For any interval $I$ such that $\wt{\mathbf{A}}^{I}$ and $\wt{\mathbf{B}}^{I}$ are defined, there is a natural bijection (called Frobenius) $\varphi: \wt{\mathbf{A}}^{I} \to \wt{\mathbf{A}}^{pI}$ which is  valuation-preserving.

  \item   For $n \in \mathbb Z^{\geq 0}$, let $r_n: =(p-1)p^{n-1}$. Let
  $$I_c: =\{[r_\ell, r_k], [r_\ell, +\infty], [0, r_k], [0, +\infty] \}, \text{ where } \ell \leq k \text{ run through } \mathbb{Z}^{\geq 0}. $$
  By item (1), in many situations, it would suffice to study $\wt{\mathbf{A}}^{I}$ (and $\wt{\mathbf{B}}^{I}$) for $I \in I_c$ or $I =[+\infty, +\infty]$. The cases for $I$ a general closed interval can be deduced using Frobenius operation; the cases for $I=[r, +\infty)$ can be deduced by taking Fr\'echet completion.
\end{enumerate}
\end{remark}

\begin{Convention} \label{convellk}
From now on, whenever we define rings with an interval as superscript (such as $\wt{\mathbf{A}}^I$, or ${\mathbf{A}}^I$, ${\mathcal{A}}^I$ etc. in the following), we always define in the general case with $\inf(I), \sup(I) \in \{Z^{\geq 0}[1/p], +\infty \}$. But we will only compute (the explicit structure of) these rings with $\inf(I), \sup(I) \in \{0, r_\ell, r_k, +\infty \}$ (when applicable); the general case can always be easily deduced using Frobenius operations.
\end{Convention}

There is another type of valuation $W^I$ on $\wt{\mathbf{B}}^{[r, +\infty]}$, which we quickly recall. A particularly useful fact is that $W^{[s, s]}$ are \emph{multiplicative} valuations (not just sub-multiplicative), see Lem. \ref{lem W} below.

\begin{defn} \label{defnew}
Suppose $r \in \mathbb{Z}^{\geq 0}[1/p]$, and let $x= \sum_{i \geq i_0} p^i[x_i] \in \wt{\mathbf{B}}^{[r, +\infty]}$ ($\subset \wt{\mathbf{B}}^{[+\infty, +\infty]}$). Denote $w_k(x) := \inf_{i \leq k} \{v_{\wt{\mathbf E}}(x_i)\}$ . See \cite[\S 5.1]{Col08} for the properties of $w_k$; in particular, we have $w_k(x+y) \geq \inf \{ w_k(x), w_k(y)\}$ with equality when
$w_k(x)\neq w_k(y)$.
For $s\geq r$ and $s>0$, let
$$W^{[s, s]}(x) :=\inf_{k \geq k_0} \{k+\frac{p-1}{ps}\cdot v_{\wt{\mathbf E}}(x_k)\} =   \inf_{k \geq k_0} \{ k+\frac{p-1}{ps}\cdot w_k(x)\};$$
this is a well-defined valuation (cf. \cite[Prop. 5.4]{Col08}).
For $I \subset [r, +\infty)$ a non-empty closed interval such that $I \neq [0, 0]$, let
$$W^{I}(x) := \inf_{\alpha \in I, \alpha \neq 0} \{W^{[\alpha, \alpha]}(x) \}.$$
\end{defn}
%When $r=0$, then $\wt{\mathbf{B}}^{[0, +\infty]}=\wt{\mathbf{B}}^+$, let $$W^{[0, 0]}(x) := \inf_{x_k \neq 0} \{k\}. $$

\begin{rem}\label{rem donot defn}
We do not define ``$W^{[0, 0]}$".
Indeed when $r=0$, then $\wt{\mathbf{B}}^{[0, +\infty]}=\wt{\mathbf{B}}^+$.
 It might seem that we could   define ``
$W^{[0, 0]}(x) := \inf_{x_k \neq 0} \{k\},$" which is precisely the $p$-adic valuation of $\wt{\mathbf{B}}^+$. However, this valuation is ``\emph{incompatible}" with the valuations $W^{[s, s]}$ for $s>0$. Indeed, one observes that the valuations $W^{[s, s]}$ behave   \emph{continuously}  with respect to $s>0$; but this continuity breaks for ``$s=0$". Indeed, $W^{[s, s]}(x)$ do  not converge to the aforementioned ``$W^{[0, 0]}(x)$" when $s \to  0$; this phenomenon is best explained using the geometric picture of the ``degeneration of annuli to a closed disk", cf. Rem \ref{remeqnw00}.
 Alternatively, it might seem that we could   define ``
$W^{[0, 0]}(x) :=+\infty, \forall x$"; however this is not a valuation anymore (cf. \S \ref{subsub val}).

\end{rem}

\begin{lemma} \label{lem W}
Suppose $r \leq s \in \mathbb{Z}^{\geq 0}[1/p]$ and $s>0$, then the following holds.
\begin{enumerate}

%\item \label{item def} Suppose $x \in \wt{\mathbf{B}}^{[r, +\infty]}$, then
%\begin{itemize} \item  when $s>0$, both $  k+\frac{p-1}{ps}\cdot v_{\wt{\mathbf E}}(x_k)$ and $ k+\frac{p-1}{ps}\cdot w_k(x)$ go to $+\infty$ as $k \to +\infty$, and so $W^{[s, s]}(x)$ is well-defined; \item when $r=s=0$, $W^{[0, 0]}$ is obviously well-defined \end{itemize}

\item \label{item comp}When $r>0$, $\wt{\mathbf{A}}^{[r, +\infty]}$ and $\wt{\mathbf{A}}^{[r, +\infty]}[1/[\overline \pi]]$ are complete with respect to $W^{[r, r]}$.
\item  \label{item mult} $W^{[s, s]}(xy)=W^{[s, s]}(x) +W^{[s, s]}(y), \forall x, y \in \wt{\mathbf{B}}^{[r, +\infty]}.$

\item  \label{item max} Let $x \in \wt{\mathbf{B}}^{[r, +\infty]}$.
\begin{enumerate}
\item When $r>0$, $W^{[r, s]}(x) = \inf \{ W^{[r, r]}(x), W^{[s, s]}(x)\}.$
\item When $r=0$, $W^{[r, s]}(x) (=W^{[0, s]}(x)) =   W^{[s, s]}(x).$
\end{enumerate}

\item \label{item compa} For $x \in \wt{\mathbf{B}}^{[r, +\infty]}$, we have $V^{[r, s]}(x) =\lfloor W^{[r, s]}(x) \rfloor,$ where $V^{[r, s]}(x)$ is defined by considering $x$ as an element in $\wt{\mathbf{B}}^{[r, s]}$.

\item \label{item rs} The completion of  $\wt{\mathbf{B}}^{[r, +\infty]}$  with respect to $W^{[r, s]}$ is isomorphic  to  $\wt{\mathbf{B}}^{[r, s]}$ as topological rings. (Thus, we can extend $W^{[r, s]}$ to $\wt{\mathbf{B}}^{[r, s]}$).
% such that for $x \in \wt{\mathbf{B}}^{[r, s]}$ we have $V^{[r, s]}(x) =\lfloor W^{[r, s]}(x) \rfloor$ and $W^{[r, s]}(x) = \inf \{ W^{[r, r]}(x), W^{[s, s]}(x)\}.$

\end{enumerate}
\end{lemma}

\begin{proof}
All these results are well-known.
Item \ref{item comp} is \cite[Prop. 5.6]{Col08}; note that the ring ``$\wt{\mathbf{A}}^{(0, r]}$" in \emph{loc. cit.} is our $\wt{\mathbf{A}}^{[(p-1)/(pr), +\infty]}[1/[\overline \pi]]$, and the ring of integers in $\wt{\mathbf{A}}^{(0, r]}$ is precisely  our $\wt{\mathbf{A}}^{[(p-1)/(pr), +\infty]}$.
Item \ref{item mult} is \cite[Lem. 21.3]{Ber}.
Item \ref{item max}(a) (the maximum modulus principle) is \cite[Cor. 2.20]{Ber02}; indeed, it follows easily by looking at the definition of $W^{[\alpha, \alpha]}(x)$. Item \ref{item max}(b) follows from similar observation, by noting that $x \in \wt{\mathbf{B}}^{[0, +\infty]}$ implies $v_{\wt{\mathbf E}}(x_k)\geq 0$ for all $k \geq k_0$ in Def. \ref{defnew}.
Item \ref{item compa} is \cite[Lem. 2.7]{Ber02}; the proof works for $r=0$ as well (which Berger did not explicitly mention).
Item \ref{item rs} follows from Item \ref{item compa} and Lem. \ref{lemdensity}.

\end{proof}

%They are proved in \cite[Prop. 2.19]{Ber02} and \cite[Cor. 2.22]{Ber02}.and one can find them in  \cite[\S 2]{Ber02} and \cite[\S 21, \S 27]{Ber}. In particular, a written proof of Item (\ref{item mult}) can be found in \cite[Lem. 21.3]{Ber}.
%Item (2) is proved in \cite[Lem. 2.7]{Ber02}.

\begin{remark} Let $r>0$.
\begin{enumerate}
  \item  Suppose $x \in \wt{\mathbf{B}}^{[r, +\infty]}$, then $W^{[r, r]}(x)\geq 0$ does not imply that $x \in \wt{\mathbf{A}}^{[r, +\infty]}$, it only implies that $x \in \wt{\mathbf{A}}^{[r, r]}$. However, if $x \in \wt{\mathbf{A}}^{[r, +\infty]}[1/[\overline{\pi}]]$, then $W^{[r, r]}(x)\geq 0$ if and only if $x \in \wt{\mathbf{A}}^{[r, +\infty]}$.
  \item In comparison to Lem. \ref{lem W}(\ref{item comp}), $\wt{\mathbf{B}}^{[r, +\infty]}$  is not complete with respect to  $W^{[r, r]}$; indeed, its completion is $\wt{\mathbf{B}}^{[r, r]}$  by Lem. \ref{lem W}(\ref{item rs}).
      \item In comparison to Lem. \ref{lem W}(\ref{item rs}), the completion of  $\wt{\mathbf{A}}^{[r, +\infty]}$  with respect to $W^{[r, s]}$ is {strictly} contained in $\wt{\mathbf{A}}^{[r, s]}$ (which is already the case when $r=s$ by Lem. \ref{lem W}(\ref{item comp})). Also note that $\wt{\mathbf{A}}^{[r, s]}$ is complete with respect to $W^{[r, s]}$, since it is the ring of integers in $\wt{\mathbf{B}}^{[r, s]}$. (Thus, $\wt{\mathbf{A}}^{[r, +\infty]}$ is a closed subset of $\wt{\mathbf{A}}^{[r, r]}$ with respect to $W^{[r, r]}$).
\end{enumerate}
\end{remark}

Let $I$ be an interval. When $\wt{\mathbf{B}}^{I}$ (resp. $\wt{\mathbf{A}}^{I}$) is defined, let $\wt{\mathbf{B}}^{I}_{K_\infty}: = (\wt{\mathbf{B}}^{I})^{G_\infty}$ (resp. $\wt{\mathbf{A}}^{I}_{K_\infty}: = (\wt{\mathbf{A}}^{I})^{G_\infty}$).
Recall that as in \cite[\S 2.2]{Ber02}, when $r_n \in I$, there exists $\iota_n: \wt{\mathbf{B}}^{I} \hookrightarrow \mathbf{B}_{\dR}^+$. Let $\theta: \mathbf{B}_{\dR}^+ \to C_p$ be the usual map.

\begin{lemma} \label{prop theta ker}
Let $q:=  ( [\underline \varepsilon]^p-1 )/([\underline \varepsilon]-1)$.
Suppose $I=[r_\ell, r_k]$ or $[0, r_k]$. We have
\begin{enumerate}
\item $\Ker(\theta\circ \iota_k: \wt{\mathbf{A}}^{I} \to C_p) =  \frac{\varphi^{k-1}(q)}{p}\wt{\mathbf{A}}^{I}=  \frac{\varphi^{k }(E(u))}{p}\wt{\mathbf{A}}^{I}$,

\noindent  $\Ker(\theta\circ \iota_k: \wt{\mathbf{B}}^{I} \to C_p) =\varphi^{k-1}(q)\wt{\mathbf{B}}^{I} =\varphi^{k }(E(u))\wt{\mathbf{B}}^{I}$.
\item $\Ker(\theta\circ \iota_k: \wt{\mathbf{A}}_{K_\infty}^{I} \to C_p) =\frac{\varphi^{k }(E(u))}{p} \wt{\mathbf{A}}^{I}_{  K_\infty}$,

\noindent  $\Ker(\theta\circ \iota_k: \wt{\mathbf{B}}_{K_\infty}^{I} \to C_p) =\varphi^{k }(E(u))\wt{\mathbf{B}}^{I}_{  K_\infty}$.
\end{enumerate}
\end{lemma}
%  \item $\Ker(\theta\circ \iota_n: \BBBKrn \to C_p) =\varphi^{n }(E(u))\BBBKrn$.
 % \item  $\Ker(\theta\circ \iota_n: \BBKrn \to C_p) =\varphi^{n}(E(u))\BBKrn$.

\begin{proof}
Item (1) is easily deduced from \cite[Prop. 2.12]{Ber02}, because $E(u)$ and $\varphi^{-1}(q)$ generate the same ideal in $\wt{\mathbf{A}}^+$ (i.e., the kernel of the $\theta$-map in \S \ref{period rings}).
Item (2) is an easy consequence of (1).
\end{proof}

In the following, we study more detailed structure of the rings $\wt{\mathbf{B}}_{K_\infty}^{I}$ and $\wt{\mathbf{A}}_{K_\infty}^{I}$. These results (Lem. \ref{lem decomp}, Prop. \ref{lem decomp 3} and Prop. \ref{lemdense}) will not be used in this paper. We still include them here because they are standard and will be useful in the future; also, they serve as prelude to the computation of the rings ${\mathbf{B}}_{K_\infty}^{I}$ and ${\mathbf{A}}_{K_\infty}^{I}$ in next subsection.

\begin{lemma}\label{lem decomp}
%Suppose $I=[0, r_k]$, and let $J=I \cap [r_\ell, +\infty] =[r_\ell, r_k]$.
Suppose $\ell \leq k$, then we have the following short exact sequence
\begin{equation} \label{eqnseqa}
 0 \to \wt{\mathbf{B}}_{K_\infty}^{[0, +\infty]} \to   \wt{\mathbf{B}}_{ K_\infty}^{[r_\ell, +\infty]}\oplus \wt{\mathbf{B}}^{[0, r_k]}_{ K_\infty}  \to  \wt{\mathbf{B}}^{[r_\ell, r_k]}_{K_\infty} \to 0,
\end{equation}
where the second arrow is $x \mapsto (x, x)$, and the third arrow is $(a, b) \mapsto a-b$.
\end{lemma}
\begin{proof}
The proof is analogous to \cite[Lem. 2.27]{Ber02}. By the proof of \cite[Lem. 2.18]{Ber02}, we have
$$ 0 \to \wt{\mathbf{B}}^{[0, +\infty]} \to   \wt{\mathbf{B}}^{[r_\ell, +\infty]}\oplus \wt{\mathbf{B}}^{[0, r_k]}   \to  \wt{\mathbf{B}}^{[r_\ell, r_k]} \to 0.$$
Take $G_\infty$-invariants, and consider the long exact sequence, it suffices to show that the map
\begin{equation}\label{eqndelta}
\delta: \wt{\mathbf{B}}^{[r_\ell, r_k]}_{K_\infty} \to H^1(G_\infty, \wt{\mathbf{B}}^+ )
\end{equation}
is the zero map. By exactly the same argument as in \cite[Lem. 2.27]{Ber02}, it suffices to show that $H^1(G_\infty, \mathfrak{m}_{\wt{\mathbf{E}}^+})=0$ (where $\mathfrak{m}_{\wt{\mathbf{E}}^+}$ is the maximal ideal of $\wt{\mathbf{E}}^+$); and this is an analogue of \cite[Prop. IV.1.4(iii)]{Col98}. Indeed, the ring $\wt{\mathbf{E}}^+$ satisfies the conditions (C1), (C2) and (C3) in \cite[IV.1]{Col98} with respect to our APF extension $K_\infty$ (note that the $K_\infty$ in \emph{loc. cit.} is our $K_{p^{\infty}}$); the proof is verbatim as in \cite[Rem. IV.1.1(iii)]{Col98}, since the theory of fields of norms for our extension $K_\infty$ also works (see e.g. \cite[\S 2]{Bre99b} for a detailed development).
\end{proof}
%(Indeed, it suffices to show that $\wt{\mathbf{A}}^{[r_\ell, +\infty]}\cap \wt{\mathbf{A}}^{[0, r_k]}=\wt{\mathbf{A}}^{[0, +\infty]}$, which follows from \cite[Lem. 2.15]{Ber02}.)
% $H^1(G_\infty, W(\mathfrak{m}_R))=0$. Via standard d\'{e}vissage, it suffices to show that $H^1(G_\infty, \mathfrak{m}_R)=0$.

\begin{prop}\label{lem decomp 3}

\begin{enumerate}
\item $ \wt{\mathbf{A}}_{K_\infty}^{[0, r_k]} = \wt{\mathbf{A}}^+_{K_\infty}\{ \frac{\varphi^k(E(u))}{p} \} =  \wt{\mathbf{A}}^+_{K_\infty}\{ \frac{u^{ep^k}}{p} \} $.
\item $\wt{\mathbf{A}}_{K_\infty}^{[r_\ell, +\infty]} = \wt{\mathbf{A}}^+_{K_\infty}\{ \frac{p}{u^{ep^\ell}}    \}  $.
\item $\wt{\mathbf{B}}_{K_\infty}^{[r_\ell, r_k]}  = \wt{\mathbf{A}}^+_{K_\infty} \{ \frac{p}{u^{ep^\ell}}, \frac{u^{ep^k}}{p}  \}[\frac{1}{p}] $.
\end{enumerate}
\end{prop}
\begin{proof}  Item (1) is an analogue of \cite[Lem. 2.29]{Ber02}. By applying $\varphi^{-k}$ to all rings, it suffices to prove it when $k=0$.
By definition of $ \wt{\mathbf{A}}^{[0, r_k]}$, it is obvious that $ \wt{\mathbf{A}}^+_{K_\infty}\{ \frac{\varphi^k(E(u))}{p} \} \subset \wt{\mathbf{A}}_{K_\infty}^{[0, r_k]} $; it suffices to show the inclusion is identity.
  Since $\wt{\mathbf{E}}^+_{K_\infty}/u^e \wt{\mathbf{E}}^+_{K_\infty}$ has a basis of $u^i$ for $i \in \mathbb Z [1/p] \cap [0, e)$, we can easily deduce that $\theta: \wt{\mathbf{A}}^+_{K_\infty} \to \O_{\hat{K_\infty}}$ is surjective.
  Given $x \in \wt{\mathbf{A}}_{K_\infty}^{[0, r_0]}$, we recursively define two sequences $x_i \in \wt{\mathbf{A}}_{K_\infty}^{[0, r_0]}$ and $a_i \in \wt{\mathbf{A}}^+_{K_\infty}$ as follows:
  \begin{itemize}
  \item  let $x_0=x$;
  \item  choose any $a_i \in \wt{\mathbf{A}}^+_{K_\infty}$ such that $\theta(a_i)=\theta(x_i)  \in \O_{\hat{K_\infty}}$;
  \item let $x_{i+1} :=(x_i-a_i)\cdot \frac{p}{E(u)}$, then $x_{i+1} \in \wt{\mathbf{A}}_{K_\infty}^{[0, r_0]}$ by Lem. \ref{prop theta ker}.
  \end{itemize}
  Then it is easy to check that $x=\sum_{i \geq 0} a_i (E(u)/p)^i$ with  $a_i \to 0$.

  For Item (2), again it suffices to consider the case $\ell=0$. Let $x \in \wt{\mathbf{A}}_{K_\infty}^{[r_0, +\infty]}$, write it as $x=\sum_{k \geq 0}p^k[x_k]$, then clearly $x_k \in (\wt{\mathbf E})^{G_\infty}$. Since $(pr_0)/(p-1)\cdot  k+  v_{\wt{\mathbf E}}(x_k) \to +\infty$ as $k \to +\infty$, so $k+v_{\wt{\mathbf E}}(x_k) \to +\infty$, and so $v_{\wt{\mathbf E}}(x_k\cdot \underline{\pi}^{ek}) \to +\infty$. Then one can easily show that $x\ \in \wt{\mathbf{A}}^+_{K_\infty}\{ \frac{p}{u^{e }}    \}$.

% For item (2), again it suffices to consider the case $\ell=0$. Recall $r_0=\frac{p-1}{p}$.Suppose $x \in \wt{\mathbf{A}}_{K_\infty}^{[r_\ell, +\infty]}$, and write $x=\sum_{k \geq 0}p^k[x_k]$, then clearly $x_k \in (\wt{\mathbf E})^{G_\infty}$.Since $W_{[r_0, r_0]}(x)$ exists, this implies that $x \in \frac{1}{u^{ek}}\wt{\mathbf{A}}^+_{K_\infty} +p^{k+1}\wt{\mathbf{A}}_{K_\infty}, \forall k \geq 0$. Then use variant of Gao-Liu lemma to conclude.

  Consider Item (3). By Lem. \ref{lem decomp}, any element  $x \in \wt{\mathbf{B}}_{K_\infty}^{[r_\ell, r_k]}$ can be written as a sum $x=a+b$ with $a \in \wt{\mathbf{B}}_{ K_\infty}^{[r_\ell, +\infty]} $ and $b \in  \wt{\mathbf{B}}^{[0, r_k]}_{ K_\infty} $, so we can apply Items (1) and (2) to conclude.
 \end{proof}

\begin{rem}
We do not know if we have
\begin{equation}\label{eqnabf}
\wt{\mathbf{A}}_{K_\infty}^{[r_\ell, r_k]}  = \wt{\mathbf{A}}^+_{K_\infty} \{ \frac{p}{u^{ep^\ell}}, \frac{u^{ep^k}}{p}  \}.
\end{equation}
Equivalently, we do not know if the ``$\wt{\mathbf{A}}$"-version of \eqref{eqnseqa} (by changing all $\wt{\mathbf{B}}$ there to $\wt{\mathbf{A}}$) holds. Indeed, to show that the $\delta$-map in \eqref{eqndelta} is the zero map following \cite[Lem. 2.27]{Ber02}, it is critical to use the fact that $u$ is invertible in $\wt{\mathbf{B}}^{[r_\ell, r_k]}_{K_\infty}$ (which fails in $\wt{\mathbf{A}}^{[r_\ell, r_k]}_{K_\infty}$). We tend to think that \eqref{eqnabf} holds. In particular, the ``$\mathbf{A}$"-version of  \eqref{eqnabf} holds, cf. Prop.\ref{cor A rl rk}; the proof critically relies on the \emph{unique} decompostion $f=f^-+f^+$ in Lem. \ref{lem laurent series}, which fails inside $\wt{\mathbf{A}}_{K_\infty}^{[r_\ell, r_k]}$. Fortunately, \eqref{eqnabf} is perhaps useless anyway; e.g., the ``$\Kpinfty$-version" was never studied in \cite{Ber02}. In contrast, Prop. \ref{cor A rl rk} (indeed Cor. \ref{lem aa inter}) plays a key role in our Thm. \ref{thm loc ana gamma 1}.
\end{rem}

\begin{prop} \label{lemdense}
\begin{enumerate}
\item  The ring $\wt{\mathbf{B}}_{K_\infty}^{[r_\ell, +\infty]}$ is dense in $\wt{\mathbf{B}}_{ K_\infty}^{[r_\ell, +\infty)}$ for the Fr{\'e}chet topology.
\item  The ring $\wt{\mathbf{B}}_{K_\infty}^{[0, +\infty]}$ is dense in $\wt{\mathbf{B}}_{ K_\infty}^{[0, +\infty)}$ for the Fr{\'e}chet topology.
\end{enumerate}

\end{prop}
\begin{proof}
The proof (for both Items) is verbatim as the proof of \cite[Prop. 2.30]{Ber02}, by changing $q$ there to $E(u)$.
\end{proof}

\subsection{The ring ${\mathbf B}^{I}$ and its $G_\infty$-invariants} \label{subsec BI}
%Note that our $\mathbf A, \mathbf B$ (and thus $\mathbf{A}^{I}, \mathbf{B}^{I}$) are different from those in \cite[\S 1.3]{Ber02} (and \cite[\S 6.3]{Col08}). Fortunately, what we really care about are $G_\infty$-invariants of these rings, which will be denoted as $\mathbf{A}^{I}_{K_\infty}, \mathbf{B}^{I}_{K_\infty}$. In the situation of \cite[\S 6.3]{Col08}, those $G_{p^{\infty}}$-invariants are denoted as $\mathbf{A}^{I}_K, \mathbf{B}^{I}_K$, so the notations are different.

\begin{definition}
\begin{enumerate}
\item When $r \in \mathbb{Z}^{\geq 0}[1/p]$, let
$$\mathbf{A}^{[r, +\infty]} := \mathbf A \cap \wt{\mathbf{A}}^{[r, +\infty]}, \quad \mathbf{B}^{[r, +\infty]} :=\mathbf B \cap \wt{\mathbf{B}}^{[r, +\infty]} .$$
\item When $r, s \in \mathbb{Z}^{\geq 0}[1/p], s \neq 0$, let $\mathbf{B}^{[r, s ]}$ be the closure of $\mathbf{B}^{[r, +\infty]}$  in $\wt{\mathbf{B}}^{[r, s]}$ with respect to $W^{[r, s]}$ (By Rem. \ref{remnoa00}, there is no $\wt{\mathbf{B}}^{[0, 0]}$ hence no  ${\mathbf{B}}^{[0, 0]}$). Let $\mathbf{A}^{[r, s ]}: =\mathbf{B}^{[r, s ]} \cap \wt{\mathbf{A}}^{[r, s]}$, which is the ring of integers in $\mathbf{B}^{[r, s ]}$.

%$\mathbf{A}^{[r, s ]}$ (resp. $\mathbf{B}^{[r, s ]}$) be the closure of $\mathbf{A}^{[r, +\infty]}$ (resp. $\mathbf{B}^{[r, +\infty]}$) in $\wt{\mathbf{B}}^{[r, s]}$ with respect to the topology defined by $W^{[r, s]}$.

\item When $r \in \mathbb{Z}^{\geq 0}[1/p]$, let
$$\mathbf{B}^{[r, +\infty)} : =\bigcap_{n \geq 0} \mathbf{B}^{[r, s_n]}$$
where $s_n  \in \mathbb{Z}^{>0}[1/p]$ is any sequence increasing to $+\infty$.
%Let $\mathbf{A}^{[r, +\infty)} : = \mathbf{B}^{[r, +\infty)} \cap \wt{\mathbf{A}}^{[r, r]}$.
%When $I$ is a  non-empty  interval, let
%\[ {B}^{I} : = \bigcap_{[r, s] \subset I}  {B}^{[r, s]}. \]
%and equip it with the Frechet topology.
\end{enumerate}
\end{definition}

\begin{definition}
 For $r \in \mathbb{Z}^{\geq 0}[1/p]$, let $\mathcal{A}^{[r, +\infty]}(K_0)$ be the ring consisting of infinite series $f=\sum_{k \in \mathbb Z} a_kT^k$ where $a_k \in W(k)$ such that $f$ is a holomorphic function on the annulus defined by
$$v_p(T) \in  (0,  \quad     \frac{p-1}{ep}\cdot \frac{1}{r} ].$$
(Note that when $r=0$, it implies that $a_k=0, \forall k<0$).
Let $\mathcal{B}^{[r, +\infty]}(K_0): =\mathcal{A}^{[r, +\infty]}(K_0)[1/p] $.
%and such that (resp. the set $\{v_p(a_k)\}_{k \in \mathbb Z}$) is a bounded below.
\end{definition}
%\item  Suppose $0 < r \leq s <+\infty$.
%Let $\mathcal{B}^{[r, s]}(K_0)$ be the set consisting of Laurent series $f=\sum_{k \in \mathbb Z} a_kT^k$ where $a_k \in K_0$ such that $f$ is a holomorphic function on the annulus defined by
%$$v_p(T) \in  [ \frac{p-1}{ep}\cdot \frac{1}{s},  \quad \frac{p-1}{ep}\cdot \frac{1}{r} ].$$

\begin{defn}\label{defcalwss}
Suppose  $f=\sum_{k \in \mathbb Z} a_kT^k \in \mathcal{B}^{[r, +\infty]}(K_0)$.
\begin{enumerate}
\item  When $s \geq r$, $s>0$, let
$$
\mathcal W^{[s, s]}(f) : =  \inf_{k \in \Z}  \{ v_p(a_k)  + \frac{p-1}{ps}\cdot   \frac{k}{e}     \}. $$

%\item When $r=0$, let $$ \mathcal W^{[0, 0]}(f) : =  \inf_{k \in \Z}  \{  v_p(a_k)   \}. $$

\item For $I  \subset [r, +\infty)$ a non-empty   closed  interval,
let
\begin{equation}\label{eqwiwi}
\mathcal W^{I}(f) := \inf_{\alpha \in I, \alpha \neq 0}\mathcal W^{[\alpha, \alpha]}(f).
\end{equation}

\end{enumerate}
\end{defn}
It is well-known that $\mathcal W^{[s, s]}$ for any $s>0$ is an multiplicative valuation; thus $\mathcal W^{I}$ is an sub-multiplicative valuation.

\begin{defn}
For $r\leq s \in \mathbb{Z}^{\geq 0}[1/p], s \neq 0$, let $\mathcal{B}^{[r, s]}(K_0)$ be the completion of $\mathcal{B}^{[r, +\infty]}(K_0)$ with respect to $\mathcal W^{[r, s]}$. Let $\mathcal{A}^{[r, s]}(K_0)$ be the ring of integers in $\mathcal{B}^{[r, s]}(K_0)$ with respect to $\mathcal W^{[r, s]}$.
\end{defn}
%Let $\mathcal{A}^{[r, s]}(K_0)$ (resp. $\mathcal{B}^{[r, s]}(K_0)$) be the completion of $\mathcal{A}^{[r, +\infty]}(K_0)$ (resp. $\mathcal{B}^{[r, +\infty]}(K_0)$) with respect to $\mathcal W^{[r, s]}$.

\begin{lemma} \label{lem laurent series}
\begin{enumerate}
\item $\mathcal{B}^{[r_\ell, +\infty]}(K_0)$ is complete with respect to $\mathcal W^{[r_\ell, r_\ell]}$, and $\mathcal{A}^{[r_\ell, +\infty]}(K_0)$ is the ring of integers with respect to this valuation. Furthermore, we have
\begin{equation} \label{eqa0}
\mathcal{A}^{[r_\ell, +\infty]}(K_0)  =W(k)[\![T]\!] \{ \frac{p}{T^{ep^\ell}} \}.
\end{equation}

\item We have $\mathcal W^{[0, r_k]}(x) =  \mathcal W^{[r_k, r_k]}(x)$. Furthermore, $\mathcal{B}^{[0, r_k]}(K_0)$ is  the ring consisting of infinite series $f=\sum_{k \in \mathbb Z} a_kT^k$ where $a_k \in K_0$ such that $f$ is a holomorphic function on the \emph{closed disk} defined by
$$v_p(T) \in  [ \frac{p-1}{ep}\cdot \frac{1}{r_k}, \quad +\infty ].$$
Indeed,  we have
\begin{equation}
\label{eqa1}\mathcal{A}^{[0, r_k]}(K_0)=W(k)[\![T]\!] \{ \frac{T^{ep^k}}{p}\},  \text{ and } \mathcal{B}^{[r, s]}(K_0)=\mathcal{A}^{[r, s]}(K_0)[1/p].
\end{equation}

\item For $I=[r, s]= [r_\ell, r_k]$, we have $\mathcal W^{I}(x) = \inf \{\mathcal W^{[r, r]}(x), \mathcal W^{[s, s]}(x)\}$. Furthermore, $\mathcal{B}^{[r, s]}(K_0)$ is  the ring consisting of infinite series $f=\sum_{k \in \mathbb Z} a_kT^k$ where $a_k \in K_0$ such that $f$ is a holomorphic function on the annulus defined by
$$v_p(T) \in  [ \frac{p-1}{ep}\cdot \frac{1}{s},  \quad \frac{p-1}{ep}\cdot \frac{1}{r} ].$$
Indeed, we have
\begin{equation}
\label{eqa2}\mathcal{A}^{[r_\ell, r_k]}(K_0)=W(k)[\![T]\!] \{ \frac{p}{T^{ep^\ell}} , \frac{T^{ep^k}}{p}\},   \text{ and } \mathcal{B}^{[r, s]}(K_0)=\mathcal{A}^{[r, s]}(K_0)[1/p].
\end{equation}

\end{enumerate}
%In particular, we have
%$$\mathcal{A}^{[r_\ell, r_k]}(K_0)=W(k)[\![T]\!] \{ \frac{p}{T^{ep^\ell}} , \frac{T^{ep^k}}{p}\}, $$
%and $\mathcal{B}^{[r_\ell, r_k]}(K_0)=\mathcal{A}^{[r_\ell, r_k]}(K_0)[1/p]$.
\end{lemma}
\begin{proof}
Everything is elementary and well-known; we only sketch how to prove \eqref{eqa2}. Let $f=\sum_{k \in \mathbb{Z}} a_kT^k \in \mathcal{A}^{[r_\ell, r_k]}(K_0)$, then we can decompose $f=f^-+f^+$ uniquely where $f^-=\sum_{k<0} a_kT^k $ and $f^+=\sum_{k \geq 0} a_kT^k $. Since the valuations $\mathcal W^{[s, s]}$ are defined in a ``term-wise" fashion (i.e., $\mathcal W^{[s, s]}(f) =\inf_k \mathcal W^{[s, s]}(a_kT^k)$), it is easy to see that $f^- \in \mathcal{A}^{[r_\ell, +\infty]}(K_0)$ and $f^+ \in \mathcal{A}^{[0, r_k]}(K_0)$; then we can conclude using \eqref{eqa0} and \eqref{eqa1}.
\end{proof}

\begin{rem}\label{remeqnw00}
When $r=0$ in Def. \ref{defcalwss}, it actually makes perfect sense to define
\begin{equation}\label{eqnw00}
\mathcal{W}^{[0, 0]}(f):= v_p(a_0).
\end{equation}
Indeed, the valuations $\mathcal{W}^{[s, s]}(f)$ (for $s>0$) correspond to the Gauss norms on the \emph{circle} of radius $p^{-(p-1)/eps}$, and this ``$ \mathcal{W}^{[0, 0]}(f)$"  corresponds precisely to the norm on the zero point.
Using \eqref{eqnw00}, we could even modify \eqref{eqwiwi} (when $0 \in I$) to be
\begin{equation}
\mathcal W^{[0, s]}(f) := \inf_{\alpha \in [0, s]}\mathcal W^{[\alpha, \alpha]}(f).
\end{equation}
But these two definitions give the same valuation (namely, $\mathcal W^{[s, s]}(f)$), because the  zero point  is not on the boundary (of the relevant closed disk) anyway! Since we do not have a ``compatible" $W^{[0, 0]}$ by Rem. \ref{rem donot defn}, it is better for us to completely ignore ``$\mathcal{W}^{[0, 0]}$".
\end{rem}

\begin{lemma} \label{lem inj}
 Let $r\leq s \in \mathbb{Z}^{\geq 0}[1/p]$, $s > 0$.
\begin{enumerate}
\item
The map $f(T) \mapsto f(u)$ induces ring isomorphisms
\begin{eqnarray*}
 \mathcal{A}^{[0, +\infty]}(K_0)  &\simeq& \mathbf{A}^{[0,+\infty]}_{K_\infty}, \text{ when } r=0; \\
 \mathcal{A}^{[r, +\infty]}(K_0) &\simeq &  \mathbf{A}^{[r,+\infty]}_{K_\infty}[1/u], \text{ when } r>0.
\end{eqnarray*}
Furthermore, given $f \in \mathcal{A}^{[r, +\infty]}(K_0)$,  we have
$$\mathcal W^{[s, s]}(f(T)) = W^{[s, s]}(f(u)).$$

\item The map $f(T) \mapsto f(u)$ induces isometric isomorphisms
\begin{eqnarray*}
 \mathcal{A}^{[0, s]}(K_0)  &\simeq& \mathbf{A}^{[0,s]}_{K_\infty}, \text{ when } r= 0; \\
 \mathcal{A}^{[r, s]}(K_0) &\simeq &  \mathbf{A}^{[r,s]}_{K_\infty}, \text{ when } r>0.
\end{eqnarray*}
\end{enumerate}

\end{lemma}
Before we prove the lemma, we introduce the section $s$ and use it to build an approximating sequence.

\subsubsection{The section $s$.}  \label{subsub sec s Kinfty}
 Denote
$$s: \mathbf{A}_{K_\infty}/p \to \mathbf{A}_{K_\infty}$$
the section where for $\overline{x}=\sum_{i \gg-\infty} \overline{a_i}u^i$, let $s(\overline{x}):=\sum_{i \gg-\infty} [\overline{a_i}]u^i$. One can see that $s(\overline{x}) \in \mathbf{A}^{[r,+\infty]}_{K_\infty}[1/u]$ for any $r\geq 0$. Furthermore, for any $k \geq 0$, we have
\begin{equation}\label{w0}
  w_k(s(\overline{x})) =\inf_i \{w_k([\overline{a_i}]u^i)  \}= \frac{1}{e}\min  \{i: \overline{a_i} \neq 0\} =v_{\wt{\mathbf E}}(\overline{x}),
\end{equation}
where the first identity holds because $w_k([\overline{a_i}]u^i)$ are distinct for different $i$.

%  =w_0(s(\overline{x})) =w_0(x) =v_{\wt{\mathbf E}}(\overline{x}) = \frac{1}{e}\min  \{i: \overline{a_i} \neq 0\}.
\subsubsection{An approximating sequence.} \label{subsub approx}
Let $r\geq 0$, given  $x \in  \mathbf{A}^{[r,+\infty]}_{K_\infty}[1/u]$, define a sequence $\{x_n\}$ in $\mathbf{A}^{[r,+\infty]}_{K_\infty}[1/u]$ where $x_0=x$ and $x_{n+1}:=p^{-1}(x_n-s(\overline{x_n}))$.
Then $x=\sum_{n\geq 0}p^n \cdot s(\overline{x_n})$, and we have that
\begin{eqnarray*}
w_k(x_{n+1}) &= & w_{k+1}(px_{n+1}) \\
  &\geq& \inf  \{ w_{k+1}(x_n), w_{k+1}( s(\overline{x_n}) )\} \\
   &=& \inf  \{ w_{k+1}(x_n), w_{0}(x_n )\}, \text{ by } \eqref{w0},\\
&=& w_{k+1}(x_n).
  \end{eqnarray*}
Iterating the above process, we get
\begin{equation}\label{wn}
  w_0(x_n) \geq w_n(x_0)=w_n(x).
 \end{equation}
%=\inf  \{ w_{k+1}(x_n), w_{0}(x_n )  \}

\begin{proof}[Proof of Lem. \ref{lem inj}]
Lem. \ref{lem inj} is  an analogue of \cite[Prop. 7.5]{Col08}, and the proof uses similar ideas. It suffices to prove Item (1).

% Note that when $s=0$ (in the case $r=0$), then everything is trivial; so in the following, we suppose $s>0$.

%In the following proof, for the writing to be rigorous, we should assume that $r>0$ and $s \neq +\infty$; but the reader can easily see that those cases work as well.

\textbf{Part 1}. Given $f(T)=\sum_{k \in \mathbb Z} a_kT^k \in \mathcal{A}^{[r, +\infty]}(K_0)$, then for any $s \in [r, +\infty), s>0$,
$$W^{[s, s]}(f(u)) \geq \inf_k \{W^{[s, s]}(a_ku^k)  \} =\inf_k \{   v_p(a_k) +\frac{p-1}{ps}\cdot \frac{k}{e}   \}  =\mathcal W^{[s, s]}(f(T)). $$
When $r>0$, $ v_p(a_k) +\frac{p-1}{pr}\cdot \frac{k}{e} \to +\infty$ for $k \to +\infty$  or $k \to -\infty$. By Lem. \ref{lem W}, $\mathbf{A}^{[r,+\infty]}_{K_\infty}[1/u]$ is complete with respect to $W^{[r, r]}$; thus $f(u) \in \mathbf{A}^{[r,+\infty]}_{K_\infty}[1/u]$ when $r>0$.
 When $r=0$, then it is clear that $f(u) \in \mathbf{A}^{[0,+\infty]}_{K_\infty}.$
Also, it is obvious that the map $f(T) \mapsto f(u)$ is injective.

% so $f(u) \in \mathbf{A}^{[r,+\infty]}_{K_\infty}[1/u]$ when $r>0$, by noting that $\mathbf{A}^{[r,+\infty]}_{K_\infty}[1/u]$ is complete with respect to $W^{[r, r]}$.

%\begin{eqnarray*}f(u) &\in & \wt{\mathbf{B}}^{[r,+\infty]} \cap \mathbf{A}_{K_\infty} \\&=& \mathbf{A}^{[r,+\infty]}_{K_\infty}[1/u] \text{ when } r>0 \quad (\text{resp. } \mathbf{A}^{[0,+\infty]}_{K_\infty} \text{ when } r=0)\end{eqnarray*}

\textbf{Part 2}.
Given $x \in \mathbf{A}^{[r,+\infty]}_{K_\infty}[1/u]$ when $r>0$ (resp. $x\in \mathbf{A}^{[0,+\infty]}_{K_\infty}$ when $r=0$), let $\{x_n\}$ be the sequence constructed in \S \ref{subsub approx} (note that when $x\in \mathbf{A}^{[0,+\infty]}_{K_\infty}$, then $x_n\in \mathbf{A}^{[0,+\infty]}_{K_\infty}, \forall n$ ).
Let $f_n(T)$ be the formal series $\sum_{k \in \mathbb Z} f_{n, k}T^k$ such that $f_n(u) =s(\overline{x_n})$, and let $f(T):=\sum_{n \geq 0}p^nf_n(T)$.
By \eqref{wn},
\begin{equation*}
  v_{\wt{\mathbf E}}(\overline{x_n})= w_0(x_n)\geq w_n(x),
\end{equation*}
so the expression for $s(\overline{x_n})$ would be of the form $\sum_{i \geq e w_n(x)} [\overline{a_i}]u^i$ (recall that $v_{\wt{\mathbf E}}(u)=1/e$). Thus $f_n(T)=\sum_{i \geq e w_n(x)} [\overline{a_i}]T^i$, and so
$$\mathcal W^{[s, s]}(p^nf_n(T)) \geq   \mathcal W^{[s, s]}(p^n T^{\lceil e w_n(x) \rceil})  \geq   n+ \frac{p-1}{ps}\cdot \frac{1}{e}\cdot ew_n(x) \geq W^{[s, s]}(x). $$
When $r>0$, $  n+\frac{p-1}{pr}\cdot w_n(x) \to +\infty$ when $n \to +\infty$, so $f(T)$ converges in  $\mathcal{A}^{[r, +\infty]}(K_0)$. (When $r=0$, $f(T)$ automatically converges in $\mathcal{A}^{[0, +\infty]}(K_0)$).
 It is clear $f(u)=x$, and $\mathcal W^{[s, s]}(f(T)) \geq W^{[s, s]}(x).$ Combined with Part 1, this concludes the proof.
\end{proof}

\begin{prop} \label{cor A rl rk}
We have
\begin{eqnarray*}
 \mathbf{A}^{[0, +\infty]}_{K_\infty} &=& \mathbf{A}^+_{K_\infty}, \\
 \mathbf{A}_{K_\infty}^{[0, r_k]} &=& \mathbf{A}^+_{K_\infty} \{\frac{u^{ep^k}}{p}\}, \\
 \mathbf{A}_{K_\infty}^{[r_\ell, +\infty]} &=& \mathbf{A}^+_{K_\infty} \{ \frac{p}{u^{ep^\ell}}\}, \\
  \mathbf{A}_{K_\infty}^{[r_\ell, r_k]} &=& \mathbf{A}^+_{K_\infty} \{ \frac{p}{u^{ep^\ell}} , \frac{u^{ep^k}}{p}\}.
\end{eqnarray*}
\end{prop}
\begin{proof}
This easily follows from Lem. \ref{lem inj} and Lem. \ref{lem laurent series}.
\end{proof}
%Simply note that the right hand side of these equalities (after changing $u$ to $T$) correspond to those Laurent series in Lemma \ref{lem laurent series}.

\begin{corollary} \label{lem aa inter}
Suppose $[r, s]=[r_\ell, r_k] \subset [r', s]=[r_{\ell'}, r_k]$, then
$\mathbf{A}^{[r, s]}_{K_\infty}  \cap \wt{\mathbf{A}}^{[r', s]} = \mathbf{A}^{[r', s]}_{K_\infty}$.
\end{corollary}
\begin{proof}
Let $f \in \mathbf{A}^{[r, s]}_{K_\infty}  \cap \wt{\mathbf{A}}^{[r', s]}$. By Prop. \ref{cor A rl rk}, we can always write $f=f_1 +f_2$, where $f_1 \in \mathbf{A}^{[r, +\infty]}_{K_\infty}$ and $f_2 \in \mathbf{A}^{[0, s]}_{K_\infty}$; it then suffices to show that $f_1 \in \mathbf{A}^{[r', s]}_{K_\infty}$. But indeed we can show that $f_1 \in \mathbf{A}^{[r', +\infty]}_{K_\infty}$, using similar argument as in  \cite[Lem. II.2.2]{CC98}.
\end{proof}

\section{Locally analytic vectors of some rings} \label{sec: loc ana}
The main result in this section is to calculate locally analytic vectors in $(\wt{\mathbf{B}}^{I})^{G_\infty}=\wt{\mathbf{B}}^{I}_{K_\infty}$. Actually, there is no group action on  $(\wt{\mathbf{B}}^{I})^{G_\infty}$ since $G_\infty$ is not normal in $G_K$; what we do instead is to calculate locally analytic vectors in $\wt{\mathbf{B}}^{I}_L:=(\wt{\mathbf{B}})^{\Gal(\overline{K}/L)}$ (with respect to the $\Gal(L/K)$-action) that are \emph{furthermore $G_\infty$-invariant}.

%It turns out that these locally analytic vectors are precisely the union of $\varphi^{-m}(\mathbf{B}^{p^mI}_{K_\infty})$ for all $m \geq 0$.
%that is, the union of images of $R_m$ in Prop. \ref{prop Tate Sen} (see Theorem \ref{thm loc ana gamma 1} for precise statements).

\subsection{Theory of locally analytic vectors} \label{subseclocan}
Let us recall the theory of locally analytic vectors, see \cite[\S 2.1]{BC16} and \cite[\S 2]{Ber16} for more details.
Recall that a $\Qp$-Banach space $W$ is a $\Qp$-vector space with a complete non-Archimedean  norm $\|\cdot\|$ such that $\|aw\|=\|a\|_p\|w\|, \forall a \in \Qp, w \in W$, where $\|a\|_p$ is the usual $p$-adic norm on $\Qp$.
Recall the multi-index notations: if $\cbf = (c_1, \hdots,c_d)$ and $\kbf = (k_1,\hdots,k_d) \in \mathbb{N}^d$ (here $\mathbb{N}=\mathbb{Z}^{\geq 0}$), then we let $\cbf^\kbf = c_1^{k_1} \cdot \ldots \cdot c_d^{k_d}$.

\subsubsection{} \label{subsub def an}
Let $G$ be a $p$-adic Lie group, and let $(W, \|\cdot \|)$ be a $\Qp$-Banach representation of $G$.
Let $H$ be an open subgroup of $G$ such that there exist coordinates $c_1,\hdots,c_d : H \to \Zp$ giving rise to an analytic bijection $\cbf : H \to \Zp^d$.
 We say that an element $w \in W$ is an $H$-analytic vector if there exists a sequence $\{w_\kbf\}_{\kbf \in \mathbb{N}^d}$ with $w_\kbf \to 0$ in $W$, such that $$g(w) = \sum_{\kbf \in \mathbb{N}^d} \cbf(g)^\kbf w_\kbf, \quad \forall g \in H.$$
Let $W^{H\dan}$ denote the space of $H$-analytic vectors. $W^{H\dan}$ injects into $\mathcal{C}^{\an}(H, W)$ (the space of analytic functions on $H$ valued in $W$), and we endow it with the induced norm, which we denote as $\|\cdot\|_H$. We have $\|w\|_H=\sup_{\kbf \in \mathbb{N}^d}\|w_{\kbf}\|$, and $W^{H\dan}$ is a Banach space.

%(One can furthermore equip $W^{H\dan}$ with a natural norm and make it into a Banach space, see \cite[\S 2.1]{BC16}).

  We say that a vector $w \in W$ is \emph{locally analytic} if there exists an open subgroup $H$ as above such that $w \in W^{H\dan}$. Let $W^{G\dla}$ (or $W^{\la}$ when there is no confusion) denote the space of such vectors. We have $W^{\la} = \cup_{H} W^{H\dan}$ where $H$ runs through  open subgroups of $G$. We can endow $W^{\la}$ with the inductive limit topology, so that $W^{\la}$ is an LB space.

 \begin{lemma} \label{ringlocan}
Keep the notations as in \S \ref{subsub def an}.
If $W$ is furthermore a ring  such that $\|xy\| \leq \|x\| \cdot \|y\|$ for $x,y \in W$, then
\begin{enumerate}
  \item $W^{H\dan}$ is a ring, and $\|xy\|_H \leq \|x\|_H \cdot \|y\|_H$ if $x,y \in W^{H\dan}$.
  \item Suppose $w \in W^\times$ and $w \in W^{G\dla}$, then $1/w \in W^{G\dla}$. (In particular, if $W$ is a field, then  $W^{G\dla}$ is also a field.)
\end{enumerate}
\end{lemma}
\begin{proof}
Item (1) is \cite[Lem. 2.5(i)]{BC16}. Item (2) is stronger than \cite[Lem. 2.5(ii)]{BC16}, but this stronger statement is proved in \emph{loc. cit.}.
\end{proof}

\subsubsection{} \label{subsub def gn}
Keep the notations as in \S \ref{subsub def an}.
By the paragraph preceding \cite[Lem. 2.4]{BC16}, there exists some (not unique) open compact subgroup $G_1$ of $G$ such that there exist local coordinates $\wt{\cbf} : G_1 \to \Zp^d$, which furthermore satisfy $\wt{\cbf}(G_n) =(p^n \Zp)^d$ where $G_n:=G_1^{p^{n-1}}$. Then we have $W^{\la} = \cup_{n} W^{G_n\dan}$.

\begin{lemma}(\cite[Lem. 2.4]{BC16}) \label{lem 2.4BC}
Keep the notations as in \S \ref{subsub def gn}.
Suppose $w\in W^{G_n\dan}$, then for all $m \geq n$, $w\in W^{G_m\dan}$ and $\|w\|_{G_m} \leq \|w\|_{G_n}$. Furthermore, $\|w\|_{G_m} = \|w\|$ when $m \gg 0$.
\end{lemma}

\subsubsection{} \label{subsub def pa}
 Let $W$ be a Fr\'echet space, whose topology is defined by a sequence $\{ p_i \}_{i \geq 1}$  of seminorms. Let $W_i$ denote the Hausdorff completion of $W$ for $p_i$, so that $W = \projlim_{i \geq 1} W_i$.
If $W$ is a Fr\'echet representation of $G$, then a vector $w \in W$ is called \emph{pro-analytic} if its image $\pi_i(w)$ in $W_i$ is a locally analytic vector for all $i$. We denote by $W^{\mathrm{pa}}$ the set of such vectors. We can extend this definition to LF spaces (cf. \cite[\S 2]{Ber16}).
%Note that if $W$ is an LB space, then $W^{\la} = W^{\pa}$. If  $W$ is an LF  space, then $W^{\la} \subset W^{\pa}$ but $W^{\pa}$ will generally be bigger.

%This space injects into $\calC^{\an}(H,W)$ and we endow it with the induced norm, which we denote by $\|\cdot\|_H$. We have $\|w\|_H = \sup_{\kbf \in \NN^d} \| w_{\kbf}\|$, and $W^{H\dan}$ is a Banach space.

\begin{prop}
\label{prop la basis}
Let $G$ be a $p$-adic Lie group, let $\hat{B}$ be a Banach (resp. Fr\'echet) $G$-ring, and $B \subset \hat{B}$ a subring (but not necessarily $G$-stable).
Let $W$ be a free $B$-module of finite rank, let $\hat{W}:=\hat{B}\otimes_B W$,
and suppose there is a $\hat{B}$-semi-linear $G$-action on $\hat{W}$.
Let ${B}^{\la}:=B\cap \hat{B}^{\la}$ and $W^{\la}:=W\cap \hat{W}^{\la}$ (resp. $ {B}^{\pa}:=B\cap \hat{B}^{\pa}$ and $W^{\pa}:=W\cap \hat{W}^{\pa}$).

If $W$ has a $B$-basis $w_1,\hdots,w_d$ such that $g \mapsto \Mat(g)$ is a \emph{globally analytic} (resp. \emph{pro-analytic}) function $G \to \GL_d(\hat{B}) \subset \mathrm{M}_d(\hat{B})$, then
$$W^{\la} = \oplus_{j=1}^d  {B}^{\la} \cdot w_j \quad (\textnormal{resp. } W^{\pa} = \oplus_{j=1}^d {B}^{\pa} \cdot w_j).$$
\end{prop}
\begin{proof}
By \cite[Prop. 2.3]{BC16} (resp. \cite[Prop. 2.4]{Ber16}), we have $\hat{W}^{\la} = \oplus_{j=1}^d \hat{B}^{\la} \cdot w_j$ (resp. $\hat{W}^{\pa} = \oplus_{j=1}^d \hat{B}^{\pa} \cdot w_j$), then we can take intersections with $W$ to conclude.
\end{proof}

%Note that the notation ${B}^{\la}$ (and $W^{\la}$ etc.) in Prop. \ref{prop la basis} is a rather bad abuse of notation;

In the following, we give a useful criterion to determine analytic vectors for the $p$-adic Lie group $\Zp$.
\begin{lem}\label{subsub an}
Suppose $(W, \|\cdot \|)$ is a $\Qp$-Banach representation of $\Zp$. Let $\tau$ be a topological generator of $\Zp$, and let $\log \tau$ denote  the (formally written) series $(-1)\cdot \sum_{k \geq 1}  (1-\tau)^k/k$.
Given $x\ \in W$, then
$x \in W^{\Zp\dan}$ if and only if the following hold:
\begin{enumerate}
  \item  the series $(\log \tau)(x)$ converges in $W$, and inductively, $(\log\tau)^i(x):=(\log\tau)((\log\tau)^{i-1}(x))$ converges in $W$ for all $i \geq 2$;

  %$\frac{(\log\tau)^i(x)}{i!} \in W, \forall i \geq 0$, where $\log \tau$ denotes the operator $(-1)\cdot \sum_{k \geq 1}  (1-\tau)^k/k$, and $(\log\tau)^i(x)$ is defined by $\log\tau ((\log\tau)^{i-1}(x))$ inductively, or equivalently by the expansion of $((-1)\cdot \sum_{k \geq 1}  (1-\tau)^k/k)^i$ acting on $x$;

  \item $\|(\log \tau)^i(x)/i!\| \to 0$ as $i \to +\infty$;
  \item  for all $a \in \Zp$,
  \begin{equation}\label{eqan}
    \tau^a(x) =\sum_{i=0}^{\infty} a^i\cdot\frac{(\log \tau)^i(x)}{i!}.
  \end{equation}
\end{enumerate}
If the above holds, then $\log \tau(x) \in W^{\Zp\dan}$, and for all $a \in \Zp$,
we have $(\log\tau^a)(x)=a\cdot \log \tau(x)$.
\end{lem}
\begin{proof}
This is standard, cf. \cite[\S 3]{ST02}.
\end{proof}

\begin{lemma} \label{lem Zp an}
Suppose $(W, \|\cdot \|)$ is a $\Qp$-Banach representation of $\Zp$ such that $\|g(w)\|=\| w\|, \forall g\in \Zp, w\in W$ (i.e., $\|\cdot \|$ is an \emph{invariant} norm). Let $x \in W$. Let $\tau$ be a topological generator of $\Zp$. If there exists some $r < \inf \{1/e, p^{-1/(p-1)}\}$ (here $e$ is Euler's number $2.718\ldots$), some $R>0$ and $k_0 \in \mathbb{Z}^{\geq 0}$, such that
\begin{eqnarray}
\label{tauR} \|  (1-\tau^{a})^k (x)\|  &\leq& R, \quad \text{ for all } a \in \Zp, k < k_0;  \\
\label{taur} \|  (1-\tau^{a})^k (x)\|  &\leq& r^k, \quad \text{ for all } a \in \Zp, k \geq k_0,
\end{eqnarray}
then $x \in W^{\Zp\dan}$.
\end{lemma}
\begin{proof}
\textbf{Step 0}: \emph{Partial log}.
Let $A$ be a $\Qp$-algebra.
Given $a\in A$, denote
$$\log_m a : = \sum_{i=1}^{p^m-1} \frac{(1-a)^i} i \in A.$$
If $A$ is furthermore a Banach algebra, and $\| \frac{(1-a)^i} i\| \to 0$ when $i \to +\infty$, then we denote $\log a:=(-1)\cdot \sum_{i=1}^{+\infty} \frac{(1-a)^i} i$ (and say $\log a$ is well-defined).
Suppose $a, b \in A$ such that $ab=ba$, then we have the identity:
$$\frac{(1-ab)^i} i =
\frac{(1-a)^i} i + \sum_{j=1}^i \binom {i-1}{j-1} \cdot a^j (1-a)^{i-j}
\cdot \frac{(1-b)^j} j.$$
So we have (cf. \cite[Eqn. (3.4)]{Car13}):
$$\log_m(ab) =  \log_m a + \sum_{j=1}^{p^m-1} \Bigg(
a^j \cdot \sum_{i=j}^{p^m-1} \binom {i-1}{j-1} \cdot (1-a)^{i-j} \Bigg)
\cdot \frac{(1-b)^j} j.$$
Note that (cf. the equation below \cite[Eqn. (3.4)]{Car13})
$$(1-X)^j \cdot \sum_{i=j}^{p^m-1} \binom{i-1}{j-1} X^{i-j}
\in 1 + X^{p^m-j}\Zp[X] .$$
Apply the above identity with $X=1-a$, then we get
\begin{equation}\label{logab}
  \log_m(ab) -\log_m a-\log_m b = \sum_{j=1}^{p^m-1} f_j(1-a)\cdot (1-a)^{p^m-j}\cdot \frac{(1-b)^j}{j},
\end{equation}
where $f_j(X) \in \Zp[X]$ are some polynomials.

\textbf{Step 1}: \emph{Logarithm of $x$}. Using condition \eqref{tauR} and \eqref{taur}, it is clear that for any $a \in \Zp$, $(\log \tau^a)(x)$ is well-defined. Furthermore, there exists some $r'>0$, such that
\begin{equation}\label{log norm}
 \|(\log \tau^a)(x) \| <r', \quad \forall a \in \Zp.
\end{equation}
We claim that
\begin{equation}\label{loga}
(\log \tau^a )(x) =a \cdot (\log  \tau )(x), \quad \forall a \in \Zp.
\end{equation}
To prove \eqref{loga}, we first show that
\begin{equation}\label{logalbe}
(\log  \tau^{\alpha+\beta} )(x) =(\log \tau^{\alpha})(x) +(\log \tau^{\beta})(x), \quad \forall \alpha, \beta \in \Zp.
\end{equation}
Using \eqref{logab}, we have
\begin{equation}\label{eq11}
 (\log_m \tau^{\alpha+\beta})(x) - (\log_m \tau^{\alpha})(x)-(\log_m \tau^{\beta})(x)
=
\sum_{j=1}^{p^m-1} f_j(1-\tau^\alpha)\cdot (1-\tau^\alpha)^{p^m-j}\cdot \frac{(1-\tau^\beta)^j}{j}(x).
\end{equation}
Since $\|\cdot \|$ is an invariant norm, it is easy to see that
\begin{equation}\label{eqinva}
\|(f(\tau))(w)\| \leq \|w \|, \quad  \forall w \in W, f(X) \in \Zp[X] \text{ a polynomial}.
\end{equation}
When $p^m /2 \geq k_0$ (so $\max \{j, p^m-j\} \geq k_0, \forall j$), the norm of the right hand side of \eqref{eq11} is bounded by $p^m r^{p^m /2}$ (using \eqref{taur} and \eqref{eqinva}).
Let $m \to +\infty$, and so \eqref{logalbe} is proved.
Now given $a \in \Zp$, let $a=a_m+p^mb_m$ where $a_m \in \mathbb{Z}, b_m \in \Zp$. By \eqref{logalbe},
$$(\log \tau^a )(x) = (\log \tau^{a_m} )(x) + (\log \tau^{p^mb_m} )(x) = a_m\cdot (\log  \tau )(x) +p^m\cdot (\log  \tau^{b_m} )(x) .$$
Use \eqref{log norm}, and let $m \to +\infty$, we can conclude \eqref{loga}.

\textbf{Step 2}: \emph{General term of a summation}. Consider the summation $\sum_{k=0}^{\infty} \frac{(\log \tau^a)^k(x)}{k!}$ where $a \in \Zp$, then its ``general term" is of the form:
$$ \frac{1}{k!}\frac{(1-\tau^a)^{i_1+ \cdots +i_k}(x)}{i_1\cdot \cdots \cdot i_k}, \text{ where } i_j \geq 1.$$
Suppose $\sum i_j=n$, then $n \geq k$. Let
$$r_k := \sup_{n \geq k}  \left\{ \|  \frac{1}{k!}\frac{(1-\tau^a)^{n}(x)}{i_1\cdots i_k} \|, \text{ where } \sum i_j=n    \right\}.$$
Note that we have
$$\|  \frac{1}{k!}\frac{(1-\tau^a)^{n}(x)}{i_1\cdots i_k} \|  \leq r^n\cdot p^{\frac{k}{p-1}} \cdot (\frac{n}{k})^k , \text{ when } n\geq k_0.$$
Fix a $k$, consider the function $f(X) =r^X \cdot X^k$ with $X \geq k$. Its logarithm is $X\ln r + k\ln X$, which has derivative $\ln r + k/X <0$ since $r<1/e$. Thus we conclude that
  $$\|  \frac{1}{k!}\frac{(1-\tau^a)^{n}(x)}{i_1\cdots i_k} \|  \leq r^k\cdot p^{\frac{k}{p-1}} \cdot (\frac{k}{k})^k = (rp^{\frac{1}{p-1}})^k, \text{ when }  n\geq k_0.$$
This implies that $r_k< +\infty, \forall k$. Furthermore,
$$r_k \leq   (rp^{\frac{1}{p-1}})^k, \text{ when } k \geq k_0,$$
and so $\lim_k r_k \to 0$ since $r < p^{-\frac{1}{p-1}}$.
This implies that the summation $\sum_{k=0}^{\infty} \frac{(\log \tau^a)^k(x)}{k!}$ converges absolutely.

\textbf{Step 3}: \emph{Conclusion}. Using Step 2 and \eqref{loga} in Step 1, it is easy to show that all the itemized conditions in Lem. \ref{subsub an} are satisfied; in particular, the equality \eqref{eqan} holds because by Step 2, we can ``re-arrange" the order of the summation.
Thus $x \in W^{\Zp\dan}$.
\end{proof}

\begin{Notation} \label{nota Zp banach}
If $(W, \|\cdot \|)$ is a $p$-adically separated and complete normed $\Zp$-module such that
$\|aw\|=\|a\|_p\|w\|$ for all $a \in \Zp$ and $w \in W$, and such that $W[1/p]$ (with the naturally induced norm) is a $\Qp$-Banach space, then we say $(W, \|\cdot \|)$ is a $\Zp$-Banach space for brevity. If furthermore such $W$ carries a continuous action by a $p$-adic Lie group $G$, then we denote $W^{G\dla}:=(W[1/p])^{G\dla} \cap W$.
\end{Notation}
%All the representations of $p$-adic Lie groups in this paper will be either $\Zp$-Banach spaces or $\Qp$-Banach spaces.

\subsection{Locally analytic representations of $\hat{G}$}
Let $\hat{G}=\gal(L/K)$ be as in Notation \ref{nota fields}. In this subsection, we mainly set up some notations with respect to
representations of $\hat{G}$.

\begin{Notation} \label{nota hatG}
\begin{enumerate}
\item Recall that:
\begin{itemize}
\item if $K_{\infty} \cap K_{p^\infty}=K$, then $\gal(L/K_{p^\infty})$ and $\gal(L/K_{\infty})$ topologically generate $\hat{G}$ (cf. \cite[Lem. 5.1.2]{Liu08});
\item if $K_{\infty} \cap K_{p^\infty} \supsetneq K$, then necessarily $p=2$, and  $\gal(L/K_{p^\infty})$ and $\gal(L/K_{\infty})$ topologically generate an open subgroup (denoted as $\hat{G}'$) of $\hat{G}$ of index $2$ (cf. \cite[Prop. 4.1.5]{Liu10}).
\end{itemize}
\item Note that:
\begin{itemize}
\item $\gal(L/K_{p^\infty}) \simeq \Zp$, and let
$\tau \in \gal(L/K_{p^\infty})$ be a topological generator;
\item $\gal(L/K_{\infty})$ ($\subset \gal(K_{p^\infty}/K) \subset \Zp^\times$) is not necessarily pro-cyclic when $p=2$.
\end{itemize}
If we let $\Delta \subset \gal(L/K_{\infty})$ be the torsion subgroup, then $\gal(L/K_{\infty})/\Delta$  is pro-cyclic; choose $\gamma' \in \gal(L/K_{\infty})$ such that its image in $\gal(L/K_{\infty})/\Delta$ is a topological generator.

\item Let $\tau_n: =\tau^{p^n}$ and $\gamma'_n: =(\gamma')^{p^n}$.
Let $\hat{G}_n \subset \hat{G}$ be the subgroup topologically generated by $\tau_n$ and $\gamma'_n$. These $\hat{G}_n$ satisfy the property  in \S \ref{subsub def gn}.
\end{enumerate}
\end{Notation}

\begin{Notation} \label{nota tau la}
\begin{enumerate}
\item Given a $\hat{G}$-representation $W$, we use
$$W^{\tau=1}, \quad W^{\gamma=1}$$
to mean $$ W^{\gal(L/K_{p^\infty})=1}, \quad
W^{\gal(L/K_{\infty})=1}.$$
And we use
$$
W^{\tau\dla}, \quad W^{\tau_n\dan}, \quad  W^{\gamma\dla} $$
to mean
$$
W^{\gal(L/K_{p^\infty})\dla}, \quad
W^{\gal(L/(K_{p^\infty}(\pi_n)))\dla}, \quad
W^{\gal(L/K_{\infty})\dla}.  $$

% $\gal(L/K_{p^\infty})$-locally analytic vectors (resp. $\gal(L/(K_{p^\infty}(\pi_n)))$-analytic vectors, resp. $\gal(L/K_{\infty})$-locally analytic vectors).

\item  Let
$$ \nabla_\tau  := \frac{\log \tau^{p^n}}{p^n} \text{ for } n \gg0, \quad \nabla_\gamma:=\frac{\log g}{\log_p \chi_p(g)} \text{  for } g \in \gal(L/K_\infty) \text{ close enough to } 1  $$
be the two differential operators (acting on $\hat{G}$-locally analytic representations).
\end{enumerate}
\end{Notation}
\begin{rem}
Note that we never define $\gamma$ to be an element of $\gal(L/K_\infty)$; although when $p>2$ (or in general, when $\gal(L/K_\infty)$ is pro-cyclic), we could have defined it as a topological generator of $\gal(L/K_\infty)$. In particular, although ``$\gamma=1$" might be slightly ambiguous (but only when $p=2$), we use the notation for brevity.
\end{rem}

\begin{lemma} \label{lem taugamma}
Let
$W^{\tau\dla, \gamma=1}:= W^{\tau\dla} \cap W^{\gamma=1},$
then $$ W^{\tau\dla, \gamma=1} \subset  W^{\hat{G}\dla}. $$
\end{lemma}
\begin{proof}
This can be deduced from the fact that any element $g \in \hat{G}$ (or $g \in \hat{G}'$ when $K_\infty \cap K_{p^\infty} \neq K$, cf. Notation \ref{nota hatG}) can be uniquely written as a product $g_1g_2$ for some $g_1 \in \gal(L/K_\infty), g_2 \in \gal(L/K_{p^\infty})$.
\end{proof}

\begin{remark}
\begin{enumerate}[leftmargin=*]
\item  Let $W^{\gamma\dla, \tau=1}:=W^{\gamma\dla} \cap W^{\tau=1}$, then
$$W^{\gamma\dla, \tau=1}=\left((W)^{\gal(L/K_{p^\infty})}\right)^{\gal(K_{p^\infty}/K)\dla}  \subset  W^{\hat{G}\dla}$$
because $\gal(L/K_{p^\infty})$ is normal in $\hat{G}$.
\item We do not know if the inclusion $W^{\hat{G}\dla} \subset W^{\gamma\dla} \cap W^{\tau\dla}$ is an equality (very probably not, see next item).
\item We thank Laurent Berger for informing us of the following example. Let $G_1=G_2=\Zp$, and let $G=G_1\times G_2$. Let $W$ be the space of continuous $\Qp$-valued functions on $G$ with the action of $G$ by translations.
Let $f(x,y) = 0$ if $(x,y)=0$ and $f(x,y)=(x^2\cdot y^2)/(x^2+p y^2)$ otherwise. Then $f \in W^{G_1\dla} \cap W^{G_2\dla}$, but $f\notin W^{G\dla}$. (Note that by Hartog's theorem, the analogous phenomenon does not happen over the usual complex numbers).
\end{enumerate}
\end{remark}

\subsection{Locally analytic vectors in $\hat{L}$}
Let $\hat{L}$ be the $p$-adic completion of $L$ (cf. Notation \ref{nota fields}).
As in \cite[\S 4.4]{BC16},
consider the 2-dimensional $\Qp$-representation of $G_K$ (associated to our choice of $\{\pi_n\}_{n \geq 0}$) such that $g \mapsto \smat{\chi(g) & c(g) \\ 0 & 1}$ where $\chi$ is the $p$-adic cyclotomic character. Since the co-cycle $c(g)$ becomes trivial over $C_p$, there exists $\alpha \in C_p$ (indeed, $\alpha \in \hat{L}$) such that $c(g) = g(\alpha)\chi(g)-\alpha$.
This implies  $g(\alpha) = \alpha/\chi(g) + c(g)/\chi(g)$ and so $\alpha \in \hat{L}^{\hat{G}\dla}$.
Now similarly as in the beginning of \cite[\S 4.2]{BC16}, let $\alpha_n \in L$ such that $\|\alpha-\alpha_n\|_p \leq p^{-n}$. Then there exists $r(n) \gg0$ such that if $m \geq r(n)$, then $\|\alpha-\alpha_n\|_{\hat{G}_m}= \|\alpha-\alpha_n\|_p$ and $\alpha-\alpha_n \in \hat{L}^{\hat{G}_m\dan}$ (see Notation \ref{nota hatG} for $\hat{G}_m$). We can furthermore suppose that $\{r(n)\}_n$ is an increasing sequence.

%Note that $\gamma(\alpha)=\frac{\alpha}{\chi(\gamma)}$ and $\tau(\alpha)=\alpha +1$. Then we have $\nabla_\gamma(\alpha)=-\alpha$ and $\nabla_\tau(\alpha)=1$.
%Define $$\partial_\gamma: =\frac{\nabla_\gamma}{-\alpha}.$$
%Note that $\hat{L}^{\hat{G}\dla}$ is a field, so $\partial_\gamma$ is a well-defined operator on it.

\begin{defn}
Let $(H, \|\cdot \|)$ be a $\Qp$-Banach algebra such that $\|\cdot \|$ is sub-multiplicative, and let $W \subset H$ be a $\Qp$-subalgebra. Let $T$ be a variable, and let  $W \dacc{T}_n$ be the vector space consisting of $\sum_{k \geq 0} a_k T^k$ with $a_k \in W$, and $p^{nk} a_k \to 0$ when $k \to +\infty$. For $h \in H$ such that $\|h \|\leq p^{-n}$, denote $W \dacc{h}_n$ the image of the evaluation map $W \dacc{T}_n \to H$ where $T \mapsto h$.
\end{defn}

\begin{prop}
\label{loc ana in L}
\begin{enumerate}
\item $\hat{L}^{\hat{G}\dla} =\cup_{n \geq 1} K({\mu_{r(n)}, \pi_{r(n)}})\dacc{ \alpha-\alpha_n }_n. $
\item $\hat{L}^{\hat{G}\dla, \nabla_\gamma=0} = L.$
\item $\hat{L}^{\tau\dla, \gamma=1} = {K_{\infty}}.$
\end{enumerate}
\end{prop}
\begin{proof}
Item (1) is \cite[Prop. 4.12]{BC16}; we quickly recall the proof here. Suppose $x\in \hat{L}^{\hat{G}_n\dan}$. For $i \geq 0$, let
$$y_i = \sum_{k \geq 0} (-1)^k (\alpha - \alpha_n)^k \nabla_\tau^{k+i}(x) \binom{k+i}{k},$$
then there exists $m\geq n$ such that $y_i \in \hat{L}^{\hat{G}_m\dan}$, and
$x = \sum_{i \geq 0} y_i (\alpha - \alpha_n)^i$ in $\hat{L}^{\hat{G}_m\dan}$.
Then the fact $\nabla_\tau(y_i)=0$ will imply that $y_i\in K(\mu_m, \pi_m)$, concluding (1).

Consider Item (2). By \cite[Prop. 6.3]{BC16}, there exists a non-zero element $\beta \in C_p \otimes \textnormal{Lie } \hat{G}$ such that $\beta=0$ on $\hat{L}^{\hat{G}\dla}$. Write $\beta=a\nabla_\tau +b \nabla_\gamma$ with $a, b \in C_p$. We have $a \neq 0$ since $\nabla_\gamma \neq 0$ on $\Kpinfty$; similarly $b \neq 0$. Thus, the condition $\nabla_\gamma=0$ in Item (2) implies $\nabla_\tau=0$, and so $y_i =0$ for $i \geq 1$, concluding (2).

Item (3) easily follows from (2).
\end{proof}
  % as two operators in $C_p \otimes \textnormal{Lie } \hat{G}$ acting on $\hat{L}^{\hat{G}\dla}$,

\subsection{Locally analytic vectors in $\wt{\mathbf{B}}^I_{K_\infty}$}

\begin{lemma} \label{lem div p}
Suppose $I=[r_\ell, r_k]$ or $[0, r_k]$.
  \begin{enumerate}
  \item $\wt{\mathbf{A}}^{[0, r_k]} =\wt{\mathbf{A}}^+ \{ \frac{\varphi^k(E(u))}{p}  \}$.
    \item $p\wt{\mathbf{A}}^{I} \cap \frac{\varphi^k(E(u))}{p}\wt{\mathbf{A}}^{I} =\varphi^k(E(u))\wt{\mathbf{A}}^{I} $.
    \item $p\wt{\mathbf{A}}^{I} \cap \wt{\mathbf{A}}^{[0, r_k]} =p\wt{\mathbf{A}}^{[0, r_k]}$.
    \item If $y \in   \wt{\mathbf{A}}^{[0, r_k]}+p\wt{\mathbf{A}}^{I} $ and $y_i \in \wt{\mathbf{A}}^+$ such that $y-\sum_{i=0}^{j-1}y_i(\frac{\varphi^k(E(u))}{p})^i$ is in $(\Ker (\theta\circ \iota_k))^j$ for all $j \geq 1$. Then there exists some $j \geq 1$ such that  $y-\sum_{i=0}^{j-1}y_i(\frac{\varphi^k(E(u))}{p})^i \in p\wt{\mathbf{A}}^{I}$.

  \end{enumerate}
\end{lemma}
\begin{proof}
These are easy analogues of \cite[Lem. 3.1, Lem. 3.2, Prop. 3.3]{Ber16}; let us sketch the proofs for the reader's convenience.

Item (1) easily follows from Def. \ref{defn wt rings} (or see \cite[Lem. 3.1]{Ber16} for a quick development).

For Item (2), suppose $px$ belongs to left hand side, then $px$ and hence $x$ belongs to the kernel of $\theta\circ \iota_k: \wt{\mathbf{A}}^{I} \to C_p$; one then concludes by Lem. \ref{prop theta ker}(1).

Item (3) is vacuous when $I=[0, r_k]$. When $I=[r_\ell, r_k]$, this is \cite[Lem. 3.2(3)]{Ber16} (or our Eqn.  \eqref{eqptimes}).

% (Let us mention that the proof in \cite[Lem. 3.2(3)]{Ber16} is erroneous; in his notation, ``$V(x, s) \geq 0$" can only imply $x \in \wt{\mathbf{A}}^{[s, s]}$).

%Here, to show the last equality above, simply write the element $px^{-} \in p\wt{\mathbf{A}}^{[r_\ell, +\infty]} \subset \wt{\mathbf{A}}$ as $\sum_{i \geq 1}p^i[x_i]$ with $x_i \in \wt{\mathbf{E}}$, then it is clear $x_i \in \wt{\mathbf{E}}^+$ and we can conclude.

    Consider Item (4). By Item (1), there exists some $j \geq 1$ and some $a_i \in \wt{\mathbf{A}}^+$ such that
\begin{equation}\label{eq div p}
y-\sum_{i=0}^{j-1}a_i (\frac{\varphi^k(E(u))}{p})^i \in p\wt{\mathbf{A}}^{I}
\end{equation}
     (note that this is possible for either $I=[r_\ell, r_k]$ or $[0, r_k]$). One then proceeds as in \cite[Prop. 3.3]{Ber16}, by changing all the $Q_k$ (resp.  $\pi$, resp. $[r, s]$) in \emph{loc. cit.} to $\varphi^k(E(u))$ (resp. $p$, resp.
     $I$), to show that one can replace the $a_i$ above by $y_i$ without changing the property in \eqref{eq div p}.

\end{proof}
%(1) is an easy analogues of \cite[Lem. 3.1]{Ber16}. When $I=[r_\ell, r_k]$, (2)-(4) are easy analogues of \cite[Lem. 3.2, Prop. 3.3]{Ber16}; one can prove them by simply changing $Q_k$ (resp.  $\pi$ ) in \emph{loc. cit.} to $\varphi^k(E(u))$ (resp. $p$). When $I=[0, r_k]$, (2) is easy an analogue of \cite[Lem. 3.2(2)]{Ber16}, (3) is vacuous, and (4) easily follows from (1).

For $I$ a closed interval, note that $(\wt{\mathbf{B}}^{I}_{ L}, W^I)$ is a $\Qp$-Banach representation of $\hat{G}$ (in particular, note that $W^I(p)=1$); also note that the valuation $W^I$ is invariant under the Galois action.

\begin{lemma} \label{lem div ana}
Suppose $I=[r_\ell, r_k]$ or $[0, r_k]$.
\begin{enumerate}
\item \label{item-varphiu}
 For each $n \geq 0$,  $\varphi^{-n}(u) \in (\wt{\mathbf{B}}^{I}_L)^{\tau_{n+k}\dan}$. Thus:
  $$\varphi^{-n}(u) \in (\wt{\mathbf{B}}^{I}_L)^{\tau_{n+k}\dan, \gamma=1} \subset (\wt{\mathbf{B}}^{I}_L)^{\hat{G}\dla}  . $$

  \item \label{item-m0}
   There exists $m_0 \geq 0$ (depending on $k$ only) such that $$ \frac{t}{\varphi^k(E(u))} \in (\wt{\mathbf{B}}^{I}_{ L})^{\tau_{m_0}\dan}.$$

  \item \label{item-divide-t}
  Suppose $x \in \wt{\mathbf{B}}^{I}_L$ such that $tx \in (\wt{\mathbf{B}}^{I}_L)^{\tau_{n}\dan}$ for some $n \geq 0$, then $x \in (\wt{\mathbf{B}}^{I}_L)^{\tau_{n}\dan}$.

  \item \label{item-divide-phie}
  Suppose $m \geq m_0$. Then
  $$ (\wt{\mathbf{B}}^{I}_{ L})^{\tau_m\dan, \gamma=1} \cap \varphi^k(E(u))\wt{\mathbf{B}}^{I}_{ L} =  \varphi^k(E(u))(\wt{\mathbf{B}}^{I}_{ L})^{\tau_m\dan, \gamma=1}.$$
\end{enumerate}
\end{lemma}
\begin{proof}
The proof of Item (\ref{item-varphiu}) follows similar ideas as in \cite[Prop. 4.1]{Ber16}.
Let us mention that it is relatively easy to show that $\varphi^{-n}(u)$ is \emph{locally} analytic, e.g., using \eqref{eqhx} below; however it is critical to control the radius of analyticity (which is $p^{-(n+k)}$ in this case) for later application in Thm. \ref{thm loc ana gamma 1}.
 Write $v$ for $[\underline{\varepsilon}]-1 \in \wt{\mathbf{A}}^+$.
  For $a \in \Zp$, we have
$$\tau^a(\varphi^{-n}(u))  = \varphi^{-n}(u\cdot(1+v)^a)= \varphi^{-n}(u)\cdot \left(\sum_{m=0}^\infty\binom{a}{m}\varphi^{-n}(v)^m\right).$$
It suffices to show that the (formally written) summation function (from  $\Zp$ to $ \wt{\mathbf{B}}_L^I$)
\begin{equation}\label{eqlah}
T \mapsto \sum_{m \geq 0}\binom{T}{m}\cdot \varphi^{-n}(v)^m
\end{equation}
is (well-defined and) analytic on the closed disk (around $0$) of radius $p^{-h}$ where $h=n+k$.
By \cite[Thm. I.4.7]{Col10une}) (due to Amice), the polynomials
 $\lfloor {m}/{p^h}\rfloor!\binom{T}{m}$ for $m \geq 0$ form an orthonormal basis of $\mathrm{LA}_h(\Zp,
 \Qp )$, where $\mathrm{LA}_h(\Zp,
 \Qp )$ is the Banach space of functions on $\Z_p$ that are analytic on all the closed sub-disks of radius $p^{-h}$ (cf. the definition above \cite[Rem. I.4.4]{Col10une}). See \cite[Def. I.1.3]{Col10une} for the definition of an orthonormal basis; in particular, it implies that the norm of $\lfloor {m}/{p^h}\rfloor!\binom{T}{m}$ on the closed disk (around $0$) of radius $p^{-h}$ is  $\leq 1$.
Note that since $\varphi^{-n}(v) \in \wt{\mathbf{A}}^+$,
$$W^I(\varphi^{-n}(v))=   W^{[r_k, r_k]}(\varphi^{-n}(v)) =\frac{1}{(p-1)p^{n+k-1}}. $$
Thus, the norm of the term $\binom{T}{m}\cdot \varphi^{-n}(v)^m$ on the closed disk of radius $p^{-h}$ is
\begin{equation*}
\leq \|\binom{T}{m}\|_{\mathrm{LA}_h(\Zp,\Qp)}\cdot
p^{  W^I(\varphi^{-n}(v)^m)} =
p^{v_p(\lfloor {m}/{p^h}\rfloor!)} \cdot
 p^{-\frac{m}{(p-1)p^{n+k-1}}} \leq p^{-\frac{m}{p^h}}.
\end{equation*}
Thus $\binom{T}{m}\cdot \varphi^{-n}(v)^m$ converges to 0 and the analyticity of \eqref{eqlah} is verified.

Consider Item (\ref{item-m0}). Denote $F:=\varphi^k(E(u))$. Since $F$ is a generator of $\Ker (\theta\circ \iota_k:\wt{\mathbf{B}}^{I} \to C_p )$, we have $\frac{t}{F} \in \wt{\mathbf{B}}_{L}^I$.
Let $m_0 \gg0$ such that when $a \in p^{m_0}\Zp$,
\begin{equation}\label{eqhx}
  (1-\tau^a)(u)=u(1-[\underline{\varepsilon}]^a) =u\cdot p^\theta t \cdot h(p^\theta t), \quad \text{ for some } \theta >0, h(X) \in \Zp[\![X]\!].
\end{equation}
By increasing $m_0$ if needed, we can further assume that
\begin{equation}\label{eqhxx}
W^I(p^\theta \cdot \frac{t}{F})=\alpha>0.
\end{equation}
We claim that for all $a \in p^{m_0}\Zp$, there exists $ f_s(X, Y) \in W(k)[\![X, Y]\!] $ (depending on $a$), such that
\begin{equation}\label{tautEu}
  (1-\tau^a)^s(\frac{t}{F}) =\frac{t(p^\theta t)^{s}\cdot f_{s}(u, p^\theta t)}{\prod_{i=0}^{s} \tau^{ai}(F)}, \quad \forall s \geq 0.
\end{equation}
When $s=0$, simply let $f_0=1$. Suppose \eqref{tautEu} is valid for $s-1$, then
$$(1-\tau^a)^s(\frac{t}{F}) = t(p^\theta t)^{s-1} \cdot \frac{ \tau^{as}(F)\cdot f_{s-1} -F \cdot \tau^a(f_{s-1}) }{\prod_{i=0}^{s} \tau^{ai}(F)}.    $$
Note that
$$\tau^{as}(F)\cdot f_{s-1} -F \cdot \tau^a(f_{s-1}) = (\tau^{as}-1)(F)\cdot f_{s-1} -F\cdot (\tau^a-1)(f_{s-1}). $$
Note that for any $i, j \geq 0$,
 $$(\tau^b -1)(u^i (p^\theta t)^j) =p^\theta t\cdot P_{i, j}(u, p^\theta t), \text{ with } P_{i, j} \in W(k)[\![X, Y]\!].$$
Thus it is easy to see that $(\tau^{as}-1)(F)=p^\theta t \cdot G(u, p^\theta t)$  and $(\tau^a-1)(f_{s-1})=p^\theta t  \cdot H(u, p^\theta t)$ with some $G, H \in W(k)[\![X, Y]\!] $, so we can simply let
$$f_s: = \frac{  \tau^{as}(F)\cdot f_{s-1} -F \cdot \tau^a(f_{s-1}) }{p^\theta t}, $$
concluding the proof of \eqref{tautEu}.
By \eqref{tautEu}, we have
\begin{equation}\label{WItauafrac}
W^I((1-\tau^a)^s(\frac{t}{F})) \geq W^I(p^{-\theta}\cdot  (\frac{p^\theta t}{F})^{s+1} ) \geq  -\theta +(s+1)\alpha.
\end{equation}
Thus it is easy to see that for the group generated by $p^{m_0}\tau$ ($\simeq \Zp$), the conditions \eqref{tauR} and \eqref{taur} in Lem. \ref{lem Zp an} are satisfied (if needed, we can increase $m_0$ to increase $\alpha$), and we can conclude Item (\ref{item-m0}).

For Item (\ref{item-divide-t}), one can assume that $n=0$ (the general case is similar). Write $I=[r, s]$. Since $W^{I} =\inf\{ W^{[r, r]}, W^{[s, s]}\}$ (or $W^{I} = W^{[s, s]} $ if $r=0$), and both $W^{[r, r]}$ and  $W^{[s, s]}$ are \emph{multiplicative} valuations, it is easy to see that there exists a constant $c(I)>0$ depending on $I$ only, such that
$$W^{I}(y) \geq W^{I}(ty)-c(I), \quad \forall  y \in \wt{\mathbf{B}}^{I}_L.$$
Using this, and the fact that  $(1-\tau^a)(tx)=t\cdot (1-\tau^a)(x)$, it is easy to see that if $tx$ satisfies the itemized conditions in Lem. \ref{subsub an}, then so does $x$.

For Item (\ref{item-divide-phie}), suppose $y \in \wt{\mathbf{B}}^{I}_{ L}$ such that $\varphi^k(E(u)) \cdot y \in (\wt{\mathbf{B}}^{I}_{ L})^{\tau_m\dan}$, it suffices to show that $y \in (\wt{\mathbf{B}}^{I}_{ L})^{\tau_m\dan}$. By Item (\ref{item-m0}), $\frac{t}{\varphi^k(E(u))} \cdot \varphi^k(E(u)) \cdot y =ty$ is an analytic vector, and we can conclude by Item (\ref{item-divide-t}).
\end{proof}

\begin{comment}
Consider Item (2). For $a \in p^{m_0+n}\Zp$, from \eqref{eqhx}, we have
$$ (1-\tau^a)(\varphi^{-n}(u))=\varphi^{-n}(u)\cdot(1-[\underline{\varepsilon}]^{a/p^n}) =\varphi^{-n}(u)\cdot p^\theta t \cdot h(p^\theta t). $$
So for any $s \geq 1$, we have
$$ (1-\tau^a)^s(\varphi^{-n}(u))=\varphi^{-n}(u)\cdot((1-[\underline{\varepsilon}]^{a/p^n}))^s =\varphi^{-n}(u)\cdot(p^\theta t \cdot h(p^\theta t))^s.$$
Thus we have
\begin{equation}\label{WIuu}
W^I((1-\tau^a)^s(\varphi^{-n}(u))) \geq s\alpha,
\end{equation}
and so we can conclude using Lem. \ref{lem Zp an}.

\end{comment}

\begin{definition} \label{def m inf}
Define
$$\mathbf{A}^I_{K_\infty, m}: =\varphi^{-m}(\mathbf{A}^{p^mI}_{K_\infty}), \quad
\mathbf{A}^I_{K_\infty, \infty}: =\cup_{m \geq 0}\mathbf{A}^I_{K_\infty, m}.
$$
Define $\mathbf{B}^I_{K_\infty, m}$ and $\mathbf{B}^I_{K_\infty, \infty}$ similarly.
\end{definition}

\begin{theorem} \label{thm loc ana gamma 1}
Suppose $I=[r_\ell, r_k]$ or $[0, r_k]$. Let $m_0$ be as in Lem. \ref{lem div ana}.
\begin{enumerate}
\item \label{itemmkan}
$(\wt{\mathbf{A}}^I_{L})^{\tau_{m+k}\dan, \gamma=1}  \subset \mathbf{A}^I_{K_\infty, m}$ for any $m \geq m_0$.

\item \label{itemkinf}
$(\wt{\mathbf{A}}^I_{ L})^{\tau\dla, \gamma=1} = \mathbf{A}^I_{K_\infty, \infty}.$
\item $(\wt{\mathbf{B}}_{L}^{[r_\ell, +\infty)})^{\tau\dpa, \gamma=1} =    \mathbf{B}_{ K_\infty, \infty}^{[r_\ell, +\infty)}.$
\item   $(\wt{\mathbf{B}}_{L}^{[0, +\infty)})^{\tau\dpa, \gamma=1} =  \mathbf{B}_{ K_\infty, \infty}^{[0, +\infty)}.$
\end{enumerate}
\end{theorem}

\begin{proof} The proof of Item (\ref{itemmkan}) follows the same strategy as in \cite[Thm. 4.4]{Ber16}. (Some error of \emph{loc. cit.} is corrected in the errata, posted on Berger's homepage.)
Suppose $x \in (\wt{\mathbf{A}}^I_{L})^{\tau_{m+k}\dan, \gamma=1} $.
\begin{itemize}
\item When $I=[0, r_k]$, for each $n \geq 0$, we let $k_n=0$, and let
$$x_n : = (\frac{u^{ep^k}}{p})^{k_n} x =x \in \wt{\mathbf{A}}^{[0, r_k]}= \wt{\mathbf{A}}^{[0, r_k]}+p^n\wt{\mathbf{A}}^{I}.$$

\item When $I=[r_\ell, r_k]$, note that
$\wt{\mathbf{A}}^I =\wt{\mathbf{A}}^+\{ \frac{p}{u^{ep^\ell}}, \frac{u^{ep^k}}{p} \}$ and note that $k \geq \ell$. Thus for each $n \geq 0$, we can choose $k_n \gg 0$ such that we have
$$x_n : = (\frac{u^{ep^k}}{p})^{k_n} x \in \wt{\mathbf{A}}^{[0, r_k]}+p^n\wt{\mathbf{A}}^{I}.$$
\end{itemize}

 For either of the above two cases,  $x_n \in (\wt{\mathbf{A}}^I_{L})^{\tau_{m+k}\dan, \gamma=1}$ by Lem. \ref{lem div ana}(\ref{item-varphiu}) (and Lem. \ref{ringlocan}). So
$$\theta\circ \iota_k(x_n) \in (\O_{\hat{L}})^{\tau_{m+k}\dan, \gamma=1  } =\O_{K(\pi_{m+k})},$$
where the last identity follows from similar argument as in \cite[Thm. 3.2]{BC16}.
Since $\theta\circ \iota_k(\varphi^{-m}(u))=\pi_{m+k}$, there exists $y_{n, 0} \in W(k)[\varphi^{-m}(u)]$ such that
$$\theta\circ \iota_k(x_n) =\theta\circ \iota_k(y_{n, 0}).$$
By Lem. \ref{prop theta ker},
$$x_n -y_{n, 0} = (F/p) \cdot x_{n, 1}, \text{ with } x_{n, 1} \in  \wt{\mathbf{A}}^I, \text{ where } F: = \varphi^{k}(E(u)).$$
By Lem. \ref{lem div ana}(\ref{item-varphiu}),
 $y_{n, 0} \in (\wt{\mathbf{A}}^I_{L})^{\tau_{m+k}\dan, \gamma=1}.$ (As we mentioned in the proof of \emph{loc. cit.}, it is important to know that $y_{n, 0}$ is ``$\tau_{m+k}\dan$" for the argument here to proceed).
Thus by Lem. \ref{lem div ana}(\ref{item-divide-phie}),  $x_{n, 1} \in  (\wt{\mathbf{A}}^I_{L})^{\tau_{m+k}\dan, \gamma=1}$.
Applying this procedure inductively gives us a sequence $\{y_{n, i} \}_{i \geq 0}$ where $y_{n, i} \in W(k)[\varphi^{-m}(u)]$  such that
$$x_n -\left(   y_{n, 0} +  (F/p) y_{n, 1} +\cdots +   (F/p)^{i-1} y_{n, i-1}\right) \in (F/p)^{i}\wt{\mathbf{A}}^I_{L}. $$
By Lem. \ref{lem div p}(4), there exists $j \gg 0$ such that
\begin{equation}\label{eqjj}
x_n -\left(   y_{n, 0} +  (F/p) y_{n, 1} +\cdots +   (F/p)^{j-1} y_{n, j-1}\right) \in p \wt{\mathbf{A}}^I_{L}.
\end{equation}
Note that the left hand side of \eqref{eqjj} belongs to $\wt{\mathbf{A}}_L^{[0, r_k]}+p^n\wt{\mathbf{A}}_L^{I}$ (since $y_{n, i}$ and $F/p$ are in $\wt{\mathbf{A}}_L^{[0, r_k]}$), and so it further belongs to
$$(\wt{\mathbf{A}}_L^{[0, r_k]}+p^n\wt{\mathbf{A}}_L^{I}) \cap  p \wt{\mathbf{A}}^I_{L} = p(\wt{\mathbf{A}}_L^{[0, r_k]}+p^{n-1}\wt{\mathbf{A}}_L^{I} ), \quad \text{ by Lem. \ref{lem div p}(3) }.$$
Let
$$x_n -\left(   y_{n, 0} +  (F/p) y_{n, 1} +\cdots +   (F/p)^{j-1} y_{n, j-1}\right) = px_n'.$$
Since $y_{n, i} \in  (\wt{\mathbf{A}}^I_{L})^{\tau_{m+k}\dan, \gamma=1}$, we have $x_n' \in  (\wt{\mathbf{A}}^I_{L})^{\tau_{m+k}\dan, \gamma=1}$. Apply to $x_n'$ the same procedure that we applied to $x_n$, and proceed inductively.
In the end, we will get $\{\tilde{y}_{n, i} \}_{i \leq j_n}$ for some $j_n \gg 0$ where $\tilde{y}_{n, i}  \in W(k)[\varphi^{-m}(u)]$, and
$$\tilde{y}_n = \tilde{y}_{n, 0} + (F/p)\tilde{y}_{n, 1}  + \cdots + ((F/p))^{j_n-1}\tilde{y}_{n, j_n-1} ,
$$
such that $$x_n-\tilde{y}_n  \in p^n\wt{\mathbf{A}}^I.$$
Let $z_n := (\frac{p}{u^{ep^k}})^{k_n} \tilde{y}_n $, then $z_n \in \varphi^{-m}(\mathbf{A}_{K_\infty}^{p^m[r_k, r_k]})$ (note that here it is critical to use the interval $[r_k, r_k]$ and not $[0, r_k]$ or $[r_\ell, r_k]$, because the element $\frac{p}{u^{ep^k}}$ belongs only to $\mathbf{A}^{[r_k, r_k]}$).
We have
$$ x- z_n = (\frac{p}{u^{ep^k}})^{k_n} (x_n-\tilde{y}_n) \in p^n \wt{\mathbf{A}}^{[r_k, r_k]},$$
and hence  $z_n$ converges to $x$ as elements in $\wt{\mathbf A}^{[r_k, r_k]}$ (with respect to $W^{[r_k, r_k]}$), and so
$$x \in \varphi^{-m}(\mathbf{A}_{K_\infty}^{p^m[r_k, r_k]}).$$
Finally by Cor. \ref{lem aa inter}, we have
$$ x \in \varphi^{-m}(\mathbf{A}_{K_\infty}^{p^m[r_k, r_k]}) \cap \wt{\mathbf{A}}^I = \varphi^{-m}(\mathbf{A}_{K_\infty}^{p^mI}) = \mathbf{A}_{K_\infty, m}^{I}.$$

Consider Item (2). Item (1) already implies that $(\wt{\mathbf{A}}^I_{ L})^{\tau\dla, \gamma=1} \subset \mathbf{A}^I_{K_\infty, \infty}.$ To show the other direction, it suffices to show that elements in $\mathbf{A}^I_{K_\infty}$ are $\tau$-locally analytic.
We claim that for any $f \in \mathbf{A}^I_{K_\infty}$, and for $a \in p^{b}\Zp$, we have
\begin{equation} \label{eqfalpha}
 W^I((1-\tau^a)^s(f) \geq s\alpha
\end{equation}
for some $\alpha$ that we can arbitrarily enlarge (after enlarging $b$); then we can conclude using Lem. \ref{lem Zp an}.
To verify \eqref{eqfalpha},
by linearity and density, it suffices to verify it for the cases $f=u^m(\frac{u^{ep^k}}{p})^n$ for $m \geq 0$ and $n \geq 0$, and (when $I=[r_\ell, r_k]$ the cases $f=u^m(\frac{p}{u^{ep^\ell}})^n$ for $m \geq 0$ and $n \geq 1$.
Indeed, we have
\begin{eqnarray*}
 W^I\left((1-\tau^a)^s (u^m(\frac{u^{ep^k}}{p})^n)\right)
 &=& W^I\left( u^m(\frac{u^{ep^k}}{p})^n \cdot(1-[\underline{\varepsilon}]^{aep^kn+am})^s \right)
  \\
  &\geq& W^I\left(  (1-[\underline{\varepsilon}]^{aep^kn+am})^s \right) , \text{ since }  W^I( u^m(\frac{u^{ep^k}}{p})^n )\geq 0\\
  &\geq& s\alpha, \text{ using \eqref{eqhxx}}.
\end{eqnarray*}
The verification for  $f=u^m(\frac{p}{u^{ep^\ell}})^n$ is similar.

For Items (3) and (4), one can argue similarly as in  \cite[Thm. 4.4(3)]{Ber16}.
\end{proof}

\begin{remark} \label{rem kisin ring}
Item (4) of Thm. \ref{thm loc ana gamma 1} (and (1), (2) when $I=[0, r_k]$) will not be used in this paper, but it has potential applications to the study of semi-stable Galois representations; indeed, the ring $\mathbf{B}_{ K_\infty}^{[0, +\infty)}$ is precisely the ring $\mathcal{O}_{[0, 1)}$ in \cite{Kis06}.
\end{remark}

\begin{defn} \label{defn rig ring}
\begin{enumerate}
\item Define the following rings (which are LB spaces):
$$
\wt{\mathbf{B}}^{\dagger}: = \cup_{r \geq 0} \wt{\mathbf{B}}^{[r, +\infty]},
\quad  \mathbf{B}^{\dagger}: = \cup_{r \geq 0} \mathbf{B}^{[r, +\infty]},
\quad \mathbf{B}_{  K_\infty}^{\dagger}: = \cup_{r \geq 0} \mathbf{B}_{K_\infty}^{[r, +\infty]}.
$$

\item Define the following rings (which are LF spaces):
$$ \wt{\mathbf{B}}_{  \rig}^{\dagger}: = \cup_{r \geq 0} \wt{\mathbf{B}}^{[r, +\infty)},
\quad  \mathbf{B}_{  \rig}^{\dagger}: = \cup_{r \geq 0} \mathbf{B}^{[r, +\infty)},
\quad \mathbf{B}_{\rig, K_\infty}^{\dagger}: = \cup_{r \geq 0} \mathbf{B}_{K_\infty}^{[r, +\infty)}.
$$
\end{enumerate}
\end{defn}

\begin{cor} \label{cor rig la}
$(\wt{\mathbf{B}}_{  \rig, L}^{\dagger})^{\tau\dpa, \gamma=1} = \cup_{m\geq 0} \varphi^{-m}({\mathbf{B}}_{  \rig, K_\infty}^{\dagger}).$
\end{cor}
\begin{rem}
  In comparison, by \cite[Thm. 4.4]{Ber16}, we have
  $$(\wt{\mathbf{B}}_{  \rig, L}^{\dagger})^{\tau=1, \gamma\dpa} = \cup_{m\geq 0} \varphi^{-m}({\mathbb{B}}_{  \rig, K_{p^\infty}}^{\dagger}),$$
   where ${\mathbb{B}}_{  \rig, K_{p^\infty}}^{\dagger}$ is the ring ``${\mathbf{B}}_{  \rig, K}^{\dagger}$" in \cite{Ber08ANT}. (As we mentioned in Rem. \ref{rem notation}, we use the font ``$\mathbb{B}$" to denote the ``$\mathbf{B}$"-rings in the $(\varphi, \Gamma)$-module setting).
\end{rem}

\begin{comment}
\begin{rem} \label{rem Ghat la}
We can actually compute the full space $(\wt{\mathbf{B}}_{  \rig, L}^{\dagger})^{\hat{G}\dpa}$, by using some similar techniques as in \cite[\S 5]{Ber16}. To save space, we defer it to a subsequent paper \cite{GP}. In particular, we will then use these $\hat{G}\dpa$-vectors to give yet another proof of our main result Thm. \ref{thm intro main}; see Rem. \ref{rem mono descent} for some more comments.
\end{rem}
\end{comment}

\section{Field of norms, and locally analytic vectors} \label{sec fieldnorm}
In this section, when $K_\infty \subset M \subset L$ where $M/K_\infty$ is a finite extension, we calculate $\hat{G}$-locally analytic vectors in $\wt{\mathbf{B}}_L^I$ which are furthermore invariant under $\gal(L/M)$; the results are parallel with the case for $M=K_\infty$.

\subsection{Field of norms}
In this subsection, we briefly recall the theory of field of norms developed by Fontaine and Wintenberger (cf. \cite{FW79, Win83}). To save space, we refer the readers to \cite{Win83} for more details.

In this subsection, let $E_1$ be a complete discrete valuation field with a perfect residue field of characteristic $p$. Let $\overline{E_1}$ be a fixed algebraic closure, and let $E_1^{\ur}$ be the maximal unramified extension of $E_1$ contained in $\overline{E_1}$.

If $E_2/E_1$ is an algebraic extension, let
$\mathcal{E}(E_2/E_1)$ be the poset consisting of fields $E$ such that $E_1 \subset E \subset E_2$ and $[E:E_1]<+\infty$.
Let
$$X_{E_1}(E_2) : =\projlim_{E \in \mathcal{E}(E_2/E_1)} E$$
where the transition maps from $E'$ to $E$ (for $E \subset E'$) are the norm maps $N_{E'/E}$.
For $\alpha \in X_{E_1}(E_2)$, we denote it as $\alpha =\{\alpha_E\}_{E_1 \subset E \subset E_2}$ where $\alpha_E \in E$ and $N_{E'/E}(\alpha_{E'})=\alpha_E$ when $E\subset E'$.
For any $\alpha \in X_{E_1}(E_2)$, the number $v_E(\alpha_E)$ for $E_1^{\ur}\cap E_2 \subset E \subset E_2$ is independent of $E$ (here, $v_E$ is the valuation such that $v_E(E)=\mathbb{Z}\cup \{\infty\}$); denote the number as $v(\alpha)$.

%Let$$X_{E_1}(E_2)^* : =\projlim_{E \in \mathcal{E}(E_2/E_1)} E^*$$
%where the transition maps from $(E')^*$ to $E^*$ (for $E \subset E'$) are the norm maps $N_{E'/E}$.
%$$X_{E_1}(E_2) :  =X_{E_1}(E_2)^* \cup \{0\}.$$

A priori, $X_{E_1}(E_2)$ is only a multiplicative monoid; however, by \cite[Thm. 2.1.3(1)]{Win83}, we can indeed equip it with a natural additive structure, making $X_{E_1}(E_2)$ into a ring. Furthermore, we have the following.

\begin{theorem}\cite[Thm. 2.1.3(2)]{Win83}
Suppose $E_2/E_1$ is an infinite APF extension (cf. \cite[\S 1.2]{Win83} for the definition of APF (and strict APF) extensions), then there exists an element $u_{E_2/E_1} \in X_{E_1}(E_2)$ such that $v(u_{E_2/E_1})=1$, and there exists a (valuation-preserving) \emph{field} isomorphism
$$ X_{E_1}(E_2) \simeq k_{E_2}((u_{E_2/E_1})), $$
where $k_{E_2}$ is the residue field of $E_2$ (which is a finite extension of $k_{E_1}$), and $k_{E_2}((u_{E_2/E_1}))$ is equipped with the $u_{E_2/E_1}$-adic valuation.
\end{theorem}

\begin{example} \label{ex XK}
Let $K, K_{p^\infty}, K_\infty$ be as in Notation \ref{nota fields}.
\begin{enumerate}
\item When $K=K_0$, the element $\wt{\mu}: =\{\mu_{n}\}_{n \geq 1}$ defines an element in $X_{K}(K_{p^\infty})$, and $\wt{\mu} -1$ is a uniformizer of $X_{K}(K_{p^\infty})$.
\item The element $\wt{\pi}:  =\{\pi_{n}\}_{n \geq 1}$ defines an element in $X_{K}(K_{\infty})$, which is a uniformizer.
\end{enumerate}
\end{example}

%Note that if we write $u_{E_2/E_1}=\{u_E\}_{E_1 \subset E \subset E_2}$, then $u_E$ is a uniformizer of $E$ for all $E_1^{\ur}\cap E_2 \subset E \subset E_2$.

Let $E_1 \subset E_2 \subset E_3$ where $E_2/E_1$ is an infinite APF extension, and $E_3/E_2$ is finite extension (so $E_3/E_1$ is also an APF extension).
Then by \cite[\S 3.1.1]{Win83}, we can naturally define an embedding $X_{E_1}(E_2) \inj X_{E_1}(E_3)$ (and we identify $X_{E_1}(E_2)$ with its image).

\begin{theorem}\cite[Thm. 3.1.2]{Win83} \label{thm Win Gal}
If $E_3/E_2$ is furthermore Galois, then $X_{E_1}(E_3)$ is Galois over $X_{E_1}(E_2)$, and there exists a natural isomorphism
$$\Gal(X_{E_1}(E_3)/X_{E_1}(E_2)) \simeq \gal(E_3/E_2).$$
\end{theorem}

\begin{remark}
We can also construct a natural separable closure of $X_{E_1}(E_2)$, see \cite[Cor. 3.2.3]{Win83}.
\end{remark}

\begin{comment}
\begin{theorem}\cite[Cor. 3.2.3]{Win83}
Suppose $E_2/E_1$ is an infinite APF extension, let
$$\overline{X}: =\projlim_{E_2 \subset E_3 \subset \overline{E_1}, \quad [E_3: E_2]<+\infty}  X_{E_1}(E_3).  $$
Then $\overline{X}$ is a separable closure of $X_{E_1}(E_2)$, and there is a natural isomorphism
$$\gal(\overline{X}/X_{E_1}(E_2)) \simeq \gal(\overline{E_1}/E_2).$$
\end{theorem}
\end{comment}

For any complete valued filed $({  A}, v_{  A})$ with a perfect residue field of characteristic $p$, let
$$R({  A}): = \{ (x_n)_{n=0}^\infty: x_n \in {  A}, x_{n+1}^p=x_n \}.$$
For $x \in R({  A})$, let $v_R(x):=v_{  A}(x_0)$. Then $R({  A})$ is a perfect field of characteristic $p$, complete with respect to $v_R$.

\begin{theorem}\cite[Thm. 4.2.1]{Win83}\label{thm Win R}
Suppose $E_2/E_1$ is an infinite \emph{strict} APF extension. Let $\hat{E_2}$ be the completion of $E_2$. There exists a natural $k_{E_2}$-algebra embedding
$$ \Lambda_{E_2/E_1} : X_{E_1}(E_2) \inj R(\hat{E_2}) \inj R(\hat{\overline{E_1}}).$$
\end{theorem}

\begin{example} \label{ex embed}
Note that $R(C_p)$ is precisely $\wt{\mathbf E}$.
Using notations in Example \ref{ex XK}, we have
\begin{enumerate}
\item when $K=K_0$, for the embedding $X_{K}(K_{p^\infty}) \to \wt{\mathbf E}$, we have $\wt{\mu}-1 \mapsto \underline{\varepsilon} -1$;
\item for the embedding $X_{K}(K_{\infty}) \to \wt{\mathbf E}$, we have $\wt{\pi} \mapsto \underline{\pi}$.
\end{enumerate}
\end{example}

\subsection{Finite extensions of $K_\infty$ and locally analytic vectors}
Let $K_\infty \subset M \subset L$ where $M/K_\infty$ is a finite extension (which is always Galois). In the following, given a ring $  A$ (possibly with superscripts), let ${ A}_M$ denote $\gal(\overline{K}/M)$-invariants of $  A$.

\subsubsection{Ramification subgroups.} Let $G_K^s$ (where $s \geq -1$) denote the usual (upper numbering) ramification subgroups of $G_K$. For any $s \geq -1$, let $\overline{K}^{(s)}:=\cap_{t> s}\overline{K}^{G_K^t}$. For any $K\subset E \subset \overline{K}$, let $E^{(s)}:=E \cap \overline{K}^{(s)}$. Let $c(E):=\inf \{s: E^{(s)}=E\}$ (called the conductor of $E$). See \cite[Lem. 4.1]{Col08} for some properties of $c(E)$.
When $n \geq 1$, let $K_n:=K(\pi_n)$. By standard computation (e.g., using the formula above \cite[Prop. 1.1]{LB10}), we have
\begin{equation}\label{eqckn}
c(K_n)=(n+\frac{1}{p-1})e.
\end{equation}
(Unfortunately, the computation of $c(K_n)$ in \cite[Prop. 1.4]{LB10} is incorrect.)

% In particular, we have $c(K_n)=n-1$ (here $K_n:=K(\pi_n)$).

\subsubsection{Finite extensions of $K_\infty$.} \label{subsub M}
Choose an $\alpha \in M$ such that $M=K_\infty[\alpha]$, and let $\wt{M}:=K[\alpha]$. Define $\wt{M}_n:=\wt{M}(\pi_n)$ (note that $\pi_0=\pi$ is not necessarily a uniformizer of $\wt{M}$). By using exactly the same argument as in \cite[Lem. 4.2, Cor. 4.3, Rem. 4.4]{Col08}, the following hold:
\begin{enumerate}
  \item  When $n \geq c(\wt{M}) $ (where $c(\wt{M})$ is the conductor), then $c(\wt{M}_n)=\sup \{c(\wt{M}), c(K_n)\} =c(K_n)$ by \eqref{eqckn}, and hence
   $\wt{M}_{n+1}/\wt{M}_n$ is totally ramified of degree $p$.
  \item  When $n \geq c(\wt{M}) $, $e(\wt{M}_{n+1}/K_{n+1})=e(\wt{M}_{n}/K_{n})$ (resp. $f(\wt{M}_{n+1}/K_{n+1})=f(\wt{M}_{n}/K_{n})$), where $e(A/B)$ (resp. $f(A/B)$) is the ramification index (resp. inertial degree) of a finite extension. Denote the common numbers as $e'$ (resp. $f'$), then $e'f'=[M: K_\infty]$.
  \item Let $K':=K^{\ur}\cap M$ where $K^{\ur}$ is the maximal unramified extension of $K$ contained in $\overline{K}$, then $[K':K]=f'$.
\end{enumerate}

\subsubsection{Construction of $u_M$} \label{subsub uM}
Let $k'$ be the residue field of $K'$, and let $M_0:=\cup_{n \geq 1} K'(\pi_n)$. Then by \S \ref{subsub M} and Examples \ref{ex XK} and \ref{ex embed}, we have $X_K(M_0) \simeq k'((\underline{\pi}))=k'((u))$ (recall $u=[\underline{\pi}]$ as in \S \ref{period rings}).
Choose any  $\overline{u}_M \in X_K(M)$  such that
$X_K(M)= k'((\overline{u}_{M}))$.
By Thm. \ref{thm Win Gal}, $X_K(M)$ is a totally ramified extension of $X_K(M_0)$ of degree $e'$, and so $v_{\wt{\mathbf{E}}}(\overline{u}_M)=1/ee'$ if we regard $\overline{u}_M \in \wt{\mathbf{E}}$ via Thm. \ref{thm Win R}.
Let $\overline{P}(X)=X^{e'} + \overline{a}_{e'-1}X^{e'-1}+\cdots + \overline{a}_0$ be the minimal polynomial of $\overline{u}_M$ over $X_K(M_0)$. Since $\overline{u}_M$ is integral over $X_K(M_0)$, $\overline{a}_i \in k'[\![u]\!]$.
Let $a_i \in W(k')[\![u]\!]$ be any lift of $\overline{a}_i$, and let $P(X)=X^{e'} +  {a}_{e'-1}X^{e'-1}+\cdots +  {a}_0$. By Hensel's Lemma, $P(X)$ has a unique root (which we denote as $u_M$) in $\mathbf{A}_M$ which reduces to $\overline{u}_M$ modulo $p$. (Note that $u_M$ depends on the choices of $\overline{u}_{M}$ and $a_i$.)

We have $\gal( X_K(M)/X_K(K_\infty)) \simeq \gal(\mathbf B_M/\mathbf B_{K_\infty}) \simeq \gal(\wt{\mathbf B}_M/\wt{\mathbf B}_{K_\infty})$ (cf. \cite[\S I.3]{CC98}). Let $v_1, \cdots, v_{f'}$ be a basis of $W(k')$ over $W(k)$, and let $x_{a+ f'b}:= v_a\cdot u_M^{b}$ with $1\leq a \leq f', 0 \leq b \leq e'-1$, then we have
$$ \mathbf{A}_M = \oplus_{i=1}^{e'f'} \mathbf{A}_{K_\infty}\cdot  x_i, $$
and so (cf. \cite[Lem. 24.5]{Ber}),
$$ \wt{\mathbf{A}}_M = \oplus_{i=1}^{e'f'} \wt{\mathbf{A}}_{K_\infty}\cdot  x_i.$$

\begin{comment}
\begin{lemma} \label{lem m0am}
If $I =[r_\ell, r_k]$, then
$$ \wt{\mathbf{A}}_{M_0}^I =W(k')\otimes_{W(k)} \wt{\mathbf{A}}_{K_\infty}^I, \quad
\mathbf{A}_{M_0}^I =W(k')\otimes_{W(k)} \mathbf{A}_{K_\infty}^I.$$
\end{lemma}
\begin{proof}
\textbf{question:} true for $I=[0, +\infty]$??
Since $\pi$ is also a uniformizer of $K'$, so $M_0$ can be indeed be denoted as ``$(K')_\infty$".
All the results we developed in \S \ref{sec: rings} (also \S \ref{sec: loc ana}) remain valid if we change $K_\infty$ to $M_0$. The lemma then follows from Prop. \ref{lem decomp 3} and Prop. \ref{cor A rl rk}.
\end{proof}
\end{comment}

\begin{lemma} \label{lem units}
Let $r>0$ and let $x =\sum_{k\geq 0}p^k[a_k] \in \wt{\mathbf{A}}^{[r, +\infty]}[1/u]$, the following are equivalent:
\begin{enumerate}
\item $x \in (\wt{\mathbf{A}}^{[r, +\infty]})^\times$;
\item $v_{\wt{\mathbf E}}(a_0)=0$, and $k+\frac{p-1}{pr}\cdot v_{\wt{\mathbf E}}(a_k) >0, \forall k>0$;
\item $v_{\wt{\mathbf E}}(a_0)=0$, and $ k+\frac{p-1}{pr}\cdot w_k(x) >0, \forall k>0$.
\end{enumerate}
\end{lemma}
\begin{proof}
The equivalence between (1) and (2) is proved in \cite[Lem. 5.9]{Col08}; see the proof of Lem. \ref{lem W} for comparison of notations.
The equivalence between (2) and (3) is trivial.
\end{proof}

\begin{lemma} \label{lem um var}
\begin{enumerate}
%\item Suppose $u_M=\sum_{n=0}^\infty p^n[x_n]$, then $v_{\wt{\mathbf E}}(x_n) \geq -(2n-1)v_{\wt{\mathbf E}}(\overline{P}'(\overline{u}_M))$ for all $n$ (where $\overline{P}'$ is the derivative of $\overline{P}$).

  \item  There exists some constant $r_M>0$ which depends only on $M$ (and not on the construction of $u_M$ as in \S \ref{subsub uM}), such that:
  \begin{enumerate}
  \item  $u_M \in \mathbf{A}_M^{[r_M, +\infty]}$, and
  \item  $u_M/[\overline{u}_M]$ is a unit in $\wt{\mathbf{A}}_M^{[r_M, +\infty]}$.
  \item $P'(u_M)/[P'(\overline{u}_M)]$ is a unit in $\wt{\mathbf{A}}_M^{[r_M, +\infty]}$, where $P'(X)$ is the derivative of $P(X)$.
\end{enumerate}

  \item  If $I =[r_\ell, r_k]$ or $[r_\ell, +\infty]$ such that $r_\ell \geq r_M$, then
$$
\mathbf{B}_M^I = \oplus_{i=1}^{e'f'}\mathbf{B}_{K_\infty}^I\cdot  x_i, \quad  \quad \wt{\mathbf{B}}_M^I  =\oplus_{i=1}^{e'f'} \wt{\mathbf{B}}_{K_\infty}^I\cdot x_i. $$
\end{enumerate}
\end{lemma}
\begin{proof}
Item (1) follows from exactly the same argument as \cite[Lem. 6.4, Lem. 6.5]{Col08} (where Item (1b) uses Lem. \ref{lem units}).
Item (2) follows from exactly the same argument as \cite[Lem. 6.11]{Col08} (i.e., an argument using the trace operator).
\end{proof}

\begin{lemma} \label{lem um}
Suppose $r_\ell \geq r_M$, then
$x_i \in (\wt{\mathbf{A}}^{[r_\ell, r_k]}_L)^{\tau\dla}$.
\end{lemma}

\begin{rem}
The proof of Lem. \ref{lem um} is inspired by the argument in the proof of \cite[Thm. 4.4(2)]{Ber16}; indeed, we use ideas inspired by the inverse function theorem on \cite[Page 73]{Ser06}.
However, since the ring $\wt{\mathbf{A}}^{[r_\ell, r_k]}_L$ (or $\wt{\mathbf{B}}^{[r_\ell, r_k]}_L$) is not a \emph{field} and the norm on it is not \emph{multiplicative}, we cannot directly apply \emph{loc. cit.}. (we thank an anonymous referee for pointing this out). Indeed, the argument in \cite[Thm. 4.4(2)]{Ber16} is incomplete. Let us mention that the argument in our proof can be easily adapted to give a corrected proof of  \emph{loc. cit.}.
\end{rem}
%In the notation of \emph{loc. cit.}, Berger claims that since ``$P'(x) \neq 0$" and hence one can apply implicit function theorem; however the problem is that ``$P'(x) \neq 0$" does not imply $P'(x)$ is a unit in the \emph{ring} in  \emph{loc. cit.}

We first start with an easy lemma.
\begin{lem} \label{lem deri}
  Let $(W, \|\cdot \|)$ be a normed $\Zp$-algebra. Let $\val(\cdot)$ be the associated valuation on $W$, and suppose it is \emph{multiplicative}. Let $f(X)=X^n+a_{n-1}X^{n-1}+\ldots  +a_0$ where $a_i \in W$ such that $\val(a_i) \geq 0$. Suppose $\rho \in W$ such that $f(\rho)=0$ and $f'(\rho)\neq 0$ (where $f'(X)$ is the derivative). Suppose $\rho' \in W$ such that $f(\rho')=0$ and $\val(\rho-\rho') >\val(a_i)$ for all $i$ such that $a_i \neq 0$. Then $\rho=\rho'$ (i.e., within a small neighborhood of $\rho$, $f(X)$ has no other roots.)
\end{lem}
\begin{proof}
  Firstly, it is easy to see that $\val(\rho) \geq 0$; then we can easily reduce the lemma to the case $\rho=0$. That is, we can assume $f(X)=X^n+a_{n-1}X^{n-1}+\ldots +a_1X$ and $a_1 \neq 0$. Now if $\rho' \neq 0$ and $\val(\rho') >\val(a_i)$ for all $i$ such that $a_i \neq 0$, then $\val(f(\rho')) =\val(a_1\rho') <+\infty$, and hence $f(\rho')\neq 0$.
\end{proof}

\newcommand{\Coeff}{\textnormal{Coeff}}
\begin{proof}[Proof of Lem. \ref{lem um}] The lemma is trivial if $e'=1$; suppose now $e' \geq 2$.
Firstly, by Lem. \ref{ringlocan}, it suffices to show that $u_M \in (\wt{\mathbf{A}}^{I}_L)^{\tau\dla}$ (here $I:=[r_\ell, r_k]$).
Recall we denote $P(X)=X^n+a_{n-1}X^{n-1}\ldots +a_0$ in \S \ref{subsub uM} (here we write $n:=e' \geq 2$ for brevity), where $a_i \in  W(k')[\![u]\!]$. Thus for all $\theta \in \Zp$, $\tau^\theta P(X):=X^n+\tau^\theta(a_{n-1})X^{n-1}\ldots+ \tau^\theta(a_0)$ has $\tau^\theta(u_M)$ as a root in $\wt{\mathbf{A}}^{I}$.

For $m \gg 0$ and for each $\beta \in   \Zp$, we will \emph{construct} another root of $\tau^{p^m\beta} P(X)$ of the form
\begin{equation}\label{eqybeta}
 y=y(m, \beta) =w_0+\sum_{k \geq 1}(p^m\beta)^k w_k' =w_0+\sum_{k \geq 1}\beta^k w_k,
\end{equation}
where $w_0=u_M$ (independent of $m$), and for each $k \geq 1$
$w_k:= w_k(m):= p^{mk}w_k'$ (here $w_k'$ depends only on $k$ and $\beta$ but not on $m$ )  such that
\begin{equation} \label{eqwk}
w_k \in \wt{\mathbf{A}}^{I}_L, \text{ and hence }  \lim_{k\to +\infty}  w_k = 0 \text{  by enlarging } m.
\end{equation}
Now fix any $s \in I$. By enlarging $m$ if necessary, we can easily make
\begin{equation}\label{eqm1}
W^{[s, s]}(y  -u_M) >W^{[s, s]}(a_i), \forall i \text{ such that } a_i \neq 0,
\end{equation}
and
\begin{equation}\label{eqm2}
   W^{[s, s]}(\tau^{p^m\beta}(u_M) -u_M)>W^{[s, s]}(a_i), \forall i \text{ such that } a_i \neq 0.
\end{equation}
Here, \eqref{eqm2} is possible because the Galois action on $\wt{\mathbf{A}}^{[s, s]}$ is continuous.
By \eqref{eqm1} and \eqref{eqm2}, we have
$$W^{[s, s]}(y -\tau^{p^m\beta}(u_M))>W^{[s, s]}(a_i), \forall i \text{ such that } a_i \neq 0.$$
By Lem. \ref{lem deri} (recall $W^{[s, s]}$ is an multiplicative valuation by Lem. \ref{lem W}), we can conclude $\tau^{p^m\beta}(u_M)=y $ as elements in
$\wt{\mathbf{A}}^{[s, s]}$. Since $\wt{\mathbf{A}}^{I} \hookrightarrow \wt{\mathbf{A}}^{[s, s]}$ (cf. \S \ref{subsubcontain}), we have  $\tau^{p^m\beta}(u_M)=y $ as elements in
 $\wt{\mathbf{A}}^{I}$. Thus $u_M \in (\wt{\mathbf{A}}^{I}_L)^{\tau\dla}$   by definition.

Now we construct $y $ in \eqref{eqybeta}. Before we do  so, we pick some $\delta \gg 0$ such that
\begin{equation}\label{eqdelta}
 {p^\delta}/{P'(u_M)} \in \wt{\mathbf{A}}^{I}_L,
\end{equation}
which is possible because of Lem. \ref{lem um var}(1)(c).
Now note that all $a_i$ are locally analytic vectors, so we can write for each $i$,
\begin{equation} \label{eqai}
\tau^{p^m\beta}(a_i)=a_{i, 0}+\sum_{j \geq 1} (p^m\beta)^j a_{i, j}'=a_i+\sum_{j \geq 1} \beta^j  a_{i, j},
\end{equation}
where again $a_{i, 0}=a_i$. By enlarging $m$ if necessary, we can suppose
\begin{equation}\label{eqaij}
  a_{i, j} \in p^{2\delta}\wt{\mathbf{A}}^{I}_L, \quad \forall 0\leq i \leq n-1, \forall j \geq 1.
\end{equation}
Plug  \eqref{eqai} and \eqref{eqybeta} into $\tau^{p^m\beta} P(X)$. We get
\begin{equation}\label{eqplug}
(w_0+\sum_{k \geq 1}\beta^k w_k)^n +  (a_{n-1, 0}+\sum_{j \geq 1} \beta^j a_{{n-1}, j}) (w_0+\sum_{k \geq 1}\beta^k w_k)^{n-1} + \cdots + (a_{0, 0}+\sum_{j \geq 1} \beta^j a_{0, j}) =0.
\end{equation}
We will let the coefficient of $\beta^k$ to be zero for each $k \geq 0$, and use these equations to solve $w_k$ inductively. Firstly, note that we automatically have
\begin{equation}\label{eqk0}
\Coeff(\beta^0) =w_0^n + \sum_{i=0}^{n-1} a_{i, 0}\cdot w_0^i=P(w_0)=P(u_M)=0.
\end{equation}
For each $k\geq 1$, one can easily compute that
\begin{equation}\label{eqkbig1}
\Coeff(\beta^k) =P'(w_0)\cdot w_k + Q_k\left((a_{i, j})_{1\leq i \leq n-1, 0\leq j \leq k-1}, w_0, \cdots, w_{k-1}\right)
\end{equation}
where $Q_k$ is a polynomial of the variables $(a_{i, j})_{1\leq i \leq n-1, 0\leq j \leq k-1}, w_0, \cdots, w_{k-1}$ with \emph{integer} coefficients.
By letting  $\Coeff(\beta^k) =0$, we will show by induction that
\begin{equation} \label{eqinduction}
w_k \in p^\delta \wt{\mathbf{A}}_L^I,\quad \forall k \geq 1.
\end{equation}
It suffices to show that each monomial in $Q_k$ is divisible by $p^{2\delta}$, since by \eqref{eqdelta}
$$p^{2\delta}\wt{\mathbf{A}}_L^I \subset P'(w_0) \cdot p^\delta \wt{\mathbf{A}}_L^I.$$
When $k=1$, each monomial in $Q_1$ contains some  $a_{i, 1}$ as a factor, and hence  one can conclude \eqref{eqinduction} for $k=1$ using \eqref{eqaij}.
  Suppose  \eqref{eqinduction} is true for $k-1$ , and consider $\Coeff(\beta^k)$ (where now $k \geq 2$). For a monomial in $Q_k$, if it does not contain any $a_{i, j}$ with $j \geq 1$ as a factor,  then it is a product of elements in $  \{ a_{0, 0}, \ldots,    a_{n-1, 0}, w_0, w_1, \ldots, w_{k-1} \}$; however, one easily observes that such product contains at least two (possibly equal) elements  from $\{w_1, \ldots, w_{k-1} \}$ (using $k \geq 2$), and hence by induction hypothesis the monomial is divisible by $p^{2\delta}$. Thus, \eqref{eqinduction} is verified for $k$, and this finishes the construction of \eqref{eqybeta}.
\end{proof}

\begin{theorem}\label{thm la M}
Suppose $[r, s]=[r_\ell, r_k]$, then
\begin{enumerate}
\item $(\wt{\mathbf{B}}_L^{[r, s]})^{\tau\dla, \gal(L/M)=1} =\cup_{m \geq 0}\varphi^{-m}(\mathbf{B}_M^{p^m[r, s]}).$
\item $(\wt{\mathbf{B}}_L^{[r, +\infty)})^{\tau\dpa, \gal(L/M)=1} =\cup_{m \geq 0}\varphi^{-m}(\mathbf{B}_M^{p^m[r, +\infty)}).$
\end{enumerate}
\end{theorem}
\begin{proof}
%When $M=M_0$, then the theorem follows from Thm. \ref{thm loc ana gamma 1} (cf. the proof in Lem. \ref{lem m0am}). For the general case,

It suffices to prove Item (1). Denote $I:=[r, s]$. Since $\varphi$ induces a bijection between $(\wt{\mathbf{B}}_L^{I})^{\tau\dla, \gal(L/M)=1}$ and $(\wt{\mathbf{B}}_L^{pI})^{\tau\dla, \gal(L/M)=1}$, it suffices to consider the case when $r> r_M$. By Lem. \ref{lem um var}(2) and Lem. \ref{lem um}, it is clear that $\cup_{m \geq 0}\varphi^{-m}(\mathbf{B}_M^{p^m[r, s]}) \subset (\wt{\mathbf{B}}_L^{[r, s]})^{\tau\dla, \gal(L/M)=1}.$ But we also have
\begin{eqnarray*}
(\wt{\mathbf{B}}_L^I)^{\tau\dla, \gal(L/M)=1} &=&    (\wt{\mathbf{B}}_M^I)^{\tau\dla} \\
 &=&     (\oplus_{i=1}^{e'f'}\wt{\mathbf{B}}_{K_\infty}^I\cdot x_i)^{\tau\dla}, \text{ by  Lem. } \ref{lem um var}(2) \\
 &=&  \oplus_{i=1}^{e'f'}(\wt{\mathbf{B}}_{K_\infty}^I)^{\tau\dla}\cdot x_i,  \text{ by Prop.} \ref{prop la basis} \text{ and Lem.} \ref{lem um} \\
 &= &  \oplus_{i=1}^{e'f'}({\mathbf{B}}_{K_\infty, \infty}^I) \cdot x_i,  \text{ by Thm.} \ref{thm loc ana gamma 1} \\
 &\subset  &\cup_{m \geq 0}\varphi^{-m}(\mathbf{B}_M^{p^m[r, s]}), \text{ by Lem.} \ref{lem um var}(2).
\end{eqnarray*}
\end{proof}

\subsection{Structure of $A_M^I$}
In this subsection, we study the concrete structure of $A_M^I$; these results will be used in \S \ref{sec oc}.

\begin{definition}
\begin{enumerate}
\item
For $0 < r <+\infty$, let $\mathcal{A}_M^{[r, +\infty]}(K_0')$ be the ring consisting of infinite series $f=\sum_{k \in \mathbb Z} a_kT^k$ where $a_k \in W(k')$ such that $f$ is a holomorphic function on the annulus defined by
$0< v_p(T)\leq (p-1)/(e'epr) .$
Let $\mathcal{B}_M^{[r, +\infty]}(K_0'): =\mathcal{A}_M^{[r, +\infty]}(K_0')[1/p] $.

\item
For $f=\sum_{k \in \mathbb Z} a_kT^k \in \mathcal{B}_M^{[r, +\infty]}(K_0')$, and $s \in [r, +\infty)$, let
$$ \mathcal W_M^{[s, s]}(f):  =\inf_{k \in \Z}  \{ v_p(a_k)  + \frac{p-1}{ps} \cdot \frac{k}{e'e}  \}.$$
For $I=[a, b] \subset [r, +\infty)$ a non-empty {closed} interval, let
$$\mathcal W_M^{[a, b]}(f): =\inf_{\alpha \in I} \{\mathcal  W_M^{[\alpha, \alpha]}(f) \}.$$

\item Let $\mathcal{B}_M^{[r, s]}(K_0')$ be the completion of $\mathcal{B}_M^{[r, +\infty]}(K_0')$ with respect to $\mathcal W_M^{[r, s]}$. Let $\mathcal{A}_M^{[r, s]}(K_0')$ be the ring of integers with respect to $\mathcal W_M^{[r, s]}$.
\end{enumerate}
\end{definition}

\begin{lemma} \label{lem laurent series M}
For $I=[r, s] \subset (0, +\infty)$, we have $\mathcal W_M^{I}(x) = \inf \{\mathcal W_M^{[r, r]}(x), \mathcal W_M^{[s, s]}(x)\}$. Furthermore,
$\mathcal{B}_M^{[r, s]}(K_0')$ is  the ring consisting of infinite series $f=\sum_{k \in \mathbb Z} a_kT^k$ where $a_k \in K_0'$ such that $f$ is a holomorphic function on the annulus defined by
$$v_p(T) \in  [ \frac{p-1}{e'ep}\cdot \frac{1}{s},  \quad \frac{p-1}{e'ep}\cdot \frac{1}{r} ].$$
\end{lemma}
\begin{proof}
This is easy.
\end{proof}

\begin{lemma} \label{lem inj M}
Suppose $r > r_M$.
\begin{enumerate}
\item  The map $f(T) \mapsto f(u_M)$ induces a ring isomorphism
$$\mathcal{A}_M^{[r, +\infty]}(K_0') \simeq  \mathbf{A}^{[r,+\infty]}_M[1/u_M]$$
such that for $f \in \mathcal{A}_M^{[r, +\infty]}(K_0')$, and all $s$ such that $r \leq s <+\infty$, we have
$$\mathcal{W}_M^{[s, s]}(f(T)) = W^{[s, s]}(f(u_M)).$$

\item For any $s\geq r$, the map $f(T) \mapsto f(u_M)$ is an isometric isomorphism
$$\mathcal{A}_M^{[r, s]}(K_0') \simeq  \mathbf{A}^{[r,s]}_M$$
\end{enumerate}
\end{lemma}

The proof uses similar strategy as in Lem. \ref{lem inj}. We first study the section $s$.

\subsubsection{The section $s$.}\label{subsub sec s}
 %$\overline{x}=\sum_{i \gg-\infty} \bar{a_i}\overline{u}_M^i$, let $s(\overline{x}):=\sum_{i \gg-\infty} [\bar{a_i}]u_M^i$.
Denote
$$s: X_K(M)=\mathbf{A}_M/p \to \mathbf{A}_{M}$$
the section where for $\overline{x}=\overline{u}_M^{b}(\sum_{i \geq 0} \bar{a_i}\overline{u}_M^i)$ with $\bar{a_0} \neq 0$, $s(\overline{x}) :=u_M^b \sum_{i \geq 0}[\bar{a_i}]u_M^i$.  (When $M=K_\infty$, this is precisely the $s$ in \S \ref{subsub sec s Kinfty}.)
Using the expression, one can check that:
\begin{enumerate}[leftmargin=*]
\item $s(\overline{x}) \in \mathbf{A}^{[r_M,+\infty]}_{M}[1/u_M]$;
\item $W^{[r_M, r_M]}(s(\overline x)) =W^{[r_M, r_M]}(u_M^b)=W^{[r_M, r_M]}([\overline{u}_M]^b)=   (p-1)(pr_M)^{-1}\cdot v_{\wt{\mathbf E}}(\overline x),$
where the first equality is because $\sum_{i \geq 0} [\bar{a_i}]{u}_M^i$ is a unit in $\mathbf{A}^{[r_M,+\infty]}_{M}$, and the second equality uses Lem. \ref{lem um var}(1b);
\item $w_0(s(\overline{x}))=v_{\wt{\mathbf E}}(\overline x)$;
\item since $s(\overline{x})/[\overline{u}_M]^b$ is a unit in $\mathbf{A}^{[r_M,+\infty]}_{M}$, Lem. \ref{lem units}(3) implies that when $k\geq 1$,
\begin{equation} \label{strictine}
w_k(s(\overline x)) > v_{\wt{\mathbf E}}(\overline x)- k\cdot {pr_M}{(p-1)^{-1}} =w_0(s(\overline{x})) - k\cdot  {pr_M}{(p-1)^{-1}}.
\end{equation}
\end{enumerate}

\subsubsection{An approximating sequence.} \label{sub appr}
Given  $x \in  \mathbf{A}^{[r_M,+\infty]}_{M}[1/u_M]$, define a sequence $\{x_n\}$ in $\mathbf{A}^{[r_M,+\infty]}_{M}[1/u_M]$ where $x_0=x$ and $x_{n+1}:=p^{-1}(x_n-s(\overline{x_n}))$. Note that $x=\sum_{n \geq 0}p^n  s(\overline{x_n})$.
Similarly as in \cite[Lem. 7.3]{Col08}, we have
\begin{eqnarray*}
w_k(x_{n+1}) &\geq & \inf \{ w_{k+1}(x_n), w_{k+1}(s(\overline{x_n})) \} \\
 &\geq & \inf \{ w_{k+1}(x_n), w_{0}(s(\overline{x_n})) -(k+1)\cdot {pr_M}{(p-1)^{-1}} \}, \text{ by } \eqref{strictine} \\
&= & \inf \{ w_{k+1}(x_n), w_{0}(x_n) -(k+1)\cdot  {pr_M}{(p-1)^{-1}} \}.
\end{eqnarray*}
Similarly as in \cite[Lem. 7.4]{Col08}, by repeatedly using the above, we have
\begin{equation} \label{eqvr}
v_{\wt{\mathbf E}}(\overline{x_n})=w_{0}(x_n) \geq \inf_{0 \leq i \leq n}\{ w_i(x)-(n-i) \cdot  {pr_M}{(p-1)^{-1}}\}.
\end{equation}

\begin{proof}[Proof of Lem. \ref{lem inj M}]
It suffices to prove Item (1). Given $f(T) \in \mathcal{A}_M^{[r, +\infty]}(K_0')$, then similarly as in (Part 1) of the proof of Lem. \ref{lem inj}, $f(u_M) \in \mathbf{A}^{[r,+\infty]}_M[1/{u_M}]$, and $W^{[s, s]}(f(u_M)) \geq \mathcal W_M^{[s, s]}(f(T))$.

For the other direction, suppose $x \in \mathbf{A}^{[r,+\infty]}_{M}[1/u_M]$, let $\{x_n\}$ be the sequence constructed in \S \ref{sub appr}. Let $f_n(T)$ be a formal series such that $f_n(u_M) =s(\overline{x_n})$. Note that $f_n(T)$ is $T^{ v_{\wt{\mathbf E}}(\overline{x_n})/v_{\wt{\mathbf E}}(\overline{u}_M) }$  times a unit in $\mathbf{A}_M^{[r_M, +\infty]}$ (note that $T^{ v_{\wt{\mathbf E}}(\overline{x_n})/v_{\wt{\mathbf E}}(\overline{u}_M) }$ makes sense since $\overline{x_n}$ belongs to $X_K(M)=k'(\!(\overline{u}_M)\!)$), and so for any $s \geq r$,
\begin{eqnarray*}
\mathcal W_M^{[s, s]}(p^nf_n(T)) &\geq& \mathcal W_M^{[s, s]}(p^nT^{ v_{\wt{\mathbf E}}(\overline{x_n})/v_{\wt{\mathbf E}}(\overline{u}_M) })\\
&\geq&  n+\frac{p-1}{ps}\cdot \inf_{0 \leq i \leq n}\{ w_i(x)-  \frac{(n-i) pr_M}{p-1}\}, \text{ by } \eqref{eqvr} \\
&=& \inf_{0 \leq i \leq n}\{  \frac{p-1}{ps}\cdot w_i(x)+  i  +  (n-i)(1-\frac{r_M}{s}) \}\\
&\geq&  \inf_{0 \leq i \leq n}\{ \frac{p-1}{ps}\cdot w_i(x)+ i\}, \text{ since } s>r_M\\
& \geq & W^{[s, s]}(x).
\end{eqnarray*}
Note that $ \inf_{0 \leq i \leq n}\{ \frac{p-1}{ps}\cdot  w_i(x)+  i  +  (n-i)(1-\frac{r_M}{s}) \}$ converges to $+\infty$ when $n \to +\infty$, so $f(T)=\sum_{n \geq 0}p^nf_n(T)$ converges in $\mathcal{A}_M^{[r, +\infty]}(K_0')$. Clearly $f(u_M)=x$, and $\mathcal{W}_M^{[s, s]}(f(T)) \geq W^{[s, s]}(x).$
\end{proof}

\begin{prop} \label{cor am}
%Any element of $\mathbf{A}_M$ can be written as an infinite summation $\sum_{k \in \mathbb{Z}} a_ju_M^j$ where $a_j \in W(k')$ and $a_j \to 0$ as $j \to -\infty$.
 Suppose $r_\ell>r_M$, then
$$ \mathbf{A}_{M}^{[r_\ell, +\infty]}=W(k')[\![u_M]\!]\{ \frac{p}{u_M^{e'ep^\ell}}\},
\quad   \mathbf{A}_{M}^{[r_\ell, r_k]}=W(k')[\![u_M]\!]\{ \frac{p}{u_M^{e'ep^\ell}} , \frac{u_M^{e'ep^k}}{p} \}  $$
\end{prop}
\begin{proof}
It follows from Lem. \ref{lem laurent series M} and Lem. \ref{lem inj M}.
\end{proof}

\begin{corollary} \label{lem aa interM}
Suppose $[r, s] \subset [r', s] \subset (r_M, +\infty]$, then
$\mathbf{A}^{[r, s]}_M  \cap \wt{\mathbf{A}}^{[r', s]} = \mathbf{A}^{[r', s]}_M$.
\end{corollary}
\begin{proof}
This is similar to Cor. \ref{lem aa inter}, by using Prop. \ref{cor am}.
\end{proof}

\begin{lemma}\label{lem unit am}
Suppose $r >r_M$.
If $x \in \mathbf{A}_M^{[r, +\infty]}[1/u_M]$ and $x \in (\wt{\mathbf{A}}^{[r, +\infty]})^\times$, then $x\in (\mathbf{A}_M^{[r, +\infty]})^{\times}$.
\end{lemma}
\begin{proof}
Let $\{x_n\}$ be the sequence constructed in \S \ref{sub appr}, and so $x=\sum_{n \geq 0}p^n s(\overline{x_n})$.
By Lem. \ref{lem units}, $v_{\wt{\mathbf E}}(\overline{x_0})=0$, and so $s(\overline{x_0}) \in (\mathbf{A}_M^{[r, +\infty]})^{\times}$.
It then suffices to show that $1+y \in (\mathbf{A}_M^{[r, +\infty]})^{\times}$, where $y=\sum_{n \geq 1}p^n s(\overline{x_n})/s(\overline{x_0})$.
As we calculated in the proof of Lem. \ref{lem inj M},
$$W^{[r, r]} (p^n s(\overline{x_n})) \geq  \inf_{0 \leq i \leq n}\{  \frac{p-1}{pr}\cdot w_i(x)+ i  + (n-i)(1-\frac{r_M}{r}) \} >0,$$
where the final inequality uses $n \geq 1$ and Lem. \ref{lem units}.
And since $W^{[r, r]} (p^n s(\overline{x_n})) \to +\infty$ when $n \to +\infty$, so $W^{[r, r]}(y)>0$, and   $(1+y)^{-1} \in  \mathbf{A}_M^{[r, r]}$. Thus by Cor. \ref{lem aa interM}, we can conclude that $(1+y)^{-1} \in  \mathbf{A}_M^{[r, r]}\cap  \wt{\mathbf{A}}_M^{[r, +\infty]} =\mathbf{A}_M^{[r, +\infty]}$.
\end{proof}

\begin{comment}
 So it suffices to show that $n+\frac{p-1}{pr}w_0(x_n)>0, \forall n \geq 1$.
Iterate the inequality
$$w_0(x_n) \geq \inf \{ w_1(x_{n-1}), w_1(s(\overline{x_{n-1}})) \},$$
we get
$$w_0(x_n) \geq \inf \{ w_n(x_0),  w_1(s(\overline{x_{n-1}})), w_2(s(\overline{x_{n-2}})), \ldots, w_n(s(\overline{x_{0}})).\} $$
By Lemma \ref{lem units},
$$ n+\frac{p-1}{pr}w_n(x_0)>0.$$
For any $1\leq i \leq n$, we have
\begin{eqnarray*}
n+\frac{p-1}{pr} w_i(s(\overline{x_{n-i}})) &> &  n+\frac{p-1}{pr}\left( v_{\wt{\mathbf E}}(\overline{x_{n-i}}) -i\frac{pr}{p-1} \right), \text{ by } \eqref{strictine} \\
&\geq & n-i+\frac{p-1}{pr}v_{\wt{\mathbf E}}(\overline{x_{n-i}}) \\
&\geq & n-i +\frac{p-1}{pr} \inf_{0\leq j \leq n-i} \{ w_j(x)-\frac{(n-i-j)pr}{p-1} \}, \text{by } \eqref{eqvr} \\
&= & \inf_{0\leq j \leq n-i} \{\frac{p-1}{pr} w_j(x)+j \} \geq 0
\end{eqnarray*}
So we have $n+\frac{p-1}{pr}w_0(x_n)>0, \forall n \geq 1$, concluding the proof.
\end{comment}

\section{Computation of $\hat{G}$-locally analytic vectors} \label{sec Ghat la}
In this section, we compute the $\hat{G}$-locally analytic vectors in $\wt{\mathbf{B}}_L^I$. The strategy is very similar to \cite[Thm. 5.4]{Ber16}: we need to find a ``formal variable" (denoted as $b$ in the following) which plays the role of $\mathbf{y}$ in \cite[Thm. 5.4]{Ber16} (or of $\alpha$ in Prop. \ref{loc ana in L}(1)). Indeed, the discovery of $b$ is the key observation for our calculations.
In the following, we define $b$, and then use Tate's normalized traces to build an approximating sequence $b_n$, and use them to determine the set of $\hat{G}$-locally analytic vectors in $\wt{\mathbf{B}}_L^I$.

\subsection{The element $b$}
Let $\lambda :=\prod_{n \geq 0} \varphi^n(\frac{E(u)}{E(0)}) \in \mathbf{B}_{K_\infty}^{[0,+\infty)}$.
Let $b:= \frac{t}{p\lambda}$, then $b$ is precisely the $\mathfrak{t}$ in \cite[Example 3.2.3]{Liu08}, and $b \in \wt{\mathbf{A}}_L^+$.
Since $\wt{\mathbf{B}}_L^\dagger$ is a field (\cite[Prop. 5.12]{Col08}), there exists some $r(b)>0$ such that $1/b \in \wt{\mathbf{B}}_L^{[r(b), +\infty]}$.

\begin{lemma} \label{lem b}
If $r_\ell \geq r(b)$, then $b, 1/b \in (\wt{\mathbf{B}}^{[r_\ell, r_k]}_{ L})^{\hat{G}\dla}$.
\end{lemma}
\begin{proof}
Since $\gamma$ acts on $b$ (resp. $1/b$) via cyclotomic character (resp. inverse of cyclotomic character), it suffices to show that $b$ (resp. $1/b$) is $\tau$-locally analytic (cf. the argument in Lem. \ref{lem taugamma}). The result for $1/b$ follows from Lem. \ref{lem div ana}(3). Then Lem. \ref{ringlocan}(2) implies that $b$ is also locally analytic.
\end{proof}

\begin{rem}
\begin{enumerate}
  \item It seems likely that $b \in (\wt{\mathbf{B}}^{[r, s]}_{ L})^{\hat{G}\dla}$ for any $[r, s] \in [0, +\infty)$, just as the element $t/(\varphi^k(E(u)))$ in Lem. \ref{lem div ana}(\ref{item-m0}); but we do not know how to prove it.
  \item The result that $b \in (\wt{\mathbf{B}}^{[r, s]}_{ L})^{\hat{G}\dla}$ for $r \geq r(b)$ implies easily that $t/(\varphi^k(E(u))) \in (\wt{\mathbf{B}}^{[r, s]}_{ L})^{\hat{G}\dla}$ for $r \geq r(b)$, because the element $\lambda/(\varphi^k(E(u)))$ is locally analytic; this (partial) proof of Lem. \ref{lem div ana}(\ref{item-m0}) avoids use of Lem. \ref{lem Zp an}.
      However, we need the full result of  Lem. \ref{lem div ana}(\ref{item-m0}) for the calculation in Thm. \ref{thm loc ana gamma 1}.
\end{enumerate}
\end{rem}

\subsection{Tate's normalized traces}
Recall (see e.g., \cite[\S 5.1]{Col08}) that the weak topology on $\wt{\mathbf{A}}$ is the one defined by the semi-valuations $w_k$, for $k \in \mathbb{N}$, meaning that $x_n \rightarrow x$ for the weak topology in $\wt{\mathbf{A}}$ if and only if for all $k \in \mathbb{N}$, $w_k(x_n-x) \rightarrow +\infty$.
In particular, the set $\{p^n\wt{\mathbf{A}}+u^k\wt{\mathbf{A}}^+\}_{n, k \geq 0}$ forms a basis of neighbourhoods of $0$ in $\wt{\mathbf{A}}$ for the weak topology. The following lemma is very useful.

\begin{lemma} \label{lem weak rr}
Let $r'>0$ and $x_n \in \wt{\mathbf A}^{[r', +\infty]}, \forall n \geq 1$. Suppose $x_n \to 0$ in $\wt{\mathbf A}$ with respect to the weak topology. Then for any $r' < s <+\infty$ (note that it is critical $s \neq r'$),
$x_n \to 0$ in $\wt{\mathbf A}^{[s, +\infty]}$ with respect to the $W^{[s, s]}$-topology.
\end{lemma}
\begin{proof}
This is implied by \cite[Prop. 5.8]{Col08}. Indeed, we can let the ``$C$" in \emph{loc. cit.} to be 0 (see the proof of our Lem. \ref{lem W} for comparison of notations).
\end{proof}

\begin{comment}
\begin{rem} \label{rem weak rrss}
We do not know if the (induced) weak topology on $\wt{\mathbf A}^{[r, +\infty]}$ (for $r>0$) is stronger than the $W^{[r, r]}$-topology.
\end{rem}
\end{comment}

In this subsection, we let $K_\infty \subset M \subset L$ where $M/K_\infty$ is a finite extension.
For $n \geq 1$ and $I$ an interval, let
$$\mathbf{A}_{M,n}:=\varphi^{-n}(\mathbf{A}_M), \quad \mathbf{A}_{M,n}^{I}:=\varphi^{-n}(\mathbf{A}_M^{p^nI}).$$
Denote $J :=  \Z[1/p] \cap [0,1)$ and for $n \in \mathbb N$, let $J_n := \{i \in J : v_p(i) \geq -n\}$.

\begin{lemma}~ \label{lem ap1}
\begin{enumerate}
\item Every element $x \in \mathbf{E}_{M,n}:=\varphi^{-n}(\mathbf{E}_{M})$ admits a unique expression $x = \sum_{i \in J_n}u^ia_i(x)$ where $a_i(x) \in \mathbf{E}_M$.

\item Every element $x \in \wt{\mathbf{E}}_M$ admits a unique expression $x = \sum_{i \in J}u^ia_i(x)$ where $a_i(x) \in \mathbf{E}_M$ and $a_i(x) \rightarrow 0$ (here convergence is with respect to the usual co-finite filter; i.e., with respect to any ordering of $J$).

\item Every element $x \in \mathbf{A}_{M,n}$ admits a unique expression $x = \sum_{i \in J_n}u^ia_i(x)$ where $a_i(x) \in \mathbf{A}_M$.

\item Every element $x \in \wt{\mathbf{A}}_M$ admits a unique expression $x = \sum_{i \in J}u^ia_i(x)$ where $a_i(x) \in \mathbf{A}_M$ and $a_i(x) \rightarrow 0$ for the weak topology.
\end{enumerate}
\end{lemma}
\begin{proof}
These are easy analogues of \cite[Prop. 8.3, Prop. 8.5]{Col08}.
\end{proof}

We now define, for $n \in \mathbb{Z}^{\geq 0}$, $R_{M,n} : \wt{\mathbf{A}}_M \to \wt{\mathbf{A}}_M$ by
\[
R_{M,n}(x) = \sum_{i \in J_n}u^ia_i(x).
\]

\begin{prop} \label{prop ap2}
\begin{enumerate}
  \item For $x \in \wt{\mathbf{A}}_M$, we have $R_{M,n}(x) \in \mathbf{A}_{M,n}$ and $R_{M,n}(x) \rightarrow x$ for the weak topology.
  \item Let $r' > 0$ and suppose $x \in \wt{\mathbf{A}}_M^{[r',+\infty]}$. Suppose $n \gg 0$ such that $p^nr' > r_M$ (where $r_M$ is as in Lem. \ref{lem um var}), then $R_{M,n}(x) \in \mathbf{A}_{M,n}^{[r',+\infty]}$, and   $R_{M,n}(x) \to x$ for both the weak topology and the $W^{[r, s]}$-topology for any $r'< r \leq s <+\infty$.
   In particular, $\mathbf{A}_{M, \infty}^{[r', +\infty]}:=\cup_{m \geq 0} \mathbf{A}_{M, m}^{[r', +\infty]}$ is dense in $\wt{\mathbf{A}}_M^{[r',+\infty]}$  for both the weak topology and the $W^{[r, s]}$-topology.
\end{enumerate}
\end{prop}
\begin{proof}
Item (1) follows from Lem. \ref{lem ap1}.
For Item (2), the result that $R_{M,n}(x) \in \mathbf{A}_{M,n}^{[r',+\infty]}$ for $n\gg 0$ is analogue of \cite[Cor. 8.11]{Col08}. The convergence $R_{M,n}(x) \to x$ with respect to the weak topology follows from Item (1); the convergence for the $W^{[r, s]}$-topology then follows from Lem. \ref{lem weak rr} (note that $W^{[r, s]}=\inf \{W^{[r, r]}, W^{[s, s]} \}$).
\end{proof}

\begin{comment}
\begin{enumerate}
  \item Lem. \ref{lem weak rr} shows that in the situation of  Prop. \ref{prop ap2}(2), we also have $R_{M,n}(x) \rightarrow x$ for the $W^{[s, s]}$-topology for any $s \geq r$; this gives a slightly simpler argument than that in \cite[Cor. 8.11]{Col08} (in \emph{loc. cit.}, they need to work with $\wt{\mathbf{A}}_M[1/u]$, not just $\wt{\mathbf{A}}_M$). Indeed, since $W^{[r, s]}=\inf \{W^{[r, r]}, W^{[s, s]} \}$, so in Prop. \ref{prop ap2}(2), we also have $R_{M,n}(x) \rightarrow x$ for the $W^{[r, s]}$-topology.

\begin{prop}
Suppose $r>0$, and $x \in \wt{\mathbf{A}}_M^{[r,+\infty]}$, then

then $\mathbf{A}_{M, \infty}^{[r, +\infty]}:=\cup_{m \geq 0} \mathbf{A}_{M, m}^{[r, +\infty]}$ is dense in $\wt{\mathbf{A}}_M^{[r,+\infty]}$ for the weak topology.
\end{prop}
\begin{proof}
Since the statement is ``invariant" after application of Frobenius, it suffices to show that $\varphi^k(\mathbf{A}_{M, \infty}^{[r, +\infty]})$ is dense in $\varphi^k(\wt{\mathbf{A}}_M^{[r,+\infty]})$ for $k \gg 0$ such that $p^kr > r_M$; then we can apply Prop. \ref{prop ap2}. Note that by Lem. \ref{lem weak rr} (and Rem. \ref{rem weak rrss}), the density also holds in the $W^{[r, s]}$-topology for any $s \geq r$.
\end{proof}
\end{comment}

\subsection{Approximation of $b$}
We now build a sequence $\{b_n\}_{n \geq 1}$ to approximate $b$, which furthermore satisfies $\nabla_{\gamma}(b_n)=0$ for all $n$.
In the following, we use $K_{\infty} \subset_{\textnormal{fin}} M \subset L$ to mean that $M$ is a intermediate extension which is finite over $K_{\infty}$.

\begin{lemma} \label{nabla=0}
Let $W$ be a $\Qp$-Banach representation of $\hat{G}$. Then
\[
(W^{\hat{G}\dla})^{\nabla_{\gamma}=0} = \bigcup_{K_{\infty} \subset_{\textnormal{fin}} M \subset L} W^{\tau\dla,\Gal(L/M)=1}.
\]
\end{lemma}
\begin{proof}
If $x \in W^{\hat{G}\dla}$ such that $\nabla_{\gamma}(x) =  0$, then there exists $m \geq 0$ such that $x \in W^{\hat{G}_m \dan}$ and $\exp(p^m\nabla_{\gamma})(x)$ converges in $W^{\hat{G}_m \dan}$.
Thus $x \in W^{\tau\dla,\Gal(L/M)=1}$ for some large  $M$.
\end{proof}

\begin{lemma}  \label{lem approx b}
Let $[r, s] \subset (0, +\infty)$ and let $n \geq 1$.
Let  $x \in \wt{\mathbf{A}}_L^+$. Then there exists $w \in (\wt{\mathbf{B}}_L^{[r,s]})^{\hat{G}\dla, \nabla_\gamma=0}$, such that $x-w \in p^n\wt{\mathbf A}^{[r, s]}_L$.
\end{lemma}
\begin{proof}
Fix some $k \gg 0$ such that $u^k \in p^n\wt{\mathbf A}^{[r, s]}_L$.

Let $\overline{x} \in \wt{\mathbf{E}}_L^+$ be the modulo $p$ reduction of $x$.
By \cite[Cor. 4.3.4]{Win83}, the set
$$\bigcup_{m \in \mathbb N}\varphi^{-m}\left(\bigcup_{K_\infty \subset_{\textnormal{fin}} M \subset L} \mathbf{E}_M^+ \right)$$
is dense in $\wt{\mathbf{E}}_L^+$ for the $\underline{\pi}$-adic topology, where $\mathbf{E}_M^+$ is the ring of integers of $X_K(M)$.
Thus, there exists some $\overline{y}_1 \in  \phi^{-m_1}(\mathbf{E}_{M_1}^+)$ for some $m_1$ and $M_1$, such that $\overline{x}-\overline{y}_1= u^k\overline{z}_1 $ where $\overline{z}_1 \in \wt{\mathbf{E}}_L^+$. Thus we can write
$$x-[\overline{y}_1] -u^k[\overline{z}_1]=px_1  \text{ for some } x_1 \in  \wt{\mathbf{A}}_L^+.$$
Now we can repeat the process for $x_1$ (in the process, we can choose $M_2$ to contain $M_1$), so we can write $x_1-[\overline{y}_2] -u^k[\overline{z}_2]=px_2$. Iterate the process, and let $y =[\overline{y}_1]+ p[\overline{y}_2]+\cdots +p^{n-1}[\overline{y}_n]$, then $y \in  \wt{\mathbf{A}}_{M_n}^+$ and
$$x-y \in p^n\wt{\mathbf{A}}^{+}_L +u^k\wt{\mathbf{A}}_L^+.$$

Pick any $r'$ such that $0<r'<r$.
By Prop. \ref{prop ap2}(2), we can choose some $N \gg 0$ (in particular, we require $p^Nr' >r_{M_n}$), such that if we let $w:= R_{M_n, N}(y)$, then we have
\begin{itemize}
\item $w \in \mathbf{A}_{M_n,N}^{[r',+\infty]} \subset \wt{\mathbf A}_L^{[r',+\infty]}   \subset \wt{\mathbf A}_L^{[r,+\infty]}$, and
\item $y-w =p^n a+ u^kb$ for some $a \in \wt{\mathbf{A}}, b \in \wt{\mathbf{A}}^+$ (note that we do not know if $a \in \wt{\mathbf{A}}_L$ or $b \in \wt{\mathbf{A}}^+_L$), and
\item $W^{[r, s]}(y-w) \geq n$.
\end{itemize}
We claim that $a \in \wt{\mathbf{A}}^{[r, s]}$. Since $p^n a=y-w -u^kb \in \wt{\mathbf{A}}^{[r, s]}$, it suffices to show that $W^{[r, s]}(a) \geq 0$. But we have
$$W^{[r, s]}(a) =W^{[r, s]}(y-w -u^kb) -n \geq \inf \{W^{[r, s]}(y-w), W^{[r, s]}(u^kb)\} -n \geq 0$$
where we use the assumption $u^k \in p^n\wt{\mathbf A}^{[r, s]}_L$ (so $W^{[r, s]}(u^k) \geq n$).

Now, we have
$$x-w \in p^n \wt{\mathbf{A}}^{[r, s]} + u^k\wt{\mathbf{A}}^+ \subset p^n\wt{\mathbf A}^{[r, s]},$$
and necessarily $x-w \in p^n\wt{\mathbf A}^{[r, s]}_L$ because $x-w$ is $G_L$-invariant. Finally, $w \in (\wt{\mathbf{B}}_L^{[r,s]})^{\hat{G}\dla, \nabla_\gamma=0}$ by Lem. \ref{nabla=0} (and Thm. \ref{thm la M}).
\end{proof}

\subsubsection{An approximating sequence for $b$} \label{subsub bn}
Let $I=[r, s] \subset (0, +\infty)$ such that $r\geq r(b)$. For any $n \geq 1$, let $b_n \in (\wt{\mathbf{B}}_L^{I})^{\hat{G}\dla, \nabla_\gamma=0}$ be as in Lem. \ref{lem approx b} such that $b-b_n \in p^n\wt{\mathbf A}^{I}_L$. For any fixed $n$, since both $b$ and $b_n$ are locally analytic, we can choose $m=m(n) \gg 0$ (which depends on $n$) such that $b-b_n \in (\wt{\mathbf{B}}^I_L)^{\hat{G}_m\dan}$ and $\|b-b_n\|_{\hat{G}_m} \leq p^{-n}$.

\subsubsection{A differential operator}
Let $I=[r, s] \subset (0, +\infty)$ such that $r\geq r(b)$.
Since $\gamma(b) = \chi(\gamma) \cdot b$, we have $\nabla_{\gamma}(b) =b $.  Since $1/b$ is in $ (\wt{\mathbf{B}}^I_L)^{\hat{G}\dla}$ by Lem \ref{lem b}, we can define
$\partial_{\gamma}: (\wt{\mathbf{B}}^I_L)^{\hat{G}\dla} \to (\wt{\mathbf{B}}^I_L)^{\hat{G}\dla}$
via
$$\partial_{\gamma}:=\frac{1}{b}\nabla_{\gamma}.$$
So in particular, we have
\[
\partial_{\gamma}(b-b_n)^k = k(b-b_n)^{k-1}, \forall k \geq 1.
\]

%Note that if  $ x_i\in (\wt{\mathbf{B}}^I_L)^{\hat{G}_m\dan,\nabla_{\gamma}=0}$ such that $\|p^{ni}x_i\|_{\hat{G}_m} \rightarrow 0$ as $i \rightarrow + \infty$, then the series $\sum_{i \geq 0}x_i(b-b_n)^i$ converges in $(\wt{\mathbf{B}}^I_L)^{\hat{G}_m\dan}$.

\begin{theorem} \label{thm Ghat la}
Let $I=[r, s] \subset (0, +\infty)$ such that $r\geq r(b)$.
Suppose $x \in (\wt{\mathbf{B}}_L^I)^{\hat{G}\dla}$, then there exists $n,m\geq 1$ and a sequence $\{x_i\}_{i \geq 0}$ in $(\wt{\mathbf{B}}_L^I)^{\hat{G}_m\dan, \nabla_{\gamma}=0}$ such that $\|p^{ni}x_i\|_{\hat{G}_m} \rightarrow 0$ and $x= \sum_{i \geq 0}x_i(b-b_n)^i$ (which converges in the norm $\|\cdot \|_{\hat{G}_m}$).
\end{theorem}
\begin{proof}
The proof is similar as \cite[Thm. 5.4]{Ber16}.
Suppose $m \geq 1$ such that $x \in (\wt{\mathbf{B}}_L^I)^{\hat{G}_m\dan}$.
Apply \cite[Lem. 2.6]{BC16} to the map $\partial_{\gamma} : (\wt{\mathbf{B}}_L^I)^{\hat{G}_m\dan} \to (\wt{\mathbf{B}}_L^I)^{\hat{G}_m\dan}$, so there exists $n \geq 1$ such that for all $k \in \mathbb{Z}^{\geq 0}$, we have $\|\partial_{\gamma}^k(x)\|_{\hat{G}_m} \leq p^{(n-1)k}\|x\|_{\hat{G}_m}$. Increase $m$ if necessary so that $m \geq m(n)$ as in \S \ref{subsub bn}.
Let
$$x_i := \frac{1}{i!}\sum\limits_{k \geq 0}(-1)^{k}\frac{(b-b_n)^{k}}{k!}\partial_{\gamma}^{k + i}(x),$$
then similarly as \cite[Thm. 5.4]{Ber16}, they satisfy the desired property.
\end{proof}

\section{Overconvergence of $(\varphi, \tau)$-modules} \label{sec oc}
In this section, for a $p$-adic Galois representation $V$ of $G_K$ of dimension $d$, we show that its  associated $(\varphi, \tau)$-module is overconvergent. We will construct $\wt{D}_L^I(V): = (\wt{\mathbf B}^I \otimes_{\Qp}V)^{G_L}$ (see \S \ref{subsec oc}), which is a finite free module over $\wt{\mathbf B}_L^I$ of rank $d$ equipped with a $\hat{G}$-action. The key point is to show that $(\wt{D}_L^I(V))^{\tau\dla, \gamma=1}$ is also finite free over $(\wt{\mathbf B}_L^I)^{\tau\dla, \gamma=1}$ of rank $d$, i.e., $\wt{D}_L^I(V)$ has ``enough" $(\tau\dla, \gamma=1)$-vectors; these vectors will further descend to ``overconvergent vectors" in the $(\varphi, \tau)$-module, via Kedlaya's slope filtration theorem. Using the classical overconvergent $(\varphi, \Gamma)$-module, we already know that $(\wt{D}_L^I(V))^{\hat{G}\dla}$ is finite free over $(\wt{\mathbf B}_L^I)^{\hat{G}\dla}$ of rank $d$. So we need to take $(\gamma=1)$-invariants in $(\wt{D}_L^I(V))^{\hat{G}\dla}$, and show it keeps the correct rank; this is achieved by a Tate-Sen descent \emph{or} a monodromy descent (followed by an \'etale descent).

In \S \ref{subsec TS}, we will carry out the descent of locally analytic vectors: the Tate-Sen descent and \'etale descent uses an axiomatic approach taken from \cite{BC08}; the monodromy descent (in Rem. \ref{rem mono des}) follows some similar argument as in \cite{Ber16}.
 In \S \ref{subsec oc}, we prove the overconvergence result.

In this section, whenever we write $I=[r, s] \subset (0, +\infty)$, we mean $[r, s]=[r_\ell, r_k]$, cf. Convention \ref{convellk}.

\subsection{Descent of locally analytic vectors} \label{subsec TS}
Since we will use results from \cite{BC08}, it will be convenient to use valuation notations.
\begin{notation}\label{nota valnorm}
Let $W$ be a $\Qp$- (or $\Zp$-) Banach representation (cf. Notation \ref{nota Zp banach}) of a $p$-adic Lie group $G$. Suppose there is an analytic bijection $\mathbf{c}: G \to \Zp^d$ (as in \S \ref{subsub def an}), and suppose $W^{G\dan}=W$. Let $\val_{G}$ denote the valuation on $W$ associated to the norm $\| \cdot \|_G$ (cf. \S \ref{subsub val}).
\end{notation}

%subsubsection{} \label{subsub prof}
%Let $H_0$ be a profinite group. Let $\wt{\Lambda}$ be an $\Zp$-algebra equipped with a sub-multiplicative valuation $\val_\Lambda: \wt{\Lambda} \to \mathbb{R}\cup\{+\infty\}$ such that $\val_\Lambda(p)>0$, $\val_\Lambda(px)=\val_\Lambda(p)+\val_\Lambda(x)$ when $x \in \wt{\Lambda}$, and such that $\wt{\Lambda}$ is complete with respect to $\val_\Lambda$.
% Suppose $H_0$ acts on $\wt{\Lambda}$  such that $\val_\Lambda(gx)= \val_\Lambda(x), \forall g \in H_0, x \in \wt{\Lambda}$.

\begin{prop}\label{prop TS}
Let $(\wt{\Lambda}, \| \cdot \|)$ be a $\Zp$-Banach algebra (cf. Notation \ref{nota Zp banach}), and let $\val_\Lambda$ be the valuation associated to  $\| \cdot \|$. (Here the notation $\val_\Lambda$ follows that of \cite[\S 3.1]{BC08}, although ``$\val_{\wt{\Lambda}}$" might be a more suggestive one).

 %Suppose $\val_\Lambda$ is sub-multiplicative, and suppose $\val_\Lambda(p)>0$ and such that $\val_\Lambda(px)=\Lambda(p)+\Lambda(x)$ for all $x\in \wt{\Lambda}$.

Let $H_0$ be a profinite group which acts on $\wt{\Lambda}$  such that $\val_\Lambda(gx)= \val_\Lambda(x), \forall g \in H_0, x \in \wt{\Lambda}$.
Let $g \mapsto U_g$ be a continuous cocycle of $H_0$ in $\GL_d(\wt{\Lambda})$.

Suppose $H \subset H_0$ is an open subgroup, and suppose there exists some $a>c_1>0$ such that the following conditions are satisfied:
\begin{itemize}
\item (TS1): there exists $\alpha \in \wt{\Lambda}^{H}$ such that $\val_\Lambda(\alpha)>-c_1$ and $\sum_{\sigma \in H_0/H} \sigma(\alpha)=1$.
\item  $\val_\Lambda(U_g -1) \geq a, \forall g \in H$.
\end{itemize}
Then there exists $M \in \GL_d(\wt{\Lambda})$ such that $\val_\Lambda(M -1)\geq a-c_1$ and the cocycle   $g \mapsto M^{-1}U_g g(M)$ is trivial when restricted to $H$.
\end{prop}
\begin{proof}
This is a slight variant of \cite[Cor. 3.2.2]{BC08}. Indeed, in \emph{loc. cit.}, it requires the condition (TS1) to be satisfied for any pair of open subgroups $H_1 \subset H_2$ in $H_0$ (cf. \cite[Def. 3.1.3]{BC08}); however, in the proof of \cite[Lem. 3.1.2, Cor. 3.2.2]{BC08}, this condition is used only for one pair.
\end{proof}
%\item (TS1): For $H_1 \subset H_2$ any two open subgroups of $H_0$, there exists $\alpha \in \wt{\Lambda}^{H_1}$ such that $\val_\Lambda(\alpha)>-c_1$ and $\sum_{\sigma \in H_2/H_1} \sigma(\alpha)=1$.

\begin{lemma} \label{lem TS verify}
Let $c_1>0$, let $I=[r, s] \subset (0, +\infty)$, and let $K_\infty \subset M \subset L$ where $[M:K_\infty]<+\infty$. Then there exists $n \gg 0$, and
$$\alpha \in (\wt{\mathbf{B}}^{I}_L)^{\tau_n\dan, \gal(L/M)=1},$$
such that the following holds:
\begin{itemize}
\item $\val_{\tau_n}(\alpha) = W^I(\alpha) >-c_1$, here $\val_{\tau_n} = \val_{<\tau_n>}$ (cf. Notation \ref{nota valnorm});
\item  $\sum_{\sigma \in \gal(M/K_\infty)} \sigma(\alpha)=1$.
\end{itemize}
\end{lemma}
\begin{proof}
Denote $\Tr:=\sum_{\sigma \in  \gal(M/K_\infty)} \sigma$ the trace operator.
By Thm. \ref{thm Win Gal}, $X_K(M)$ is a finite Galois extension of $X_K(K_\infty)$, and so there exists $\beta \in X_K(M)$ such that $\Tr(\beta)=1$. Note that we necessarily have $v_{\wt{\mathbf E}}(\beta) \leq 0$.

Suppose $m \gg0$ ($m$ depends on $M$ and $I$) such that $p^{-m}r_{M} <r$ (where $r_M >0$ as in Lem. \ref{lem um var}), and such that
    \begin{equation}\label{eq c1}
    \frac{p-1}{pr} \frac{1}{p^m} v_{\wt{\mathbf E}}(\beta) >- c_1, \quad \text{ and such that }
         \end{equation}
   \begin{equation}\label{eqadd}
          (1-\frac{r_{M}}{p^mr}) + \frac{p-1}{p^mpr}v_{\wt{\mathbf E}}(\beta)>0.
        \end{equation}

Let $\gamma=\varphi^{-m}(s(\beta))$ (where $s$ is the map in \S \ref{subsub sec s}), then
\begin{itemize}
\item  since $p^{-m}r_{M } <r$, $\gamma \in \varphi^{-m}(\mathbf{A}_{M}^{[r_{M }, +\infty]}[1/u_{M}]) \subset \wt{\mathbf{A}}^{[r, +\infty]}[1/u]$;
\item for any $a \in [r, s]$, by using similar argument as in \S \ref{subsub sec s}(2) and apply \eqref{eq c1}, we have
$$W^{[a, a]}(\gamma)=W^{[p^m a, p^m a]}(s(\beta)) = \frac{p-1}{p \cdot p^m a} v_{\wt{\mathbf E}}(\beta) >- c_1,$$
and so $W^I(\gamma) >-c_1.$
\end{itemize}
Since $\Tr(\varphi^{-m}(\beta))=1$, we have $\Tr(\gamma) =1+\sum_{k \geq 1}p^k[a_k]$. Furthermore, for any $k \geq 1$,
$$w_k(\Tr(\gamma)) \geq \inf_{\sigma \in  \gal(M/K_\infty)} \{w_k(\sigma(\gamma))\} =w_k(\gamma)
= p^{-m}w_k(s(\beta)) >p^{-m} \cdot(v_{\wt{\mathbf E}}(\beta) -kpr_{M }(p-1)^{-1} ) ,$$
where the final inequality uses \eqref{strictine}.
So when $k \geq 1$,
\begin{eqnarray*}
 k+ \frac{p-1}{pr}\cdot  w_k(\Tr(\gamma))  &>&  k+ \frac{p-1}{pr}\cdot p^{-m} \cdot (v_{\wt{\mathbf E}}(\beta) -kpr_{M }(p-1)^{-1} )\\
&=&   k(1-\frac{r_{M }}{p^mr}) + \frac{p-1}{pr}\cdot  \frac{1}{p^m }v_{\wt{\mathbf E}}(\beta)\\
&\geq &  (1-\frac{r_{M }}{p^mr}) + \frac{p-1}{pr}\cdot  \frac{1}{p^m }v_{\wt{\mathbf E}}(\beta), \quad \text{ since }  1-\frac{r_{M }}{p^mr}>0 \\
& >&0, \quad \text{ by } \eqref{eqadd}.
\end{eqnarray*}
By Lem. \ref{lem units}, $\Tr(\gamma) \in (\wt{\mathbf{A}}^{[r, +\infty]})^\times$, and so $\varphi^m(\Tr(\gamma)) \in (\wt{\mathbf{A}}^{[p^mr, +\infty]})^\times$.
Since  $\varphi^m(\gamma) \in \mathbf{A}_{M }^{[r_{M }, +\infty]}[1/u_{M }]$, we obtain
$$\varphi^m(\Tr(\gamma)) \in   \mathbf{A}_{K_\infty}^{[r_{M }, +\infty]} \subset   \mathbf{A}_{K_\infty}^{[p^mr, +\infty]}, \text{ since } p^{-m}r_{M } <r.$$
By Lem. \ref{lem unit am} (note that $p^m r >r_{M}$), $\varphi^m(\Tr(\gamma)) \in  (\mathbf{A}_{K_\infty}^{[p^mr, +\infty]})^\times$, and so
$\Tr(\gamma) \in (\varphi^{-m}(\mathbf{A}_{K_\infty}^{[p^mr, +\infty]}))^\times$, and so by Thm. \ref{thm loc ana gamma 1},
$$(\Tr(\gamma))^{-1} \in (\wt{\mathbf{B}}^{I}_L)^{\tau\dla, \gal(L/K_\infty)=1}.$$
Let $\alpha: =\gamma \cdot (\Tr(\gamma))^{-1}$.
Note that
$$ \gamma \in \varphi^{-m}(\mathbf{A}_{M}^{[r_{M }, +\infty]}[1/u_{M}]) \subset \varphi^{-m}(\mathbf{B}_{M}^{p^mI}) \subset (\wt{\mathbf{B}}^{I}_L)^{\tau\dla, \gal(L/M)=1}, \text{ by Thm.} \ref{thm la M}.
$$
Thus, we have $\alpha \in (\wt{\mathbf{B}}^{I}_L)^{\tau\dla, \gal(L/M)=1}$. We also note that $W^I(\alpha)=W^I(\gamma)>-c_1$. Finally, the existence of $n \gg 0$ such that $\alpha \in (\wt{\mathbf{B}}^{I}_L)^{\tau_n\dan, \gal(L/M)=1}$ is by definition; the existence of $n \gg 0$  such that $\val_{\tau_n}(\alpha)=W^I(\alpha)$ is by Lem. \ref{lem 2.4BC}.
\end{proof}

\subsubsection{} \label{ringB}
Let $B$ be a $\Qp$-Banach algebra, equipped with an action by a finite group ${G}$. Let $ {B}^{\natural}$ denote the ring $B$ with trivial $G$-action. Suppose that
\begin{enumerate}
\item $B$ is a  finite free $ {B}^G$-module;
\item there exists a $G$-equivariant decomposition $ {B}^\natural \otimes_{ {B}^G} B \simeq \oplus_{g \in G} {B}^\natural \cdot e_g$ such that $e_g^2=e_g$, $e_ge_h=0$ for $g\neq h$, and $g(e_h)=e_{gh}$.
\end{enumerate}

\begin{prop} \label{prop etale}
Let $B$ and $G$ be as in \S \ref{ringB}.
Suppose $N$ is a finite free $B$-module with semi-linear $G$-action, then
\begin{enumerate}
\item $N^G$ is a finite free $ B^G$-module;
\item the map $  B \otimes_{ {B}^G} N^G \to N$ is a $G$-equivariant isomorphism.
\end{enumerate}
\end{prop}
\begin{proof}
This is \cite[Prop. 2.2.1]{BC08}.
\end{proof}
%Note that in \emph{loc. cit.}, it asks $B$ to be a Banach algebra, but it is not needed if the $S$ in \emph{loc. cit.} is just $\Qp$.

\begin{comment}

\begin{lemma}\label{lem verify et}
Let $I=[r, s] \subset (0, +\infty)$, and let $K_\infty \subset M \subset L$ where $[M: K_\infty]<+\infty$. Then the pair
$$B: =(\wt{\mathbf{B}}^I_L)^{\tau\dla, \gal(L/M)=1}, G:=\gal(M/K_\infty),$$
satisfies the two conditions in \S \ref{ringB}.
\end{lemma}
\begin{proof}
Via the same argument as in \cite[Lem. 4.2.5]{BC08}, there exists some $s(M)>0$ such that if $a>s(M)$, then the pair $ B_1:= \mathbf{B}_M^{[a, +\infty]}$ and $G$ satisfies the two conditions in \S \ref{ringB}.
Thus, when $m \gg 0$ such that $p^mr >s(M)$, then the pair $ B_2:= \mathbf{B}_M^{p^mI}$ and $G$ also satisfies the two conditions in \S \ref{ringB}. Thus we can conclude by using Thm. \ref{thm la M}.

Let $N:=(\mathcal M)^{\gal(L/M)}$, and $G: =\gal(M/K_\infty)$, then we can use Prop. \ref{prop etale} (via Lem. \ref{lem verify et}) to conclude the proof.
\end{proof}

\end{comment}

\begin{prop} \label{prop M triv}
Let $I=[r, s] \subset (0, +\infty)$.
Let $\mathcal M$ be a finite free $(\wt{\mathbf{B}}^{I}_L)^{\hat{G}\dla}$-module of rank $d$, with a semi-linear and locally analytic $\hat{G}$-action. Then $(\mathcal M)^{\gal(L/K_\infty)}$ is finite free over $(\wt{\mathbf{B}}^I_L)^{\tau\dla, \gamma=1}$ of rank $d$, and
$$ (\wt{\mathbf{B}}^I_L)^{\hat{G}\dla}\otimes_{(\wt{\mathbf{B}}^I_L)^{\tau\dla, \gamma=1}}(\mathcal M)^{\gal(L/K_\infty)} \simeq \mathcal M.  $$
\end{prop}
\begin{proof} The following proof is via Tate-Sen descent; see Rem. \ref{rem mono des} for another proof via monodromy descent.

Since $\gal(L/K_\infty)$ is topologically generated by finitely many elements (in most cases, by one element; cf. Notation \ref{nota hatG}), there exists a basis $e_1, \cdots, e_d$ of $\mathcal M$ such that the co-cycle $c$ associated to the $\gal(L/K_\infty)$-action on $\mathcal M$ (with respect to this basis) is of the form  $g \mapsto U_g$ where
$U_g \in \GL_d((\wt{\mathbf{B}}^{I}_L)^{\hat{G}_n\dan})$ for some $n\gg 0$.

%The $\gal(L/K_\infty)$-action on $\mathcal M$ gives a continuous co-cycle $c$ of $\gal(L/K_\infty)$ in $\GL_d((\wt{\mathbf{B}}^{I}_L)^{\hat{G}\dla})$ where $g \mapsto U_g$. Since  we can assume that the co-cycle takes value in $\GL_d((\wt{\mathbf{B}}^{I}_L)^{\hat{G_n}\dan})$ for some $n>0$.

Let $a>c_1>0$.
Choose some $M$ such that $K_\infty \subset_{\textnormal{fin}} M \subset L$ and  such that
$$ \val_{\hat{G}_n}(U_g -1) \geq a, \text{ when } g \in \gal(L/M),$$
where $\val_{\hat{G}_n}$ is as in Notation \ref{nota valnorm}.
By Lem. \ref{lem TS verify}, there exists some $n' \gg 0$ and $\alpha \in (\wt{\mathbf{B}}^{I}_L)^{\tau_{n+n'}\dan, \gal(L/M)=1}$ such that
$\val_{\hat{G}_{n+n'}}(\alpha) >-c_1$, and $\sum_{\sigma \in \gal(M/K_\infty)} \sigma(\alpha)=1$. Apply Prop. \ref{prop TS} to the pair
$$(\wt{\Lambda}, \val_\Lambda)= ((\wt{\mathbf{B}}^{I}_L)^{\hat{G}_{n+n'}\dan}, \val_{\hat{G}_{n+n'}}),$$
(where $\val_{\hat{G}_{n+n'}}$ is sub-multiplicative by Lem. \ref{ringlocan}),
the restricted co-cycle $c|_{\gal(L/M)}$, when considered as evaluated in $\GL_d((\wt{\mathbf{B}}^{I}_L)^{\hat{G}_{n+n'}\dan})$, is trivial after base change. So:
\begin{itemize}
  \item[(*)]: $(\mathcal M)^{\gal(L/M)}$ is finite free over $(\wt{\mathbf{B}}^I_L)^{\tau\dla, \gal(L/M)=1}$ of rank $d$.
\end{itemize}

Let $G: =\gal(M/K_\infty)$.
Fix a basis $e_1', \cdots, e_d'$ of $(\mathcal M)^{\gal(L/M)}$, and suppose the co-cycle associated to the $G$-action on $(\mathcal M)^{\gal(L/M)}$ with respect to this basis has value in $\GL_d(\varphi^{-m}(\mathbf{B}_M^{p^mI}))$ for some $m \gg 0$ (using Thm. \ref{thm la M}). Let $N_m$ be the $\varphi^{-m}(\mathbf{B}_M^{p^mI})$-span of  $e_1', \cdots, e_d'$.

Via the same argument as in \cite[Lem. 4.2.5]{BC08}, there exists some $s(M)>0$ such that if $a>s(M)$, then the pair $ (\mathbf{B}_M^{[a, +\infty]}, G)$ satisfies the two conditions in \S \ref{ringB}.
So when $m\gg 0$ such that $p^mr >s(M)$, then the pair $(\mathbf{B}_M^{p^mI}, G)$, and thus also the pair $(\varphi^{-m}(\mathbf{B}_M^{p^mI}), G)$ satisfy the two conditions in \S \ref{ringB}. By Prop. \ref{prop etale}, $(N_m)^G$ is finite free over $\varphi^{-m}(\mathbf{B}_{K_\infty}^{p^mI})$ of rank $d$; this implies the desired result.
\end{proof}

\begin{remark} \label{rem mono des}
Keep the notations in Prop. \ref{prop M triv} above. Suppose \emph{furthermore} that $r \geq r(b)$ (see \S \ref{sec Ghat la} for $r(b)$), then we can give another proof of Prop. \ref{prop M triv} via monodromy descent. The proof follows similar ideas as in \cite[\S 6]{Ber16}.

In this second proof, we only reprove the statement (*) above, namely, we show that there exists some $K_\infty \subset M\subset L$ such that
$(\mathcal M)^{\gal(L/M)}$ is finite free over $(\wt{\mathbf{B}}^I_L)^{\tau\dla, \gal(L/M)=1}$ of rank $d$.
By Lem. \ref{nabla=0}, it suffices to show that $(\mathcal M)^{\nabla_{\gamma}=0}$ is finite free over $(\wt{\mathbf{B}}^I_L)^{\hat{G}\dla, \nabla_{\gamma}=0}$ of rank $d$, and
$$ (\wt{\mathbf{B}}^I_L)^{\hat{G}\dla}\otimes_{(\wt{\mathbf{B}}^I_L)^{\hat{G}\dla, \nabla_{\gamma}=0}}(\mathcal M)^{\nabla_{\gamma}=0} \simeq \mathcal M.  $$
Let $D_{\gamma}=\Mat(\partial_{\gamma})$ ($\partial_{\gamma}$ is well-defined because $r\geq r(b)$), then it suffices to show that there exists $H \in \GL_d((\wt{\mathbf{B}}_L^I)^{\la})$ such that $\partial_{\gamma}(H)+D_{\gamma}H = 0$.
For $k \in \mathbb N$, let $D_k = \Mat(\partial_{\gamma}^k)$. For $n$ large enough, the series given by
$$H = \sum_{k \geq 0}(-1)^kD_k\frac{(b-b_n)^k}{k!}$$
converges in $M_d((\wt{\mathbf{B}}_L^I)^{\la})$ to a solution of the equation $\partial_{\gamma}(H)+D_{\gamma}H = 0$. Moreover, for $n$ big enough, we have $W^I(D_k\cdot (b-b_n)^k/k!)>0$ for $k \geq 1$, so that $H \in \GL_d((\wt{\mathbf{B}}_L^I)^{\la})$.
\end{remark}

\begin{remark}
The condition $r \geq r(b)$ in the proof of Rem. \ref{rem mono des} is actually harmless for application in  our main theorem Thm. \ref{thm final} (i.e., in the proof of Thm. \ref{thm final}, we could equally apply Rem. \ref{rem mono des} instead of Prop. \ref{prop M triv}). Indeed, at the very beginning of the proof of Thm. \ref{thm final}, we could assume the ``$\tilde r_0$" there to be bigger than $r(b)$.
\end{remark}

\subsection{Overconvergence of $(\varphi, \tau)$-modules}\label{subsec oc}

\begin{defn}\hfill
\begin{enumerate}
\item Let $\text{Mod}_{\mathbf{A}_{K_\infty}}^\varphi$ denote the category of finite free $\mathbf{A}_{K_\infty}$-modules $M $  equipped with a $\varphi_{\mathbf{A}_{K_\infty}}$-semi-linear endomorphism $\varphi _M : M\to M$ such that $1 \otimes \varphi : \varphi ^*M \to M $ is an isomorphism. Morphisms in this category  are just $\mathbf{A}_{K_\infty}$-linear maps compatible with $\varphi$'s.

\item  Let $\textnormal{Mod}_{\mathbf{B}_{K_\infty}}^\varphi$ denote the category of finite free $\mathbf{B}_{K_\infty}$-modules $D$  equipped with a $\varphi_{\mathbf{B}_{K_\infty}}$-semi-linear endomorphism $\varphi _D : D\to D$ such that there exists a finite free $\mathbf{A}_{K_\infty}$-lattice $M$ such that $M[1/p]=D$, $\varphi_D(M)\subset M$, and $(M, \varphi_D|_M) \in  \textnormal{Mod}_{\mathbf{A}_{K_\infty}}^\varphi$.
\end{enumerate}
 We call objects in $\textnormal{Mod}_{\mathbf{A}_{K_\infty}}^\varphi$ and $\textnormal{Mod}_{\mathbf{B}_{K_\infty}}^\varphi$ finite free
 {\em \'etale $\varphi$-modules}.
\end{defn}

\begin{defn} \label{defn phi tau mod} \hfill
\begin{enumerate}
\item Let $\textnormal{Mod}_{\mathbf{A}_{K_\infty}, \wt{\mathbf{A}}_L}^{\varphi, \hat{G}}$ denote the category consisting of triples $(M, \varphi_M, \hat G)$ where
\begin{itemize}
\item $(M , \varphi_M) \in \text{Mod}_{\mathbf{A}_{K_\infty}}^\varphi$;
\item $\hat G$ is a continuous $\wt{\mathbf{A}}_L$-semi-linear $\hat G$-action on $\hat M : =\wt{\mathbf{A}}_L \otimes_{\mathbf{A}_{K_\infty}} M$, and $\hat G$ commutes with $\varphi_{\hat M}$ on $\hat M$;
\item regarding $M$ as an $\mathbf{A}_{K_\infty} $-submodule in $ \hat M $, then $M
\subset \hat M ^{\gal(L/K_\infty)}$.
\end{itemize}

\item Let $\textnormal{Mod}_{\mathbf{B}_{K_\infty}, \wt{\mathbf{B}}_L}^{\varphi, \hat{G}}$ denote the category consisting of triples $(D, \varphi_D, \hat G)$ which contains a lattice (in the obvious fashion) $(M, \varphi_M, \hat G) \in \textnormal{Mod}_{\mathbf{A}_{K_\infty}, \wt{\mathbf{A}}_L}^{\varphi, \hat{G}}$.
\end{enumerate}
\end{defn}
The category $\textnormal{Mod}_{\mathbf{A}_{K_\infty}, \wt{\mathbf{A}}_L}^{\varphi, \hat{G}}$ (and $\textnormal{Mod}_{\mathbf{B}_{K_\infty}, \wt{\mathbf{B}}_L}^{\varphi, \hat{G}}$) are precisely the \'etale $(\varphi, \tau)$-modules as in \cite[Def. 2.1.5]{GL}.

\subsubsection{} Let $\textnormal{Rep}_{\Qp}(G_{\infty}) $ (resp. $\textnormal{Rep}_{\Qp}(G_{K}) $ ) denote the category of finite dimensional $\Qp$-vector spaces $V$   with  continuous $\Qp$-linear $G_{\infty}$ (resp. $G_K$)-actions.
\begin{itemize}[leftmargin=0cm]
\item For $D \in \textnormal{Mod}_{\mathbf{B}_{K_\infty}}^\varphi$, let
$$ V(D):= ( \wt{\mathbf{B}} \otimes_{\mathbf{B}_{K_\infty}} D) ^{\varphi =1},$$
then $V(D) \in \textnormal{Rep}_{\Qp}(G_{\infty})$.
If furthermore $(D, \varphi_D, \hat G) \in \textnormal{Mod}_{\mathbf{B}_{K_\infty}, \wt{\mathbf{B}}_L}^{\varphi, \hat{G}}$, then $ V(D) \in \textnormal{Rep}_{\Qp}(G_{K})$.

\item For $V \in \textnormal{Rep}_{\Qp}(G_{\infty}) $, let
$$     D_{K_\infty}(V):= (\mathbf B \otimes_{\Qp} V) ^{G_\infty}, $$
then $ D_{K_\infty}(V) \in \textnormal{Mod}_{\mathbf{B}_{K_\infty}}^\varphi$.
If furthermore $V \in \textnormal{Rep}_{\Qp}(G_K) $, let
$$     \wt{D}_L(V):= ( \wt{\mathbf{B}} \otimes_{\Qp} V) ^{G_L}, $$
then $\wt{D}_L(V) = \wt{\mathbf{B}}_L \otimes_{\mathbf{B}_{K_\infty}}  D_{K_\infty}(V)$ has a $\hat{G}$-action, making $(D_{K_\infty}(V), \varphi, \hat G)$ an \'etale $(\varphi, \tau)$-module.
\end{itemize}

\begin{thm}  \label{thm etale phi tau}
\hfill
\begin{enumerate}
\item The functors $V$ and $D_{K_\infty}$ induce  an exact tensor equivalence between the categories $\textnormal{Mod}_{\mathbf{B}_{K_\infty}}^\varphi$ and $\textnormal{Rep}_{\Qp}(G_{\infty}) $.
\item The functors $V$ and $(D_{K_\infty}, \wt{D}_L)$ induce  an exact tensor equivalence between the categories $\textnormal{Mod}_{\mathbf{B}_{K_\infty}, \wt{\mathbf{B}}_L}^{\varphi, \hat{G}}$ and $\textnormal{Rep}_{\Qp}(G_K) $.
\end{enumerate}
\end{thm}
\begin{proof}
(1) is \cite[Prop. A 1.2.6]{Fon90} (and using \cite[Lem. 2.1.4]{GL}). (2) is due to \cite{Car13} (cf. also \cite[Prop. 2.1.7]{GL}).
\end{proof}

Let  $V\in \textnormal{Rep}_{\Qp}(G_K) $.
Given $I \subset [0, +\infty]$ any interval, let
\begin{eqnarray*}
D_{K_\infty}^I(V) &:=& ( \mathbf{B}^I \otimes_{\Qp} V) ^{G_\infty}, \\
 \wt{D}_L^I(V)&:=&( \wt{\mathbf{B}}^I \otimes_{\Qp} V) ^{G_L}.
\end{eqnarray*}

\begin{defn} \label{def oc}
Let $V\in \textnormal{Rep}_{\Qp}(G_K) $, and let $\hat{D}= (D_{K_\infty}(V), \varphi, \hat G)$ be the \'etale $(\varphi, \tau)$-module associated to it.
Say that $\hat{D}$ is \emph{overconvergent} if there exists $r>0$, such that for $I'=[r, +\infty]$,
\begin{enumerate}
\item $D_{K_\infty}^{I'}(V)$ is finite free over  $\mathbf{B}^{I'}_{K_\infty}$, and
 $ \mathbf{B}_{K_\infty} \otimes_{\mathbf{B}^{I'}_{K_\infty}} D_{K_\infty}^{I'}(V) \simeq D_{K_\infty}(V);$
\item $\wt{D}_L^{I'}(V)$ is finite free over  $\wt{\mathbf{B}}^{I'}_L$ and
 $$   \wt{\mathbf{B}}_L \otimes_{\wt{\mathbf{B}}^{I'}_L}  \wt{D}_L^{I'}(V) \simeq \wt{D}_L(V).$$
\end{enumerate}
\end{defn}

\begin{comment}
 \begin{remark}

$D_{K_\infty}^{I'}(V)$ is finite free over  $\mathbf{B}^{I'}_{K_\infty}$, and
 $$ \mathbf{B}_{K_\infty} \otimes_{\mathbf{B}^{I'}_{K_\infty}} D_{K_\infty}^{I'}(V) \simeq D_{K_\infty}(V).$$
The ``overconvergence of the $\hat{G}$-action" follows from ``overconvergence of the $\varphi$-action". Indeed, note that $\wt{\mathbf{B}}^{I'}_L \otimes  D_{K_\infty}^{I'}(V) \subset \wt{D}_L^I(V)$ is stable....
\begin{enumerate}
\item $D_{K_\infty}^{I'}(V)$ is finite free over  $\mathbf{B}^{I'}_{K_\infty}$, and
 $ \mathbf{B}_{K_\infty} \otimes_{\mathbf{B}^{I'}_{K_\infty}} D_{K_\infty}^{I'}(V) \simeq D_{K_\infty}(V);$
\item $    \wt{\mathbf{B}}_L \otimes_{\mathbf{B}^{I'}_{K_\infty}} D_{K_\infty}^{I'}(V) \simeq \wt{D}_L(V).$
\end{enumerate}
\end{remark}
\end{comment}

\begin{theorem}\label{thm final}
For any $V\in \textnormal{Rep}_{\Qp}(G_K) $, its associated  \'etale $(\varphi, \tau)$-module is overconvergent.
\end{theorem}
\begin{proof}
\textbf{Step 1:} \emph{locally analytic vectors in $\wt{D}^I_L(V)$.}
For $I=[r, s] \subset (0, +\infty)$, let
$$D^{I}_{K_{p^\infty}}(V): =(\mathbb{B}^{I} \otimes_{\Qp}V)^{G_{p^\infty}},$$
where (as we mentioned in Rem. \ref{rem notation}) $\mathbb{B}$ and $\mathbb{B}^{I}$ are the rings denoted as ``$\mathbf{B}$" and ``$\mathbf{B}^{I}$" in \cite{Ber08ANT}. We still have $\mathbb{B} \subset \wt{\mathbf{B}}$ and $\mathbb{B}^I \subset \wt{\mathbf{B}}^I$.
By the main result of \cite{CC98}, there exists some $\tilde r_0>0$, such that when $r\geq \tilde r_0$, then $D^{I}_{K_{p^\infty}}(V)$ is finite free over $\mathbb{B}^{I}_{K_{p^\infty}}$ of rank $d$ (here $\mathbb{B}^{I}_{K_{p^\infty}}$ is precisely ``$\mathbf{B}^{I}_{K}$" in \cite{Ber08ANT}). Furthermore, there exists $G_K$-equivariant and $\varphi$-equivariant isomorphism
\begin{equation} \label{wtbi}
\wt{\mathbf{B}}^{I}  \otimes_{\Qp} V \simeq \wt{\mathbf{B}}^{I} \otimes_{\mathbb{B}^{I}_{K_{p^\infty}}}  D^{I}_{K_{p^\infty}}(V).
\end{equation}
Also, by \cite[\S 5.1]{Ber02},
\begin{equation}\label{phigammala}
D^{I}_{K_{p^\infty}}(V) \subset (\wt{D}^{I}_{L}(V))^{\tau=1, \gamma\dla} \subset (\wt{D}^{I}_{L}(V))^{\hat{G}\dla}.
\end{equation}
By Prop. \ref{prop la basis}, \eqref{phigammala} implies
\begin{equation} \label{phitaupa}
\wt{D}^{I}_{L}(V)^{\hat{G}\dla} =  (\wt{\mathbf{B}}_{L}^{I})^{\hat{G}\dla} \otimes_{\mathbb{B}^{I}_{ K_{p^\infty}}}  D^{I}_{ K_{p^\infty}}(V) .
\end{equation}
So in particular $\wt{D}^{I}_{L}(V)^{\hat{G}\dla} $ is finite free over $(\wt{\mathbf{B}}_{L}^{I})^{\hat{G}\dla}$. By Prop. \ref{prop M triv},   $\wt{D}^{I}_{L}(V)^{\tau\dla, \gamma=1} $ is finite free over $(\wt{\mathbf{B}}_{L}^{I})^{\tau\dla, \gamma=1}$. By \eqref{wtbi} and \eqref{phitaupa}, we also have
\begin{equation} \label{wtbmi}
\wt{\mathbf{B}}^{I}  \otimes_{(\wt{\mathbf B}^{I}_{L})^{\tau\dla, \gamma=1}} \wt{D}^{I}_{L}(V)^{\tau\dla, \gamma=1} \simeq \wt{\mathbf{B}}^{I}  \otimes_{\Qp} V
\end{equation}

\textbf{Step 2:} \emph{glueing $\wt{D}^{I}_{L}(V)^{\tau\dla, \gamma=1}$ as a vector bundle.}
For each $X \subset [\tilde r_0, +\infty)$ a closed interval, denote $M^X: =\wt{D}^{X}_{L}(V)^{\tau\dla, \gamma=1}$, and  $R^X: =(\wt{\mathbf B}^{X}_{L})^{\tau\dla, \gamma=1}$, and so Step 1 says that $M^X$ is finite free over $R^X$.
Let $I=[r, s] \subset [\tilde r_0, +\infty)$ such that $I\cap pI$ is non-empty.
For each $k\geq 1$, $\varphi^k$ induces a bijection between $\wt{D}^{I}_{L}(V)$ and $\wt{D}^{p^kI}_{L}(V)$, and thus also a bijection between $M^I$ and $M^{p^kI}$.
Let $m_1, \cdots, m_d$ be a basis of $M^I$, and so $\varphi(m_1), \cdots, \varphi(m_d)$ is a basis of $M^{pI}$.
Let $J: =I\cap pI$, then by using Prop. \ref{prop la basis}, we have
$$M^J = R^J \otimes_{R^I}M^I, \quad M^J = R^J \otimes_{R^{pI}}M^{pI}.$$
So if we write $(\varphi(m_1), \cdots, \varphi(m_d)) =(m_1, \cdots, m_d)P$, then $P \in \GL_d(R^J)$, and so $P \in \GL_d(\mathbf{B}^J_{K_\infty, m})$ for some $m \gg 0$.

Let $I_k: =p^kI, J_k:=I_k \cap I_{k+1}=p^k J$. For each $k \geq 1$, let $E_k$ be the $\mathbf{B}^{I_k}_{K_\infty, m}$-span of $\varphi^k(m_i)$. Since $\varphi^k(P) \in  \GL_d(\mathbf{B}^{J_k}_{K_\infty, m})$, we have
$$  \mathbf{B}^{J_k}_{K_\infty, m} \otimes_{\mathbf{B}^{I_k}_{K_\infty, m}} E_k \simeq \mathbf{B}^{J_k}_{K_\infty, m} \otimes_{\mathbf{B}^{I_{k+1}}_{K_\infty, m}} E_{k+1}.$$
This says that the collection $\{\varphi^m(E_k)\}_{k \geq 1}$ forms a vector bundle over $\mathbf{B}^{[p^mr, +\infty)}_{K_\infty}$ (cf. \cite[Def. 2.8.1]{Ked05}), and so by \cite[Thm. 2.8.4]{Ked05}, there exists $n_1, \cdots, n_d \in \cap_{k \geq 1}\varphi^m(E_k)$, such that if we let
$$D_{K_\infty}^{[p^mr, +\infty)}: =\oplus_{i=1}^d \mathbf{B}_{K_\infty}^{[p^mr , +\infty)}\cdot n_i,$$
then $$  \mathbf{B}_{K_\infty}^{p^mI_k} \otimes_{\mathbf{B}_{K_\infty}^{[p^mr , +\infty)}} D_{K_\infty}^{[p^mr , +\infty)} \simeq \varphi^m(E_k).$$
Now, define
$$D^{\dagger}_{\rig, K_\infty}:=\mathbf{B}_{\rig, K_\infty}^{\dagger}\otimes_{\mathbf{B}_{K_\infty}^{[p^mr , +\infty)}}  D_{K_\infty}^{[p^mr , +\infty)}$$
Then by \eqref{wtbmi}, we have
\begin{equation} \label{Vwtb}
\wt{\mathbf{B}}_{  \rig}^{\dagger}  \otimes_{\mathbf B^{\dagger}_{\rig, K_\infty} } D^{\dagger}_{\rig, K_\infty} = \wt{\mathbf{B}}_{  \rig}^{\dagger} \otimes_{\Qp}V.
\end{equation}
Eqn. \eqref{Vwtb} implies that $D^{\dagger}_{\rig, K_\infty}$ is pure of slope $0$ (cf. \cite{Ked05}). By \cite[Thm. 6.3.3]{Ked05}, there exists an \'etale $\varphi$-module $D^{\dagger}_{ K_\infty}$ over $\mathbf B^{\dagger}_{  K_\infty}$ such that $$\mathbf{B}^{\dagger}_{\rig, K_\infty}\otimes_{\mathbf B^{\dagger}_{  K_\infty}}D^{\dagger}_{ K_\infty} =D^{\dagger}_{\rig, K_\infty}.$$

%Here, \'etale means that $\varphi$ induces an isomorphism
%$$1\otimes \varphi: \mathbf B^{\dagger}_{  K_\infty}\otimes_{\varphi, \mathbf B^{\dagger}_{  K_\infty}}   D^{\dagger}_{ K_\infty} \simeq D^{\dagger}_{ K_\infty}$$

\textbf{Step 3:} \emph{overconvergence.}
We claim that
\begin{equation} \label{etale}
\mathbf{B}_{K_\infty} \otimes_{\mathbf B^{\dagger}_{  K_\infty} } D^{\dagger}_{  K_\infty} \simeq D_{K_\infty}(V).
\end{equation}
Let $D':=\mathbf{B}_{K_\infty} \otimes_{\mathbf B^{\dagger}_{  K_\infty} } D^{\dagger}_{  K_\infty}$.
By Thm. \ref{thm etale phi tau}(1), it suffices to show that
\begin{equation} \label{repeq}
V':= (\wt{\mathbf B}\otimes_{\mathbf{B}_{K_\infty}} D')^{\varphi=1} \simeq V|_{G_\infty}.
\end{equation}
 Note that $V'$ is always a $G_\infty$-representation over $\Qp$ of dimension $d$.
We have
\begin{eqnarray*}
V'&= &(\wt{\mathbf B} \otimes_{\mathbf{B}^{\dagger}_{  K_\infty}} D^{\dagger}_{  K_\infty})^{\varphi=1}\\
&= &(\wt{\mathbf{B}}^{\dagger}  \otimes_{\mathbf{B}^{\dagger}_{  K_\infty}} D^{\dagger}_{  K_\infty})^{\varphi=1}, \quad \text{ by \cite[Thm. 8.5.3(d)(e)]{KL15} },\\
 &\subset & (\wt{\mathbf{B}}^{\dagger}_{\rig}  \otimes_{\mathbf{B}^{\dagger}_{\rig,  K_\infty}} D^{\dagger}_{\rig,  K_\infty})^{\varphi=1}   \\
&= &  (\wt{\mathbf{B}}^{\dagger}_{\rig}  \otimes_{\Qp} V  )^{\varphi=1}, \quad \text{ by } \eqref{Vwtb}, \\
&=& V.
\end{eqnarray*}
So \eqref{repeq} holds for dimension reasons, and so \eqref{etale} holds, concluding the  overconvergence of $\varphi$-action (i.e., Def. \ref{def oc}(1) is verified).

Finally, note that $\wt{\mathbf{B}}^{\dagger}  \otimes_{\mathbf{B}^{\dagger}_{  K_\infty}} D^{\dagger}_{  K_\infty} \simeq \wt{\mathbf{B}}^{\dagger}  \otimes_{\Qp} V$, so if we let
$$\wt{D}^\dagger_L(V): = (\wt{\mathbf{B}}^{\dagger}  \otimes_{\Qp} V)^{G_L},$$
then $\wt{D}^\dagger_L(V) \simeq \wt{\mathbf{B}}^{\dagger}_L \otimes_{\mathbf{B}^{\dagger}_{  K_\infty}} D^{\dagger}_{  K_\infty}$. This implies the overconvergence of the $\tau$-action (i.e., Def. \ref{def oc}(2) is verified).
\end{proof}
%$\mathbf{B}^{\dagger}  \otimes_{\mathbf{B}^{\dagger}_{  K_\infty}} D^{\dagger}_{  K_\infty} \simeq \mathbf{B}^{\dagger}  \otimes_{\Qp} V$

\bibliographystyle{alpha}
%\bibliography{20180908Bib}

\end{document}